\documentclass[11pt]{article}
\usepackage{epsfig,psfrag,mathrsfs,amsmath,amsfonts,amssymb,amsthm}
\usepackage{color}
\usepackage{array}
\usepackage{lscape}
\usepackage{hyperref}
\usepackage{algorithm}
\usepackage{rotating}
\usepackage{algpseudocode}

\usepackage{cite}
\usepackage{amsmath,amsfonts,amssymb}
\usepackage{multirow}

\usepackage{xcolor}
\definecolor{inkblue}{RGB}{75,0,130}

\newtheorem{remark}{Remark}[section]

\newtheorem{lemma}{Lemma}[section]
\newtheorem{example}{Example}[section]

\def\D{\mathrm{d}}

\newcommand{\dthe}{\,{\rm d} {\theta}}

\newcommand{\TS}{\Delta t}
\newcommand{\rb}{\right)}
\newcommand{\lb}{\left(}

\newcommand{\ignore}[1]{}

\usepackage{siunitx}
\usepackage{booktabs}

\usepackage{seqsplit}
\usepackage{picinpar}

\title{{New Fully Discrete Active Flux Methods with Truly Multi-Dimensional Evolution Operators and WENO Reconstruction}}

\author{Amelie Porfetye\thanks{Heinrich-Heine-University D\"usseldorf, Germany (amelie.porfetye@hhu.de)},~ 
Zhuyan Tang\thanks{Corresponding author: Johannes Gutenberg University Mainz, Germany (zhtang@uni-mainz.de)},~
Shaoshuai Chu\thanks{RWTH Aachen University, Germany (chu@igpm.rwth-aachen.de)},~ 
Christiane Helzel\thanks{Heinrich-Heine-University D\"usseldorf, Germany (christiane.helzel@hhu.de)},~ \\
M\'aria Luk\'a\v{c}ov\'a-Medvid'ov\'a\thanks{Johannes Gutenberg University Mainz, Germany (lukacova@uni-mainz.de)}
}

\date{}

\begin{document}
\allowdisplaybreaks
\maketitle
\footnotetext{The first two authors are co-first authors of this paper, arranged by alphabetical order.}

\begin{abstract}
We propose new fully discrete third-order accurate Active Flux and WENO methods based on truly multidimensional evolution operators for the two-dimensional acoustic equations. Building on the method of bicharacteristics,  several approximate evolution operators are derived that yield an improved  stability of the resulting schemes.  A linear stability analysis is applied to determine the maximal CFL number. The schemes are tested extensively on both continuous and discontinuous problems, confirming their robustness and accurate approximation even on coarse grids.
%
\end{abstract}




\section{Introduction}

The Active Flux (AF) method, originally introduced by Eymann and Roe \cite{eymann2011active,eymann2013multidimensional} and Roe et al. \cite{FR2015,article:Roe2017,article:Roe2018,article:Roe2020,article:Roe2021,proc:Roe2025}, is a relatively new variant of finite volume methods for hyperbolic conservation laws. Its key feature is the use of  point value degrees of freedom located along cell boundaries, in addition to cell average values that are commonly used in finite volume methods. This enables the construction of globally continuous piecewise quadratic reconstructions, which can be used to obtain fully discrete
third-order accurate methods with a compact stencil. By contrast, classical finite volume methods, which only use cell averages of conserved quantities as degrees of freedom, require an increased stencil to achieve high-order accuracy. The evolution of the point values is central to the AF method. For scalar conservation
laws, characteristics can be used to compute the evolution of the point values; see, e.g., \cite{eymann2013multidimensional,article:Barsukow2021,article:CCH2023}.
Exact evolution operators are available for the trivial case of advection and the nontrivial case of acoustics in one, two, and three dimensions. More details can be found in \cite{eymann2013multidimensional,FR2015,phd:Barsukow,article:BHKR2019,article:Barsukow2023}. Barsukow et al. \cite{article:BHKR2019,article:Barsukow2023} showed that Cartesian grid AF methods for acoustics, which use the exact evolution operator for evolving the point values, preserve all steady states. 

The development of AF methods for more general hyperbolic problems, where no exact evolution operators are available, is currently a very active field of research. The Euler equations of gas dynamics, as well as the equations of magnetohydrodynamics, are important examples. Eymann and Roe \cite{eymann2013multidimensional} suggested using splitting methods that separately approximate acoustic wave propagation and nonlinear transport. This approach was further developed in the PhD thesis of Fan \cite{PhD:Fan2017}. A related recent contribution of Barsukow \cite{article:Barsukow2025} used a splitting method for updating the point values in the Euler equations.

Intensive recent research activities are devoted to the development of semi-discrete methods that use the degrees of freedom of AF methods. These methods are named generalized AF methods, see  \cite{article:ABK2025,article:DBK2025}, and point-average-moment polynomiAl-interpreted (PAMPA) schemes, see \cite{preprint:ALB2025}. The point and cell average values provide compact stencils for the spatial discretisation. High-order stability-preserving Runge-Kutta methods are used for the temporal discretisation. The method of lines approach also simplifies theoretical studies, as shown in \cite{preprint:Barsukow-PG-2025,preprint:BKLNOR2025,preprint:AOL2025}. However, the temporal discretisation of semi-discrete methods increases the stencil, and the resulting methods have quite restrictive CFL conditions. 

Our own work is devoted to the further development of fully discrete AF methods. Such methods rely on evolution operators for the point value degrees of freedom. Previously, Luk\'a\v{c}ov\'a et al. \cite{article:LMW2000,article:LSW2002} developed evolution operators for linear hyperbolic systems using the Method of Bicharacteristics. These evolution operators are based on integrals along the  base of characteristic cones and thus take all directions of wave propagation into account. The so-called EG2 evolution operator for acoustics \cite{article:LMW2000} and linearized Euler equations \cite{article:LSW2002} is so far the only known third-order accurate approximate evolution operator of this type. In \cite{article:CHL2024}, fully discrete third-order accurate AF methods for acoustics and linearised Euler equations have been presented using the EG2 evolution operator for the point value update. The study of the acoustic equations allows one to compare the resulting AF methods with the presumably optimal method that uses the exact evolution operator.  While the use of the exact evolution operator for acoustics provides AF methods stable for time steps corresponding to a CFL condition of the form $\mbox{CFL}\le 0.5$, i.e., optimal for AF methods with this stencil \cite{article:CHK2021}, the use of the EG2 evolution operator leads to a more restrictive stability condition of the form $\mbox{CFL}\le 0.279$. The same time step restriction was observed for the linearised Euler equations \cite{article:CHL2024}. Early evidence of the EG2 evolution operator for the nonlinear Euler equations appears in \cite{article:CHL2024}. In \cite{preprint:CHP2025}, the nonlinear situation was explored in more detail, including an accuracy study for smooth problems and computations of shock waves that require limiting. All previous studies confirm that fully discrete AF methods yield accurate results even on coarse grids. 

In this paper, we consider the two-dimensional acoustic equations, which are given by
	\begin{equation}
		\label{eq:acoustic}
		\begin{split}
			\partial_t p &+ c \nabla \cdot \mathbf{u}  = 0 \\
			\partial_t \mathbf{u} &+ c \nabla p  = \mathbf{0},
		\end{split}
	\end{equation}
where $\mathbf{u}:\mathbb{R}^2 \times \mathbb{R}^+ \to \mathbb{R}^2$ denotes the velocity, $p: \mathbb{R}^2 \times \mathbb{R}^+ \to \mathbb{R}$ is the pressure, and $c \in \mathbb{R}^+$  represents the speed of sound. In the following, we propose new approximate evolution operators for acoustics that yield third-order accurate AF and improved stability.  To this end, we first derive evolution operators that are exact for quadratic plane waves. This suggests a modification of the EG2 evolution operator by applying the Taylor expansion over mantle integrals and carefully adjusting the integral over the basis of the characteristic cone to yield an exact solution to the wave equation in 1D. While the resulting AF method does not offer improved stability, it yields smaller global errors and a simpler implementation.  In the next step, we propose consistent modifications to the EG2 evolution operator that improve the stability of the AF method. These new operators integrate over several circles of different radii rather than a single one. Here, more efficient implementations can be obtained by replacing the exact integration with quadrature rules. Parameter studies show that versions of these new evolution operators lead to AF methods, which are stable for time steps that satisfy $\mbox{CFL} \le 0.44$, i.e., close to the optimal limit. Finally, we also study the influence of the reconstruction operator on the stability and accuracy of the AF method. We investigate two biquadratic  reconstruction operators using either the AF degrees of freedom or a CWENO reconstruction based on the neighbouring cell averages.
In the latter case, the numerical fluxes are  again computed using Simpson's rule, and the solution at the quadrature nodes is approximated using our different approximate evolution operators. The resulting finite volume methods can be seen as generalised AF methods. They are, by construction, third-order accurate and appear to be stable for all considered approximate EG operators at $\mbox{CFL} = 0.5$. Moreover, by including all six neighbours of the edge midpoints in the approximate evolution, larger time steps corresponding to CFL=0.7 are allowed.

For smooth test problems, the magnitude of the error of the AF method with a CWENO reconstruction (AFCW)  is typically larger than that of the  AF methods that use the classical AF degrees of freedom for the reconstruction. This seems to be due to a larger stencil, which, on the other hand, allows a less restricted stability condition.  Note, however, that for CFL=0.7, the AFCW method with the so-called EG$^{\rm{quad}}$ operator yields errors comparable to those obtained by the AF reconstruction.
While all methods provide accurate approximations of discontinuous solution structures, the AFCW method shows slightly more smearing and a slightly less accurate approximation of the stationary vortex.  

This paper is organised as follows. In Section \ref{sec:2}, we briefly review the two-dimensional AF method together with two existing evolution operators for updating of point values—the exact evolution operator and the EG2 evolution operator;  see, e.g.,\cite{article:CHL2024}. We then introduce a series of test problems that will be tested by different versions of AF methods. In Section \ref{sec:EGquadra}, we introduce a new evolution operator that exactly reproduces planar waves for quadratic initial data. We show that this operator is third-order accurate. These accuracy studies motivate a simplified implementation that replaces integration with simple quadrature formulas. In Section \ref{sec:EG2delta} we introduce another new approximate evolution operator, which modifies the EG2 operator in such a way that the stability of the resulting AF method is improved. We show that this operator is third-order accurate  and that stable methods are obtained for time steps which satisfy $\mbox{CFL} \le 0.44$. In Section \ref{sec:CWENO}, we explore the influence of the spatial reconstruction by constructing fully discrete AFCW methods, i.e. the AF methods with a CWENO reconstruction. While all methods are  third-order accurate, the AFCW methods are more sensitive to the choice of the approximate evolution operator and the newly proposed EG$^{\rm quad}$ yields the most accurate results.

\section{Fully discrete Cartesian Grid Active Flux methods}\label{sec:2}

Let us divide a two-dimensional computational domain $\Omega$ into uniform rectangular cells
\begin{equation*}
\Omega_{i,j}:=[x_{i-\frac{1}{2}},x_{i+ \frac{1}{2}}]\times[y_{j-\frac{1}{2}}, y_{j+ \frac{1}{2}}]
\end{equation*}
centered at $\lb x_i,y_j\rb =\lb \lb x_{i-\frac12}+x_{i+\frac12}\rb /2,\ \lb y_{j-\frac12}+y_{j+\frac12}\rb /2\rb $, with $x_{i+\frac12}-x_{i-\frac12}\equiv\Delta x$ and $y_{j+\frac12}-y_{j-\frac12}\equiv\Delta y$ for all $i,j$. Denote by $Q_{i,j}\lb t\rb $ the cell average of $Q\lb \cdot,\cdot,t\rb $ over $\Omega_{i,j}$
\begin{equation*}
Q_{i,j}\lb t\rb \;\approx\;\frac{1}{\Delta x\,\Delta y}\int_{\Omega_{i,j}} Q\lb x,y,t\rb \,\mathrm{d}x\,\mathrm{d}y,
\end{equation*}
and suppose that all $Q_{i,j}\lb t\rb $ are available at a given time level $t=t^n\ge0$.
A finite volume method can be written as follows
\begin{equation}\label{eqn:fv}
Q_{i,j}^{n+1} = Q_{i,j}^n - \frac{\Delta t}{\Delta x} \left(
  F_{i+\frac{1}{2},j}-F_{i-\frac{1}{2},j} \right) - \frac{\Delta
  t}{\Delta y} \left( G_{i,j+\frac{1}{2}}-G_{i,j-\frac{1}{2}}\right),
\end{equation}
with numerical fluxes that approximate two-dimensional integrals, i.e.
\begin{equation}\label{eqn:FGintegrals}
  \begin{aligned}
F_{i-\frac{1}{2},j} & \approx \frac{1}{\Delta t \Delta y}
\int_{t_n}^{t_{n+1}} \int_{y_{j-\frac{1}{2}}}^{y_{j+\frac{1}{2}}}
f( q( x_{i-\frac{1}{2}},y,t))  {\rm d}y {\rm d}t,\\
G_{i,j-\frac{1}{2}} & \approx \frac{1}{\Delta t \Delta x}
\int_{t_n}^{t_{n+1}} \int_{x_{i-\frac{1}{2}}}^{x_{i+\frac{1}{2}}}
g(q(x,y_{j-\frac{1}{2}},t ))  {\rm d} x {\rm d}t.
\end{aligned}
\end{equation}

Numerical methods typically require a {\em reconstruction}, an {\em evolution} and an {\em averaging} step. The averaging step computes cellaverage values at the new time level. It is implemented via the finite volume update in (\ref{eqn:fv}) and a choice of quadrature rules for approximating the integrals in (\ref{eqn:FGintegrals}). In AF methods, the most common choice of the quadrature rule is Simpson's rule, which leads to fluxes of the form
\begin{equation*}\label{eqn:simpson}
\begin{aligned}
F_{i-\frac{1}{2},j}
&= \frac{1}{36}\Big[
 f\big(q(x_{i-\frac{1}{2}},y_{j-\frac{1}{2}},t_n)\big)
+ 4\,f\big(q(x_{i-\frac{1}{2}},y_{j},t_n)\big)
+ f\big(q(x_{i-\frac{1}{2}},y_{j+\frac{1}{2}},t_n)\big) \\
&+ 4\,f\big(q(x_{i-\frac{1}{2}},y_{j-\frac{1}{2}},t_{n+\frac{1}{2}})\big)
+ 16\,f\big(q(x_{i-\frac{1}{2}},y_{j},t_{n+\frac{1}{2}})\big)
+ 4\,f\big(q(x_{i-\frac{1}{2}},y_{j+\frac{1}{2}},t_{n+\frac{1}{2}})\big) \\
&+ f\big(q(x_{i-\frac{1}{2}},y_{j-\frac{1}{2}},t_{n+1})\big)
+ 4\,f\big(q(x_{i-\frac{1}{2}},y_{j},t_{n+1})\big)
+ f\big(q(x_{i-\frac{1}{2}},y_{j+\frac{1}{2}},t_{n+1})\big)
\Big],
\end{aligned}
\end{equation*}
and analogously for $G_{i,j-\frac{1}{2}}$.

For the reconstruction, we use the globally continuous, piecewise quadratic Cartesian grid AF reconstruction.
On each cell we introduce reference coordinates $\lb \xi,\eta\rb \in[-1,1]^2$ via the affine mapping
$$
\xi = \frac{2\lb x - x_i\rb }{\Delta x}, \quad \eta = \frac{2\lb y - y_j\rb }{\Delta y}.
$$
We reconstruct a local quadratic polynomial on $\Omega_{i,j}$ at time $t^n$ in a modal basis,
\begin{equation}\label{recon-formula}
q_{i,j}\lb \xi,\eta\rb =\sum\limits^{m}_{k=0}C_kN_k\lb \xi,\eta\rb.
\end{equation}
For AF reconstruction, $m=8$ is chosen, the coefficients $C_k$ and basis functions $N_k$ are given in \cite{article:CHK2021}.

The AF method requires approximations of point values along the grid cell boundary at times $t_{n+\frac{1}{2}}$ and $t_{n+1}$. Their approximations are crucial parts of fully discrete AF methods. In this paper, we use truly multi-dimensional evolution operators to approximate these point values. In particular, we will derive new evolution operators for the acoustic equations and compare the resulting AF method with previously proposed methods.

\subsection{Previously proposed truly multi-dimensional evolution operators }
\subsubsection{Exact evolution operator}
Eymann and Roe \cite{eymann2013multidimensional}, Fan and Roe \cite{FR2015} and Barsukow et al.\ \cite{article:BHKR2019} derived the exact evolution formulas for point values. The formulas, expressed using derivatives in radial direction $r$, have the form
\begin{equation*}
  \begin{aligned}
    p\lb {\bf x},t\rb  & = \partial_r \lb r M_r \{p\lb {\mathbf x},0\rb  \} \rb |_{r=c t} -
      \frac{1}{c t} \partial_r \lb r^2 M_r \{ {\mathbf n} \cdot
        \mathbf{u}\lb {\mathbf x},0\rb   \} \rb |_{r=c t} \\
      \mathbf{u}\lb {\mathbf x},t\rb  & = \mathbf{u}( {\mathbf x},0(  - \frac{1}{c t} \partial_r
      ( r^2 M_r \{  \mathbf{n} p( {\mathbf x},0)  \} ) |_{r = ct}
      \\
      & \qquad + \int_0^{c
          t} \frac{1}{r} \partial_r ( \frac{1}{r} \partial_r
          ( r^3 M_r \{ ( {\mathbf n} \cdot {\mathbf u}( {\mathbf x},0)  )  {\mathbf n}
           ( - r M_r \{ {\mathbf u}( {\mathbf x},0)  \} ) \, \D r,
  \end{aligned}
\end{equation*}
where $M_r \{f\lb {\mathbf x}\rb  \}$ describes the spherical mean of a scalar function $f$ over a disc with radius $r$ and can be computed via 
\begin{equation*}
M_r \{f\lb {\mathbf x}\rb  \} := \frac{1}{2 \pi r} \int_{0}^{2 \pi } \int_{0}^{r} f\lb x + \tilde{r} \cos \lb \theta\rb , y + \tilde{r} \sin \lb \theta\rb \rb  \frac{\tilde{r}}{\sqrt{r^2 - \tilde{r}^2}} \D \tilde{r} \D \theta.
\end{equation*}
The resulting AF method based on the globally continuous piecewise quadratic reconstruction is third-order accurate. Note that the accuracy is limited by the reconstruction. The flux computation using Simpson's rule would allow fourth-order accuracy. It was shown in \cite{article:CHK2021}, that the resulting method is stable for time steps that satisfy $\mbox{CFL} \le 0.5$, with 
\begin{equation*}
\mbox{CFL}  = \max \left\{\frac{c \Delta t}{\Delta x}, \frac{c \Delta t}{\Delta y}\right\}.
\end{equation*}
This is optimal for a third-order accurate method that uses a compact stencil. Barsukow et al. \cite{article:BHKR2019} showed that the resulting AF method preserves all steady states. Thus, the method in particular preserves vorticity.

\subsubsection{EG2 approximate evolution operator}
Luk\'a\v{c}ov\'a et al.\ \cite{article:LMW2000} derived an exact and several approximate evolution operators for acoustics using the method of bicharacteristics. They derived an exact and approximate expression for the solutions $p, u, v$ at a given point $\lb x,y\rb $ at time $t+\Delta t$ that can be computed from a known representation of the solution in the neighbourhood of the point $\lb x,y\rb $ at time $t$. The exact operator reads
	\begin{equation}\label{eq:exactEvoBic}
	\begin{split}
		p\lb P\rb  &= \frac{1}{2 \pi} \int_{0}^{2 \pi} \lb p\lb \mathbf{Q}\lb \theta\rb \rb  - u\lb \mathbf{Q}\lb \theta\rb \rb \cos \lb \theta\rb  -  v\lb \mathbf{Q}\lb \theta\rb \rb \sin \lb \theta\rb  \rb \D 	\theta \\
		&-  \frac{1}{2 \pi } \int_{t}^{t + \Delta t} \int_{0}^{2 \pi}S\lb \tilde{t},\theta\rb \D \theta \D \tilde{t},\\
		u \lb P\rb  &= \frac{1}{2 \pi} \int_{0}^{2 \pi} \lb p\lb \mathbf{Q}\lb \theta\rb \rb \cos \lb \theta\rb  + u\lb \mathbf{Q}\lb \theta\rb \rb \cos ^2 \lb \theta\rb + v\lb \mathbf{Q}\lb \theta\rb \rb \sin\lb \theta\rb  \cos \lb \theta\rb \rb \D \theta \\
		& + \frac{1}{2} u\lb P'\rb  + \frac{1}{2 \pi } \int_{t}^{t + \Delta t} \int_{0}^{2 \pi}S\lb \tilde{t},\theta\rb  \cos \lb \theta\rb \D \theta \D \tilde{t} \\
		& - \frac{1}{2} c \int_{t}^{t + \Delta t} p_x\lb  P\lb \tilde{t}\rb \rb  \D \tilde{t}, \\
		v \lb P\rb  &= \frac{1}{2 \pi} \int_{0}^{2 \pi} \lb p\lb \mathbf{Q}\lb \theta\rb \rb \sin \lb \theta\rb  + v\lb \mathbf{Q}\lb \theta\rb \rb \sin ^2 \lb \theta\rb  +u\lb \mathbf{Q}\lb \theta\rb \rb \sin\lb \theta\rb  \cos \lb \theta\rb \rb \D \theta \\
		& + \frac{1}{2} v\lb P'\rb  + \frac{1}{2 \pi } \int_{t}^{t + \Delta t} \int_{0}^{2 \pi}S\lb \tilde{t},\theta\rb  \sin \lb \theta\rb \D \theta \D \tilde{t} \\
		& - \frac{1}{2} c \int_{t}^{t + \Delta t} p_y\lb  P\lb \tilde{t}\rb \rb  \D \tilde{t}, 
	\end{split}
\end{equation}
where 
$P = \lb x,y,t+\Delta t\rb $, $P' = \lb x,y,t\rb $ , $\mathbf{Q}\lb \theta\rb =\lb x+c\Delta t
\cos\lb \theta\rb ,y+c\Delta t \sin\lb \theta\rb ,t\rb$, and
\begin{equation*}
		S\lb \tilde{t},\theta\rb  = c \lb u_x\lb \tilde{x},\tilde{y},\tilde{t}\rb \sin ^2\lb \theta\rb  - \lb u_y\lb \tilde{x},\tilde{y},\tilde{t}\rb +v_x\lb \tilde{x},\tilde{y},\tilde{t}\rb \rb \sin \lb \theta\rb  \cos \lb \theta\rb  + v_y\lb \tilde{x},\tilde{y},\tilde{t}\rb \cos ^2 \lb \theta\rb  \rb
\end{equation*}
with 
\begin{equation*}
		\lb \tilde{x},\tilde{y}\rb  = \lb  x + c\lb t + \Delta t -\tilde{t}\rb \cos \lb \theta\rb ,  y + c\lb t + \Delta t -\tilde{t}\rb \sin \lb \theta\rb \rb .
\end{equation*}

By combining exact and approximate steps, they derived a third-order evolution operator. An exact step is given by 
\begin{equation} \label{eq:integration}
	\begin{split}
		u\lb P\rb  &= u\lb P'\rb  - c \int_{t}^{t + \Delta t} p_x\lb x,y,\tilde{t}\rb  \,\D \tilde{t},\\
		v\lb P\rb  &= v\lb P'\rb  - c \int_{t}^{t + \Delta t} p_y\lb x,y,\tilde{t}\rb  \,\D \tilde{t},
	\end{split}
\end{equation}
which follows by integrating the second and third equations of \eqref{eq:acoustic} from $P'$ to $P$. Further details on the derivation of a third-order approximate evolution operator are given in \cite{article:CHL2024} and \cite{article:LMW2000}. The operator is called EG2 and is given as follows
\begin{equation}\label{eq:eg2}
\resizebox{0.93\linewidth}{!}{$
  \begin{aligned}
   p\lb P\rb  & = \frac{1}{\pi} \int_0^{2\pi} \lb p\lb \mathbf{Q}\lb \theta\rb \rb  -
      u\lb \mathbf{Q}\lb \theta\rb \rb  \cos\lb \theta\rb  - v\lb \mathbf{Q}\lb \theta\rb \rb  \sin\lb \theta\rb  \rb
     {\rm d} \theta - p\lb P'\rb  + \mathcal{O} \lb \Delta t^3\rb,  \\
    u\lb P\rb  &= \frac{1}{\pi} \int_0^{2\pi} \lb -p\lb \mathbf{Q}\lb \theta\rb \rb 
      \cos\lb \theta\rb  + u\lb \mathbf{Q}\lb \theta\rb \rb  \lb 2 \cos^2\lb \theta\rb  -
        \frac{1}{2} \rb + 2 v\lb \mathbf{Q}\lb \theta\rb \rb  \sin\lb \theta\rb \cos\lb \theta\rb 
    \rb {\rm d} \theta \\&+ \mathcal{O} \lb \Delta t^3\rb, \\
    v\lb P\rb  & =\frac{1}{\pi} \int_0^{2\pi} \lb -p\lb \mathbf{Q}\lb \theta\rb \rb 
      \sin\lb \theta\rb  + 2 u\lb \mathbf{Q}\lb \theta\rb \rb  \sin\lb \theta\rb  \cos\lb \theta\rb  +
      v\lb \mathbf{Q}\lb \theta\rb \rb  \lb 2 \sin^2 \lb \theta\rb  - \frac{1}{2}\rb
    \rb {\rm d} \theta \\&+ \mathcal{O} \lb \Delta t^3\rb .
  \end{aligned}
  $}
\end{equation}

The AF method based on the EG2 approximate  evolution operator has been introduced and studied in \cite{article:CHL2024}. The method is third-order accurate and stable for time steps satisfying $\mbox{CFL} \le 0.279$. Note that the accuracy of the method is limited by the reconstruction as well as the accuracy of the evolution operator. In contrast to the exact evolution operator, the AF method  no longer exactly preserves vorticity.

\subsection{Test problems}
In this section, we define a series of test problems for the acoustic equations and illustrate the performance of the AF method with exact evolution for these problems.

The first problem was proposed by Luk\'a\v{c}ov\'a et al. \cite{article:LMW2000}, providing an exact smooth time periodic solution for the acoustic equations with irrotational initial values, which can be used for numerical convergence studies. Results are given in Tables \ref{Tab:convergenceProblem1exact} and \ref{Tab:convergenceProblem1exactt1}.

\begin{example}\label{ex:1}
We consider exact solutions of the acoustic equations of the form
  \begin{equation*}
    \begin{split}
p\lb x,y,t\rb  & = - \frac{1}{c} \cos\lb 2 \pi c t\rb \lb  \sin\lb 2 \pi x\rb  +
  \sin\lb 2 \pi y\rb \rb, \\
u\lb x,y,t\rb  & = \phantom{-} \frac{1}{c} \sin\lb 2 \pi c t\rb  \cos\lb 2 \pi x\rb,  \\
v\lb x,y,t\rb  & = \phantom{-} \frac{1}{c} \sin\lb 2 \pi c t\rb  \cos\lb 2 \pi y\rb.
    \end{split}
  \end{equation*}
Let the computation domain be $[-1,1]\times[-1,1]$, with periodic boundary conditions imposed in both the $x$- and $y$-directions. The solution at time $t=0$ is used as the initial data.
\end{example}

\begin{table}[!ht]
	\centering
	\caption{Errors measured in the $L_1$-norm and EOC for Example~\ref{ex:1} 
		using exact evolution with $\mathrm{CFL} = 0.5$ at $t = 0.1$.}
	\label{Tab:convergenceProblem1exact}
	\vspace*{0.15cm}
	
	\sisetup{
		scientific-notation = true,
		round-mode = places
	}
	
	\renewcommand{\arraystretch}{0.9}
	\setlength{\tabcolsep}{4pt}
	
	\begin{tabular}{
			c
			*{2}{S[round-precision=2, table-format=1.2e-2]}
			*{2}{S[round-precision=3, table-format=1.4]}
		}
		\toprule
		Res. & \multicolumn{2}{c}{Error} & \multicolumn{2}{c}{EOC} \\
		\cmidrule(lr){2-3} \cmidrule(lr){4-5}
		& {$p$} & {$u,v$} & {$p$} & {$u,v$} \\
		\midrule
		$64\times64$  & \num{1.6042545843454000e-05} & \num{9.3133794254334842e-06} & {---} & {---} \\
		$128\times128$ & \num{2.0386219828327593e-06} & \num{1.1726376779459001e-06} & 2.9762 & 2.9895 \\
		$256\times256$ & \num{2.5426897907752349e-07} & \num{1.4564171803536017e-07} & 3.0032 & 3.0093 \\
		\bottomrule
	\end{tabular}
\end{table}

\begin{table}[!ht]
	\centering
	\caption{Errors measured in the $L_1$-norm and EOC for Example~\ref{ex:1} 
		using exact evolution with $\mathrm{CFL} = 0.5$ at $t = 1$.}
	\label{Tab:convergenceProblem1exactt1}
	\vspace*{0.15cm}
	
	\sisetup{
		scientific-notation = true,
		round-mode = places
	}
	
	\renewcommand{\arraystretch}{0.9}
	\setlength{\tabcolsep}{4pt}
	
	\begin{tabular}{
			c
			*{2}{S[round-precision=2, table-format=1.2e-2]}
			*{2}{S[round-precision=3, table-format=1.4]}
		}
		\toprule
		Res. & \multicolumn{2}{c}{Error} & \multicolumn{2}{c}{EOC} \\
		\cmidrule(lr){2-3} \cmidrule(lr){4-5}
		& {$p$} & {$u,v$} & {$p$} & {$u,v$} \\
		\midrule
		$64\times64$   & \num{1.9950350026772478e-04} & \num{2.0145858664183745e-06} & {---} & {---} \\
		$128\times128$ & \num{2.5059362405697071e-05} & \num{1.2735804246916971e-07} & 2.9930 & 3.9835 \\
		$256\times256$ & \num{3.1362045778548605e-06} & \num{7.9895233839950322e-09} & 2.9983 & 3.9946 \\
		\bottomrule
	\end{tabular}
\end{table}

The second problem was proposed by Chudzik et al. \cite{article:CHL2024}. Similar to Example \ref{ex:1}, the test consists of an exact smooth time-periodic solution of the acoustic equations but with non-irrotational data. The results are given in Tables \ref{Tab:convergenceProblem2exact} and \ref{Tab:convergenceProblem2exactt1}.
\begin{example}\label{ex:2}
We consider exact solutions of the acoustic equations of the form
\begin{equation*}
  \begin{split}
p\lb x,y,t\rb  & = \phantom{-} \frac{1}{c} \lb \cos\lb 2 \pi x\rb  - \cos\lb 2 \pi y\rb  \rb
\sin\lb 2 \pi c t\rb  \\
u\lb x,y,t\rb  & = - \frac{1}{c} \lb \sin\lb 2 \pi x\rb  \cos\lb 2 \pi c t\rb  +
  \sin\lb 2 \pi y\rb  \rb \\
v\lb x,y,t\rb  & = \phantom{-} \frac{1}{c} \lb \sin\lb 2 \pi x\rb  + \sin\lb 2 \pi y\rb  \cos\lb 2
  \pi c t\rb  \rb.
  \end{split}
\end{equation*}
Test computations are performed on the interval $[-1,1]\times [-1,1]$ with periodic boundary conditions in $x$- and $y$-directions, using the solution at time $t=0$ as initial values. 
\end{example}

\begin{table}[!ht]
	\centering
	\caption{Errors measured in the $L_1$-norm and EOC for Example~\ref{ex:2} 
		using exact evolution with $\mathrm{CFL} = 0.5$ at $t = 0.1$.}
	\label{Tab:convergenceProblem2exact}
	\vspace*{0.15cm}
	
	\sisetup{
		scientific-notation = true,
		round-mode = places
	}
	
	\renewcommand{\arraystretch}{0.9}
	\setlength{\tabcolsep}{4pt}
	
	\begin{tabular}{
			c
			*{2}{S[round-precision=2, table-format=1.2e-2]}
			*{2}{S[round-precision=3, table-format=1.4]}
		}
		\toprule
		Res. & \multicolumn{2}{c}{Error} & \multicolumn{2}{c}{EOC} \\
		\cmidrule(lr){2-3} \cmidrule(lr){4-5}
		& {$p$} & {$u,v$} & {$p$} & {$u,v$} \\
		\midrule
		$64\times64$   & \num{1.1963386887980978e-05} & \num{1.2652061957520213e-05} & {---} & {---} \\
		$128\times128$ & \num{1.5022014177906815e-06} & \num{1.6021279860692110e-06} & 2.9935 & 2.9813 \\
		$256\times256$ & \num{1.8610048114829737e-07} & \num{1.9969677502133022e-07} & 3.0129 & 3.0041 \\
		\bottomrule
	\end{tabular}
\end{table}

\begin{table}[!ht]
	\centering
	\caption{Errors measured in the $L_1$-norm and EOC for Example~\ref{ex:2} 
		using exact evolution with $\mathrm{CFL} = 0.5$ at $t = 1$.}
	\label{Tab:convergenceProblem2exactt1}
	\vspace*{0.15cm}
	
	\sisetup{
		scientific-notation = true,
		round-mode = places
	}
	
	\renewcommand{\arraystretch}{0.9}
	\setlength{\tabcolsep}{4pt}
	
	\begin{tabular}{
			c
			*{2}{S[round-precision=2, table-format=1.2e-2]}
			*{2}{S[round-precision=3, table-format=1.4]}
		}
		\toprule
		Res. & \multicolumn{2}{c}{Error} & \multicolumn{2}{c}{EOC} \\
		\cmidrule(lr){2-3} \cmidrule(lr){4-5}
		& {$p$} & {$u,v$} & {$p$} & {$u,v$} \\
		\midrule
		$64\times64$   & \num{2.4980281566831169e-06} & \num{1.5719665415403451e-04} & {---} & {---} \\
		$128\times128$ & \num{1.6016748496368367e-07} & \num{1.9697610150159046e-05} & 3.9631 & 2.9965 \\
		$256\times256$ & \num{1.0112398678926602e-08} & \num{2.4636731237458242e-06} & 3.9854 & 2.9991 \\
		\bottomrule
	\end{tabular}
\end{table}

The next test problem was introduced by Barsukow et al. \cite{article:BHKR2019} to study the ability of methods to preserve a stationary vortex for acoustics. It is motivated by the well-known Gresho vortex problem for the  Euler equations.
\begin{example}\label{ex:3}
We approximate solutions of the two-dimensional acoustic equations with initial values of the form
\begin{equation*}
p\lb r,0\rb  = 0, \quad \mathbf{u}\lb x,y,0\rb  = \mathbf{n} \left\{
  \begin{array}{ccc}
5 r & : & 0 \le r \le 0.2 \\
2-5 r & : & 0.2 < r \le 0.4 \\
0 & : & r  > 0.4,\end{array}\right.
\end{equation*}
with $r = \sqrt{x^2 + y^2}$, $\mathbf{n} =\lb -\sin\lb \theta\rb ,\cos\lb \theta\rb \rb ^T$, $\theta \in [0,2\pi) $ and $\mathbf{u}= \lb u,v\rb ^T$. The computational domain is $[-1,1]\times [-1,1]$, and double periodic boundary conditions are imposed.
\end{example}
Barsukow et al.\ \cite{article:BHKR2019} showed that the third-order accurate Cartesian grid AF method with an exact evolution operator is stationary preserving, i.e.,\ it exactly preserves a discrete representation of all stationary states. In order to test how well different methods  approximate the steady state, they proposed to compute solutions at time $t=100$. The left column of  Figure \ref{fig:EG2_EG11_Vortex} shows numerical solutions on grids with  $64 \times 64$ and $128\times 128$ cells that conform to the theoretical result.

In the final test problem, we study the performance of the method for the approximation of discontinuous solution structures. This problem was also considered in \cite{article:LMW2000,article:CHL2024}.
\begin{example}\label{ex:4}
We consider the two-dimensional acoustic equations with initial values of the form
  \begin{equation*}
    \begin{split}    
u_0\lb x,y\rb &=v_0\lb x,y\rb =\begin{cases}
\phantom{-}\frac{1}{\sqrt{2}}&: |y|<|x| \\ -\frac{1}{\sqrt{2}}&: |y|\ge|x| 
\end{cases}\\
p_0\lb x,y\rb &=1
\end{split}
\end{equation*}
The speed of sound is set to $c=1$ and the computational domain is $[-1,1]\times [-1,1]$. We use zero-order extrapolation at the boundaries, which takes the wave propagation in the diagonal direction into account, e.g., ghost cells at the left boundary $\mathcal{C}_{0,j}$ are filled with cells $\mathcal{C}_{1,j-1}$ if they are not close to the bottom
left corner; otherwise, they are filled with cells $\mathcal{C}_{1,j+1}$.  
\end{example}

The top row of Figure \ref{fig:EG2cDiscontinuous} shows approximate solutions of Example \ref{ex:4} at $t=0.5$ calculated on grids with $64 \times 64$ and $128 \times 128$ grid cells. 

In the remaining sections, we introduce new third-order accurate evolution operators and investigate their stability and performance for the test problems.

\section{New Evolution Operator Reproducing Exactly Quadratic Plane Waves}\label{sec:EGquadra}

Inspired by the previous work of Luk\'a\v{c}ov\'a \cite{article:LMW2004}, where stable approximate evolution operators for the FVEG method were derived, we apply a similar strategy here to improve the accuracy and stability of the AF method. In this section, we introduce a new evolution operator that is derived based on the exact solution for one-dimensional quadratic initial data. Motivated by the quadratic continuous reconstruction employed in our scheme, we aim to construct an evolution operator that can naturally accommodate and accurately evolve quadratic terms. The use of this operator leads to a third-order fully discrete AF scheme, which allows for time steps comparable to those employed in the method using EG2 as the evolution operator.

We consider one-dimensional data of the form
\begin{equation*}
\begin{aligned}
		p\lb x,y,0\rb &=
		\begin{cases}
			p^R x^2, & x>0,\\
			0, & x\le 0,
		\end{cases}
		&\qquad
		u\lb x,y,0\rb &=
		\begin{cases}
			u^R x^2, & x>0,\\
			0, & x\le 0,
		\end{cases}
		&\qquad
		v\lb x,y,0\rb =0.
\end{aligned}
\end{equation*}
For simplicity, we have taken the left state to be zero. Note that for the linear acoustic equation system, the superposition principle holds, and a more general piecewise quadratic solution can easily be deduced. The exact solution is given by
\begin{equation}\label{exact-solution-q}
\begin{aligned}
		p\lb x,y,t\rb &=
		\begin{cases}
			p^R\lb x^2+c^2t^2\rb -2u^R x c t, & x>ct,\\
			\dfrac{1}{2}\lb p^R - u^R\rb \lb x+ct\rb ^2, & -ct< x\le ct,\\
			0, & x<-ct,
		\end{cases}\\[4pt]
		u\lb x,y,t\rb &=
		\begin{cases}
			u^R\lb x^2+c^2t^2\rb -2p^R x c t, & x>ct,\\
			\dfrac{1}{2}\lb u^R - p^R\rb \lb x+ct\rb ^2, & -ct< x\le ct,\\
			0, & x<-ct,
		\end{cases}\\[4pt]
		v\lb x,y,t\rb &=0.
\end{aligned}
\end{equation}

Substituting these data into the mantle integrals \eqref{eq:exactEvoBic} leads to an approximate evolution operator, called $\mathrm{EG}^{\text{quad}}$.

More precisely, we modify the EG2 operator such that its numerical solution over one time step remains consistent with the exact solution \eqref{exact-solution-q}. 
Direct calculations in the first equation of the approximate evolution operator EG2 \eqref{eq:eg2} give

\begin{align*}
	\frac1\pi\int^{2\pi}_0p\lb \mathbf{Q}\lb \theta\rb \rb \dthe=&\frac12\lb c\TS\rb ^2p^R, \\
	\frac1\pi\int^{2\pi}_0u\lb \mathbf{Q}\lb \theta\rb \rb \cos\lb \theta\rb \dthe=&\frac4{3\pi}\lb c\TS\rb ^2u^R,\\
	\frac1\pi\int^{2\pi}_0v\lb \mathbf{Q}\lb \theta\rb \rb \sin\lb \theta\rb \dthe=&\frac4{3\pi}\lb c\TS\rb ^2v^R.
\end{align*}

In order to preserve an exact planar wave solution, we derive a suitable approximation justified by  
Lemma \ref{lem} with $n=4$. This yields
\begin{align*}
\int^{2\pi}_0u\lb \mathbf{Q}\lb \theta\rb \rb \cos\lb \theta\rb \dthe = &\frac{\pi}{2}\left[u\lb \mathbf{Q}\lb 0\rb \rb -u\lb \mathbf{Q}\lb \pi\rb \rb \right] + \mathcal{O}\lb \TS^3\rb ,\\
\int^{2\pi}_0v\lb \mathbf{Q}\lb \theta\rb \rb \sin\lb \theta\rb \dthe =  &\frac{\pi}{2}\left[v\lb \mathbf{Q}\lb \frac{\pi}{2}\rb \rb -v\lb \mathbf{Q}\lb \frac{3\pi}{2}\rb \rb \right]+ \mathcal{O}\lb \TS^3\rb .
\end{align*}
Next, we rewrite the corresponding terms in the EG2 operator to achieve the exact solution after one time step for one-dimensional quadratic data. Specifically, we have for $\omega_1, \omega_2 \in \mathbb{R}$
\begin{align*}		&\frac{1}{\pi}\int^{2\pi}_0u\lb \mathbf{Q}\lb \theta\rb \rb \cos\lb \theta\rb \dthe\\
=&\omega_1\int^{2\pi}_0u\lb  \mathbf{Q}\lb \theta\rb \rb \cos\lb \theta\rb \dthe+\lb \frac{1}{\pi}-\omega_1\rb \int^{2\pi}_0u\lb \mathbf{Q}\lb \theta\rb \rb \cos\lb \theta\rb \dthe\\
			=&\omega_1\int^{2\pi}_0u\lb \mathbf{Q}\lb \theta\rb \rb \cos\lb \theta\rb \dthe+\lb \frac{1}{\pi}-\omega_1\rb \frac{\pi}{2}[u\lb \mathbf{Q}\lb 0\rb \rb -u\lb \mathbf{Q}\lb \pi\rb \rb ]+\mathcal{O}\lb \TS^3\rb \\
			=&\frac{4}{3}\omega_1 u^R\lb c\TS\rb ^2+\lb \frac12-\frac{\pi}{3}\omega_1\rb  u^R\lb c\TS\rb ^2 +\mathcal{O}\lb \TS^3\rb \\
			=&\lb \frac{4}{3}\omega_1-\lb \frac12-\frac{\pi}{2}\omega_1\rb \rb u^R\lb c\TS\rb ^2 +\mathcal{O}\lb \TS^3\rb , 
\end{align*} 
and
\begin{align*}		&\frac{1}{\pi}\int^{2\pi}_0v\lb \mathbf{Q}\lb \theta\rb \rb \sin\lb \theta\rb \dthe\\
=&\omega_2\int^{2\pi}_0v\lb \mathbf{Q}\lb \theta\rb \rb \sin\lb \theta\rb \dthe+\lb \frac{1}{\pi}-\omega_2\rb \int^{2\pi}_0v\lb \mathbf{Q}\lb \theta\rb \rb \sin\lb \theta\rb \dthe\\
			=&\omega_2\int^{2\pi}_0v\lb \mathbf{Q}\lb \theta\rb \rb \sin\lb \theta\rb \dthe+\lb \frac{1}{\pi}-\omega_2\rb \frac{\pi}{2}\left[v\lb \mathbf{Q}\lb \frac{\pi}{2}\rb \rb -v\lb \mathbf{Q}\lb \frac{3\pi}{2}\rb \rb \right]+\mathcal{O}\lb \TS^3\rb \\
			=&\frac{4}{3}\omega_2 v^R\lb c\TS\rb ^2+\lb \frac12-\frac{\pi}{2}\omega_2\rb  v^R\lb c\TS\rb ^2 +\mathcal{O}\lb \TS^3\rb \\
			=&\lb \frac{4}{3}\omega_2-\lb \frac12-\frac{\pi}{2}\omega_2\rb \rb v^R\lb c\TS\rb ^2 +\mathcal{O}\lb \TS^3\rb . 
\end{align*} 
The choice of $\omega_1 = 0$ and $\omega_2 = 0$ leads to 
\begin{align*}
\frac{1}{\pi}\int^{2\pi}_0u\lb \mathbf{Q}\rb \cos\theta\dthe&= \frac12u^R\lb c\TS\rb ^2+\mathcal{O}\lb \TS^3\rb , \\
\frac{1}{\pi}\int^{2\pi}_0v\lb \mathbf{Q}\rb \sin\theta\dthe&=\frac12v^R\lb c\TS\rb ^2+\mathcal{O}\lb \TS^3\rb .
\end{align*}
Substituting the above results in \eqref{eq:eg2}, one can obtain $\mathrm{EG}^{\text{quad}}$ \eqref{eq:phiP}. An analogous procedure leads directly to \eqref{eq:uP}, \eqref{eq:vP}.

\begin{align}
p(P)  =& -p \big( P' \big) + \frac{1}{\pi}\int_{0}^{2\pi}p(\mathbf{Q}(\theta)) \,\mathrm{d}\theta - \frac{1}{2}\left[u(\mathbf{Q}(0))  - u(\mathbf{Q}(\pi)) \right] \notag\\
		 &- \frac{1}{2}\left[v\big(\mathbf{Q}\left(\tfrac{\pi}{2}\right)\big)  - v\big(\mathbf{Q}\left( \tfrac{3\pi}{2}\right)\big) \right]+ \mathcal{O}(\Delta t^{3}), \label{eq:phiP} \\
u(P)     =& -\frac{1}{2}\left[p(\mathbf{Q}(0)) - p(\mathbf{Q}(\pi)) \right]+ \frac{1}{\pi}\int_{0}^{2\pi}[u(Q(\theta)) \left( 2\cos^{2}( \theta)  - \tfrac{1}{2}\right) \notag\\
		 &+ 2v(\mathbf{Q}(\theta)) \sin(\theta) \cos(\theta)]\,\mathrm{d}\theta+ \mathcal{O}(\Delta t^{3}), \label{eq:uP} \\
v\lb P\rb     =& -\frac{1}{2}\left[p\big(\mathbf{Q}\left(\tfrac{\pi}{2}\right)\big)  - p\big(\mathbf{Q}\left(\tfrac{3\pi}{2}\right)\big) \right]+ \frac{1}{\pi}\int_{0}^{2\pi}[v(\mathbf{Q}(\theta)) \left( 2\sin^{2}( \theta)  - \tfrac{1}{2}\right) \notag\\
		 &+ 2u(\mathbf{Q}(\theta)) \sin(\theta) \cos(\theta) ]\,\mathrm{d}\theta+ \mathcal{O}(\Delta t^{3}). \label{eq:vP}
\end{align}

Numerical results indicate that the AF method based on the $\mathrm{EG}^{\text{quad}}$ operator \eqref{eq:phiP}-\eqref{eq:vP} remains stable under the same CFL condition as the EG2 operator. In Sections \ref{sec:accuracy}, \ref{sec:vortex}, and \ref{sec:disc}, we will investigate the numerical performance of this new evolution operator and compare it with new operators that permit larger time steps, which will be introduced in the following section.

\section{New Evolution Operator with Increased Stability} \label{sec:EG2delta}
In this section, we derive an evolution operator that increases the stability of the resulting AF method. This operator is derived from the EG2 formulas with a slight modification that does not affect the order of accuracy but leads to a method that is stable even for much larger time steps. To derive the new evolution operator for $p$, we use the existing one from the EG2 operator. The first approximation in equation (\ref{eq:eg2}) depends on the value of a point at the previous time level. This point value is replaced by a third-order approximation. A suitable approximation can be found using Lemma  \ref{th:approximation_of_pv}.
\begin{lemma}\label{th:approximation_of_pv}
	Let $ f \in C^3\lb \Omega\rb $, where $ \Omega \subset \mathbb{R} ^2$ and that contains the closed disk $\overline{B_R\lb x_0,y_0\rb }$. For $R>0$,  define the circle parametrisation $\mathbf{Q}_R\lb \theta\rb =\lb x_0+R\cos\lb \theta\rb , y_0+R\sin\lb \theta\rb \rb $, $\theta \in [0, 2\pi]$. Then 
	\begin{equation*}
		f\lb x_0,y_0\rb  = \frac{1}{3}\lb\frac{4}{2\pi}\int_{0}^{2 \pi} f\lb \mathbf{Q}_{\frac{R}{2}}\lb \theta\rb \rb  {\rm d} \theta - \frac{1}{2\pi}\int_{0}^{2 \pi} f\lb \mathbf{Q}_{R}\lb \theta\rb \rb  {\rm d} \theta \rb + \mathcal{O}\lb R^3\rb ,
	\end{equation*}
	for $ R > 0 $ sufficiently small.
\end{lemma}

\begin{proof}
	Using the two-dimensional Taylor expansion at $\lb x_0,y_0\rb$, we obtain
	  \begin{equation*}
		\begin{aligned}
		f\lb x_0+R \cos \lb \theta\rb , y_0 +R \sin \lb \theta\rb \rb  = f\lb x_0,y_0\rb  + R \cos \lb \theta\rb  f_x\lb x_0,y_0\rb  +  R \sin \lb \theta\rb  f_y\lb x_0,y_0\rb  \\
		+  \frac{R^2}{2} \lb \cos ^2 \lb\theta\rb f_{xx}\lb x_0,y_0\rb + 2\cos\lb\theta\rb \sin\lb\theta\rb  f_{xy}\lb x_0,y_0\rb + \sin ^2 \lb\theta\rb  f_{yy}\lb x_0,y_0\rb  \rb \\
		+ \mathcal{O} \lb R^3\rb .
		\end{aligned}
	\end{equation*}
	Hence,
		\begin{equation*}
        \resizebox{\linewidth}{!}{$
		\begin{aligned}
		&\frac{1}{2\pi}\int_{0}^{2 \pi}4 f\lb\mathbf{Q}_{\frac{R}{2}}\lb\theta\rb \rb   -  f\lb\mathbf{Q}_{R}\lb\theta\rb \rb  {\rm d} \theta \\ &= \frac{1}{2\pi}\int_{0}^{2 \pi} 3 f\lb x_0,y_0\rb  + R \cos \lb \theta\rb  f_x\lb x_0,y_0\rb +R \sin \lb\theta\rb  f_y\lb x_0,y_0\rb  {\rm d} \theta
		 + \mathcal{O}\lb R^3\rb  \\
		 &=  \frac{1}{2\pi}\lb 3f\lb x_0,y_0\rb  \int_{0}^{2 \pi} 1{\rm d} \theta + R f_x\lb x_0,y_0\rb  \int_{0}^{2 \pi} \cos\lb\theta\rb {\rm d} \theta + R f_y\lb x_0,y_0\rb \int_{0}^{2 \pi} \sin\lb\theta\rb {\rm d} \theta \rb + \mathcal{O}\lb R^3\rb  \\
		 &= 3 f\lb x_0,y_0\rb  + \mathcal{O}\lb R^3\rb.
		\end{aligned}
        $}
	\end{equation*}
	Dividing by three yields the claim.
\end{proof}
\begin{remark}
	If $ f \in C^4 \lb \Omega\rb $ the error term improves from $\mathcal{O}\lb R^3\rb $ to $\mathcal{O}\lb R^4\rb $.
\end{remark}
The combination of Lemma~\ref{th:approximation_of_pv} and EG2 leads to an infinite family of third-order evolution formulas for $p\lb P\rb $ using $R=\delta c \Delta t$ with $\delta \in [0,1]$. It has the form 
\begin{equation}
	\label{eq:eg2_delta}
	\begin{split}
		   p\lb P\rb  = &\frac{1}{\pi} \int_0^{2\pi} \lb p\lb \mathbf{Q}\lb \theta\rb \rb  -
		u\lb \mathbf{Q}\lb \theta\rb \rb  \cos\lb \theta\rb  - v\lb \mathbf{Q}\lb \theta\rb \rb  \sin\lb \theta\rb  \rb
		{\rm d} \theta \\ -& \frac{1}{3}\lb\frac{4}{2\pi}\int_{0}^{2 \pi} p\lb \mathbf{Q}_{\frac{\delta}{2}c \Delta t}\lb \theta\rb \rb  {\rm d} \theta - \frac{1}{2\pi}\int_{0}^{2 \pi} p\lb \mathbf{Q}_{\delta c \Delta t}\lb \theta\rb \rb  {\rm d} \theta \rb + \mathcal{O} \lb \Delta t^3\rb . 
	\end{split}
\end{equation}  
The combination of the formula for updating $p\lb P\rb $ and the original EG2 formulas for $u\lb P\rb $ and $v\lb P\rb $ leads to new evolution operators, which we refer to as EG2$_\delta$. Note that for $\delta = 0$,  EG2$_{0}$ reduces to EG2.

Furthermore, the exact equation in \eqref{eq:exactEvoBic} for $u\lb P\rb $ contains the term $\frac{1}{2}u\lb P'\rb $. This contribution is lost when using the exact integration step \eqref{eq:integration}, which suggests two possible ways of approximation. The first option is to replace the term $\frac{1}{2}u\lb P'\rb $ in the exact formula by an approximation obtained from Lemma~\ref{th:approximation_of_pv}. The second option is to replace the term $u\lb P'\rb $ in the integration step \eqref{eq:integration} by such an approximation. The same applies to the update of $v\lb P\rb $. The formulas obtained by the first option are
\begin{equation} \label{eq:EG2nu}
\resizebox{0.93\linewidth}{!}{$
	\begin{aligned}
		u\lb P\rb  & =\frac{1}{\pi} \int_0^{2\pi} \lb -p\lb \mathbf{Q}\lb \theta\rb \rb 
		\cos\lb \theta\rb  + u\lb \mathbf{Q}\lb \theta\rb \rb  \lb 2 \cos^2\lb \theta\rb  -
		\frac{1}{2} \rb + 2 v\lb \mathbf{Q}\lb \theta\rb \rb  \sin\lb \theta\rb \cos\lb \theta\rb 
		\rb {\rm d} \theta \\
		& -  \lb u\lb P'\rb -\frac{1}{3}\lb\frac{4}{2\pi}\int_{0}^{2 \pi} u\lb \mathbf{Q}_{\frac{\nu}{2}c \Delta t}\lb \theta\rb \rb  {\rm d} \theta - \frac{1}{2\pi}\int_{0}^{2 \pi} u\lb \mathbf{Q}_{\nu c \Delta t}\lb \theta\rb \rb  {\rm d} \theta \rb\rb+ \mathcal{O} \lb \Delta \textcolor{red}{t}^3\rb , \\[1.ex]
		v\lb P\rb  & =\frac{1}{\pi} \int_0^{2\pi} \lb -p\lb \mathbf{Q}\lb \theta\rb \rb 
		\sin\lb \theta\rb  + 2 u\lb \mathbf{Q}\lb \theta\rb \rb  \sin\lb \theta\rb  \cos\lb \theta\rb  +
		v\lb \mathbf{Q}\lb \theta\rb \rb  \lb 2 \sin^2 \lb \theta\rb  - \frac{1}{2}\rb
		\rb {\rm d} \theta \\
		& -  \lb v\lb P'\rb -\frac{1}{3}\lb\frac{4}{2\pi}\int_{0}^{2 \pi} v\lb \mathbf{Q}_{\frac{\nu}{2}c \Delta t}\lb \theta\rb \rb  {\rm d} \theta - \frac{1}{2\pi}\int_{0}^{2 \pi} v\lb \mathbf{Q}_{\nu c \Delta t}\lb \theta\rb \rb  {\rm d} \theta \rb\rb+ \mathcal{O} \lb \Delta {t}^3\rb . \\
	\end{aligned}
    $}
\end{equation}
Using the second option yields
\begin{equation*}
\resizebox{0.93\linewidth}{!}{$
	\begin{aligned}
		u\lb P\rb  & =\frac{1}{\pi} \int_0^{2\pi} \lb -p\lb \mathbf{Q}\lb \theta\rb \rb 
		\cos\lb \theta\rb  + u\lb \mathbf{Q}\lb \theta\rb \rb  \lb 2 \cos^2\lb \theta\rb  -
		\frac{1}{2} \rb + 2 v\lb \mathbf{Q}\lb \theta\rb \rb  \sin\lb \theta\rb \cos\lb \theta\rb 
		\rb {\rm d} \theta \\
		& +  \lb u\lb P'\rb -\frac{1}{3}\lb\frac{4}{2\pi}\int_{0}^{2 \pi} u\lb \mathbf{Q}_{\frac{\nu}{2}c \Delta t}\lb \theta\rb \rb  {\rm d} \theta - \frac{1}{2\pi}\int_{0}^{2 \pi} u\lb \mathbf{Q}_{\nu c \Delta t}\lb \theta\rb \rb  {\rm d} \theta \rb\rb+ \mathcal{O} \lb \Delta {t}^3\rb , \\[1.ex]
		v\lb P\rb  & =\frac{1}{\pi} \int_0^{2\pi} \lb -p\lb \mathbf{Q}\lb \theta\rb \rb 
		\sin\lb \theta\rb  + 2 u\lb \mathbf{Q}\lb \theta\rb \rb  \sin\lb \theta\rb  \cos\lb \theta\rb  +
		v\lb \mathbf{Q}\lb \theta\rb \rb  \lb 2 \sin^2 \lb \theta\rb  - \frac{1}{2}\rb
		\rb {\rm d} \theta \\
		& +  \lb v\lb P'\rb -\frac{1}{3}\lb\frac{4}{2\pi}\int_{0}^{2 \pi} v\lb \mathbf{Q}_{\frac{\nu}{2}c \Delta t}\lb \theta\rb \rb  {\rm d} \theta - \frac{1}{2\pi}\int_{0}^{2 \pi} v\lb \mathbf{Q}_{\nu c \Delta t}\lb \theta\rb \rb  {\rm d} \theta \rb\rb+ \mathcal{O} \lb \Delta {t}^3\rb . \\
	\end{aligned}
    $}
\end{equation*}
For both options, we have $\nu \in [0,1]$. However, the formulas in \eqref{eq:EG2nu} yield better numerical results. We refer to the evolution that uses \eqref{eq:eg2_delta} for updating $p\lb P\rb $ and \eqref{eq:EG2nu} for updating $u\lb P\rb $ and $v\lb P\rb $ as EG2$_{\delta,\nu}$. In Section~\ref{sec:stability}, we will show that both new evolution operators, EG2$_\delta$ and EG2$_{\delta,\nu}$, increase the admissible time step of the AF methods for well-chosen values of $\delta$ and $\nu$. However, they also increase the computational cost, since additional integrals are involved. To reduce computational costs, Lemma \ref{lem} is applied to all integrals. This ensures that the numerical integration is third-order accurate. Numerical experiments have shown that it is necessary to choose $n = 8$ instead of $4$, which is used for the $\mathrm{EG}^{\text{quad}}$ operator. Consequently, all integrals are approximated using the values at the points $\mathbf{Q}\lb 0\rb $, $\mathbf{Q}\lb\frac{\pi}{4}\rb$, $\mathbf{Q}\lb\frac{\pi}{2}\rb$, $\mathbf{Q}\lb\frac{3\pi}{4}\rb$, $\mathbf{Q}\lb \pi\rb $, $\mathbf{Q}\lb\frac{5\pi}{4}\rb$, $\mathbf{Q}\lb\frac{3\pi}{2}\rb$, and $\mathbf{Q}\lb\frac{7\pi}{4}\rb$. We denote by $\widehat{\mathrm{EG2}}_{\delta}$ the ${\mathrm{EG2}}_{\delta}$ operator based on the circle approximations. Analogously, we refer to $\widehat{\mathrm{EG2}}_{\delta,\nu}$.

We will now study the stability and accuracy of the new AF methods.

\subsection{Investigation of linear stability}\label{sec:stability}
To analyse the linear stability of the AF method with different evolution operators for updating the point values, we consider discretisations on a quadratic domain with double periodic boundary conditions. The domain is discretised using a grid with $m \times m$ grid cells. The linear method can be written in the form 
$$
U^{n+1} = B\lb \Delta t\rb  U^n,
$$
where the vector $U^n$ contains all degrees of freedom of the two-dimensional Cartesian grid simulation and the matrix $B$ describes the evolution of the degrees of freedom during one time step with size $\Delta t$ using the AF method. We can investigate the stability of the method under different time step restrictions. A necessary condition for stability is that all eigenvalues of $B\lb \Delta t\rb $ lie within the unit circle. We have previously found that the AF method with EG2 is stable under the condition $\mbox{CFL} \le 0.279$. As indicated in Figure~\ref{fig2}, the new approximate evolution operator EG$^\text{quad}$ remains also stable  for $\mbox{CFL} \le 0.279$. For a method employing EG2$_\delta$, $\delta = 0, 0.1, \dots, 1$,  we observe that, as the radius of the circles used for the approximation of the point values increases, the maximum admissible CFL number also increases. The results are shown in Table \ref{Tab:CFL_values_EG2i}. Up to approximations, using $ R = 0.5 c \Delta t$ the CFL number is comparable to the one that we get using the original EG2 operator.  For $\delta \ge 0.7$, the CFL number increases significantly to $0.419$, \ignore{$0.4189$}, which represents a considerable improvement over the previous evolution operators. For the calculations, we use a $20\times 20$ grid. Numerical simulations on finer grids confirm stability under the same time step restriction. 
\begin{table}[!ht]
\caption{Maximum admissible CFL numbers for the stable method with EG2$_\delta$, $\delta = 0, 0.1, \dots, 1$.}
\vspace*{0.1cm}
\sisetup{round-mode=places,round-precision=3}
\hspace*{+0.2cm}
\begin{minipage}{0.95\textwidth}
\centering
\begin{tabular}{l *{11}{S[table-column-width=0.8cm,table-text-alignment=center]}}
\toprule
{$\delta$}& {{0.0}} & {0.1} & {0.2} & {0.3} & {0.4} & {0.5} & {0.6} & {0.7} & {0.8} & {0.9} & {1.0} \\
\midrule
CFL & \num{0.2791} & \num{0.2865} & \num{0.2859} & \num{0.2899} & \num{0.2973} & \num{0.3091} & \num{0.3316} & \num{0.4189} & \num{0.4189} & \num{0.4189} & \num{0.4189} \\
\bottomrule
\end{tabular}
\label{Tab:CFL_values_EG2i}
\end{minipage}
\end{table}

Calculations using the exact formulas for $u\lb P\rb $ and $v\lb P\rb $ indicate that, for EG2$_\delta$ with $\delta > 0.7$, the admissible CFL number is determined by these formulas; see Table \ref{Tab:CFL_values_EG2i_exact}. 
\begin{table}[!ht]
\caption{Maximum admissible CFL numbers for the stable method with EG2$_\delta$, $\delta = 0, 0.7, 0.8, 0.9,1 $ for the evolution of $p\lb P\rb $ and exact evolution for $u\lb P\rb $ and $v\lb P\rb $.}
\vspace*{0.1cm}
\sisetup{round-mode=places,round-precision=3}
\hspace*{+0.2cm}
\begin{minipage}{0.95\textwidth}
\centering
\begin{tabular}{l *{6}{S[table-column-width=0.8cm,table-text-alignment=center]}}
\toprule
{$\delta$}& {0} &{0.7} & {0.8} & {0.9} & {1.0} \\
\midrule
CFL & \num{0.2791}& \num{0.4314} & \num{0.4531} & \num{0.4619} & \num{0.4712} \\
\bottomrule
\end{tabular}
\label{Tab:CFL_values_EG2i_exact}
\end{minipage}
\end{table}

These results suggest that further improvement of the stability of the AF methods requires adapting the approximate evolution equations for $u\lb P\rb $ and $v\lb P\rb $, as realised in EG2$_{\delta,\nu}$. Its maximum admissible CFL number increases further, for well-chosen $\delta$ and $\nu$, to 0.44. Table~\ref{Tab:CFL_values_EG2dn} lists the maximum admissible CFL numbers for different combinations of $\nu$ and $\delta$.
\begin{table}[ht]
\centering
\caption{Maximum admissible CFL numbers for the stable method EG2$_{\delta,\nu}$ for different combinations of $\delta$ and $\nu$.}
\label{Tab:CFL_values_EG2dn}
\begin{tabular}{c *{4}{S[table-format=1.3, table-column-width=1.5cm]}}
\toprule
\multicolumn{5}{c}{\textbf{CFL numbers for different values of $\delta$ and $\nu$}} \\
\cmidrule(lr){2-5}
{} & {$\delta = 0.7$} & {$\delta = 0.8$} & {$\delta = 0.9$} & {$\delta = 1.0$} \\
\midrule
$\nu = 0.1$  & 0.434 & 0.371  & 0.434 & 0.434 \\
$\nu = 0.2$ & 0.426 & 0.440 & 0.437 & 0.433 \\
$\nu = 0.3$ & 0.417  & 0.438 & 0.435 & 0.431 \\
$\nu = 0.4$ & 0.411  & 0.435 & 0.434 & 0.430 \\
$\nu = 0.5$ & 0.403 & 0.421  & 0.431 & 0.429 \\
\bottomrule
\end{tabular}
\end{table}

 In Figure~\ref{fig:EG2_del_eigenvalues} we show the eigenvalues of the matrix $B$ for the AF method using different evolution operators with CFL = 0.44. The second plot displays the results obtained with the original EG2 operator, where the real part of many eigenvalues lies outside the unit circle. After modifying the formulas for $p\lb P\rb $, the distribution of the eigenvalues changes significantly (see fourth plot) and becomes more comparable to that of the method using the exact evolution operator (see first plot). Further adjustments to the formulas for $u\lb P\rb $ and $v\lb P\rb $ alter the eigenvalue distribution such that all eigenvalues remain inside the unit circle up to CFL = 0.44 (see fifth plot). As documented in Figure~\ref{fig:EG2_del_eigenvalues} (third plot), and in Figure~\ref{fig2}, the AF method with EG$^{\text{quad}}$ operator is unstable for CFL = 0.44, but remains stable for CFL=0.279. 

\begin{figure}[!htb]
	\centering

			\includegraphics[width=0.2\textwidth]{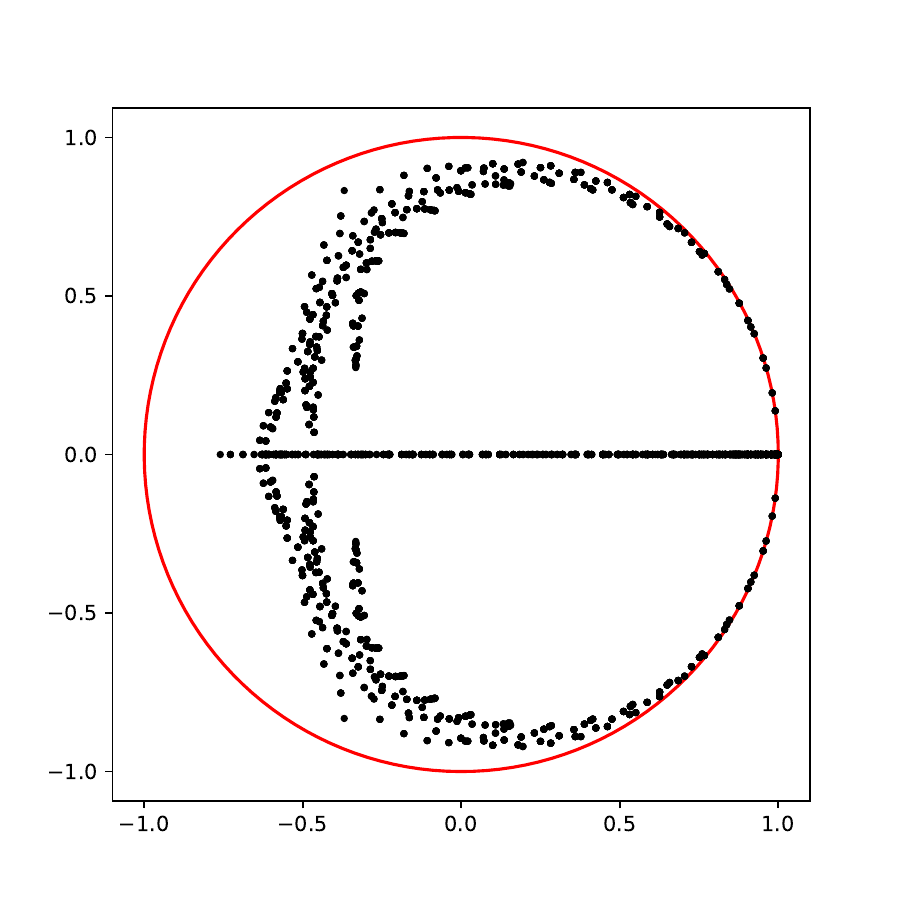}%
			\includegraphics[width=0.2\textwidth]{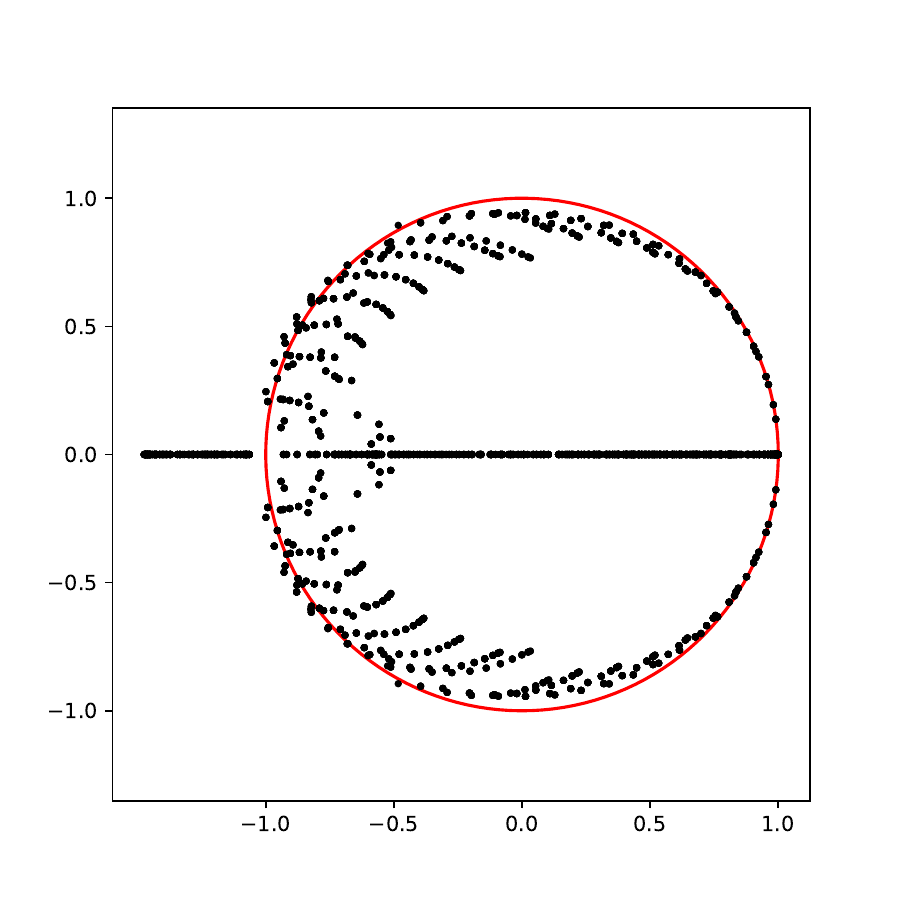}%
            \includegraphics[width=0.2\textwidth]{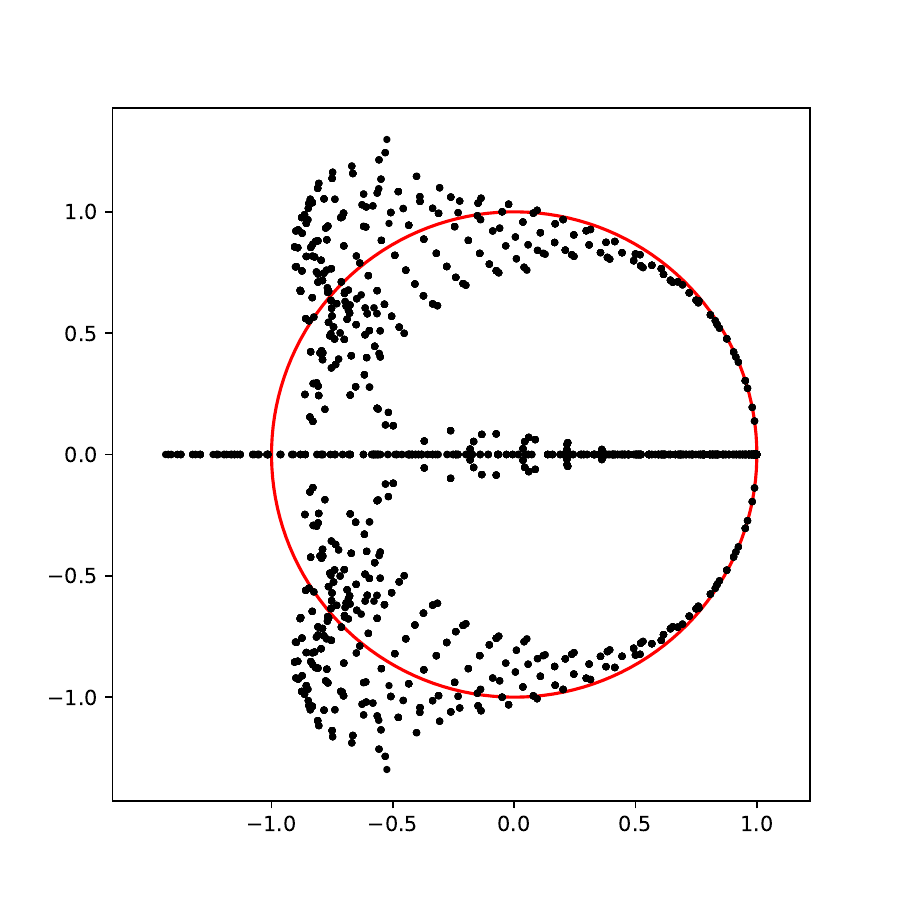}%
			\includegraphics[width=0.2\textwidth]{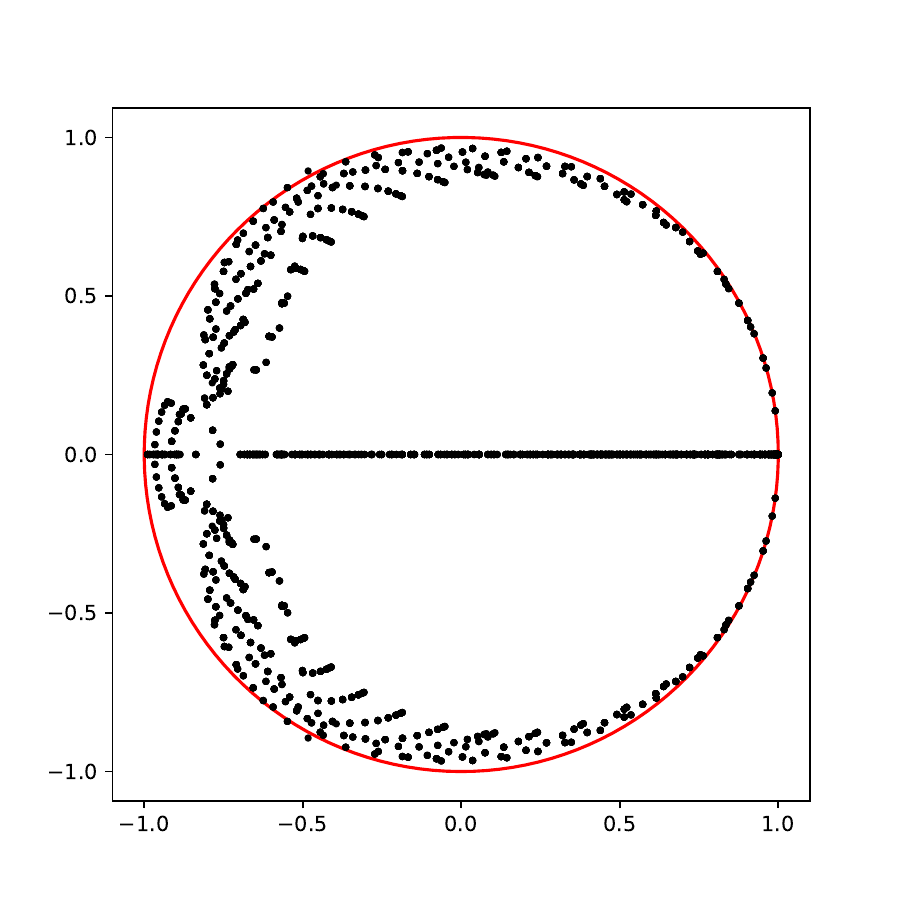}%
			\includegraphics[width=0.2\textwidth]{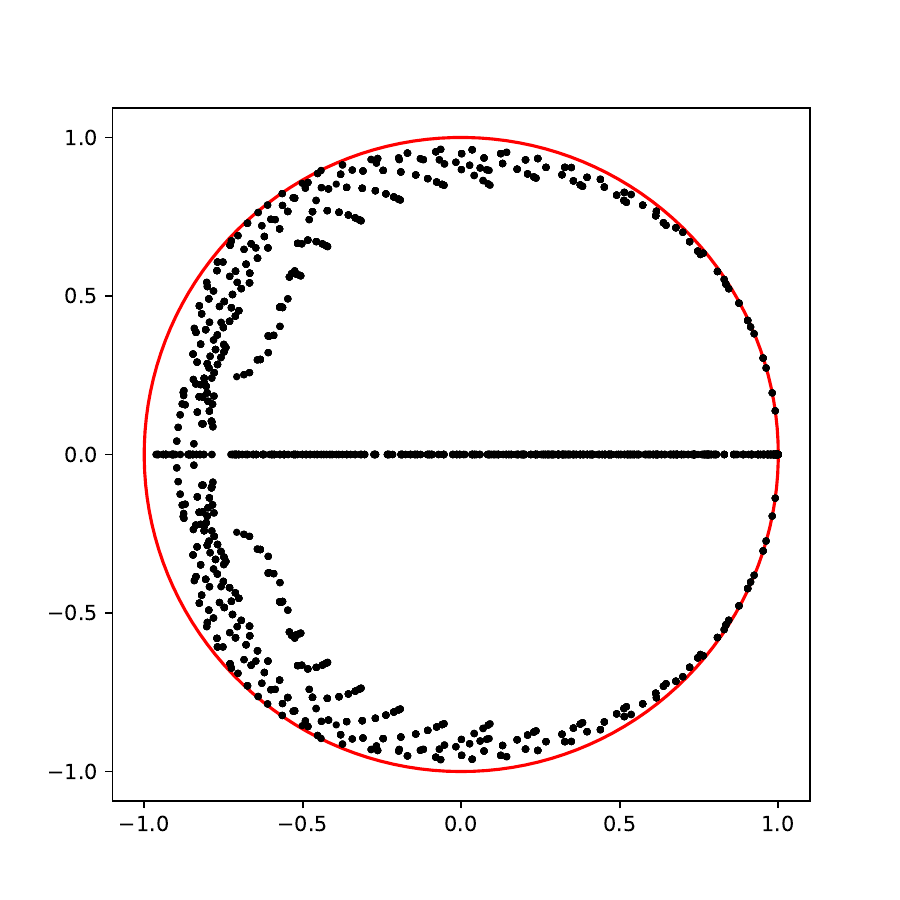}%

\caption{\label{fig:EG2_del_eigenvalues}%
		Eigenvalues of the matrix $B$ for the AF method using (from left to right):
		exact evolution, EG2, $\mathrm{EG}^{\text{quad}}$,  EG2$_{0.7}$ and EG2$_{0.8,0.2}$ with CFL $=0.44$.
		A $20\times 20$ grid with periodic boundary conditions was used.}

\end{figure}
			
\begin{figure}[!htb]
	\centering
\includegraphics[width=0.2\textwidth]{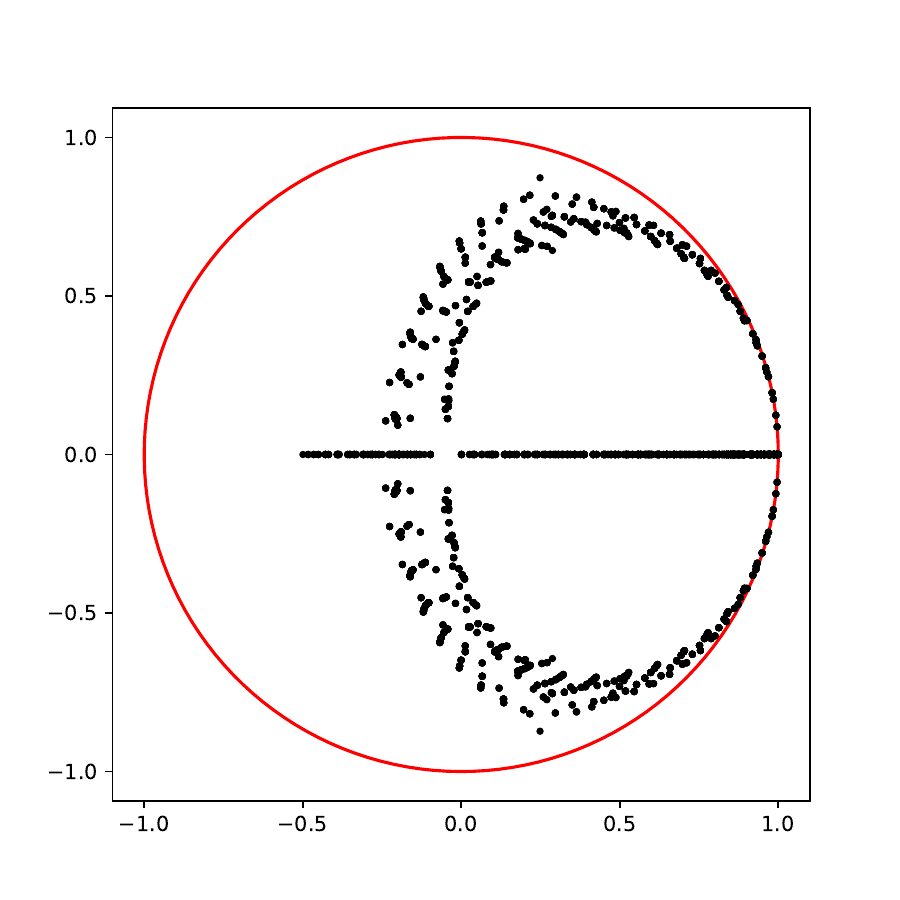}
\includegraphics[width=0.2\textwidth]{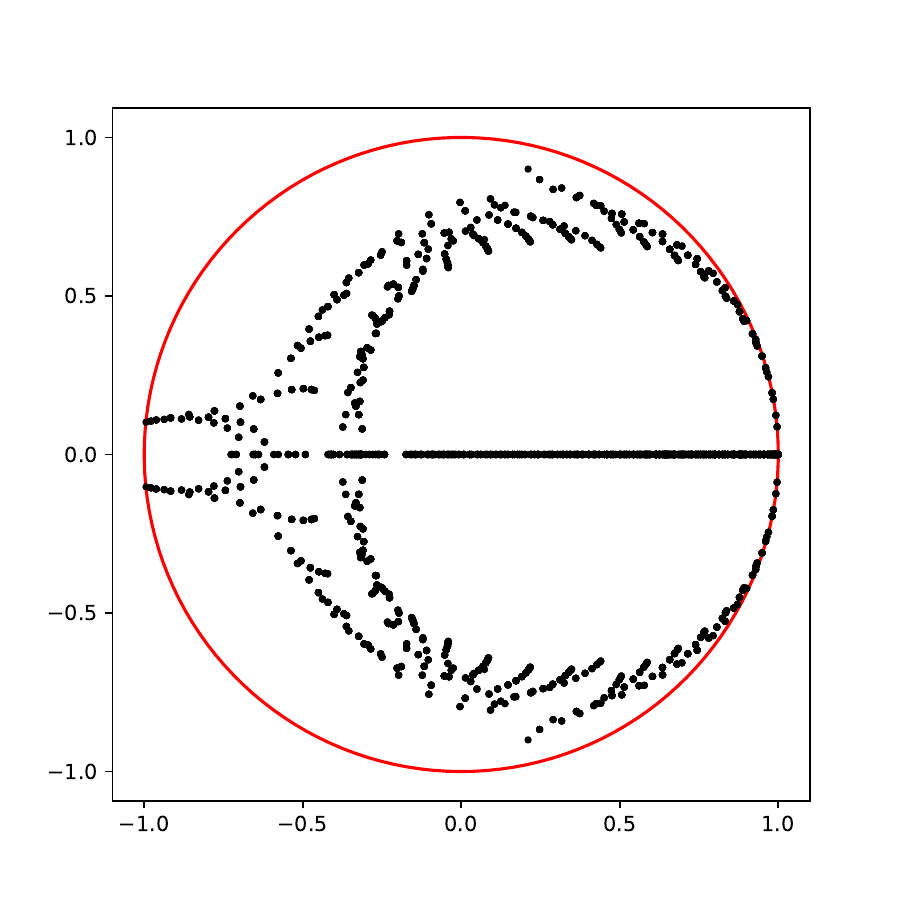}
\includegraphics[width=0.2\textwidth]{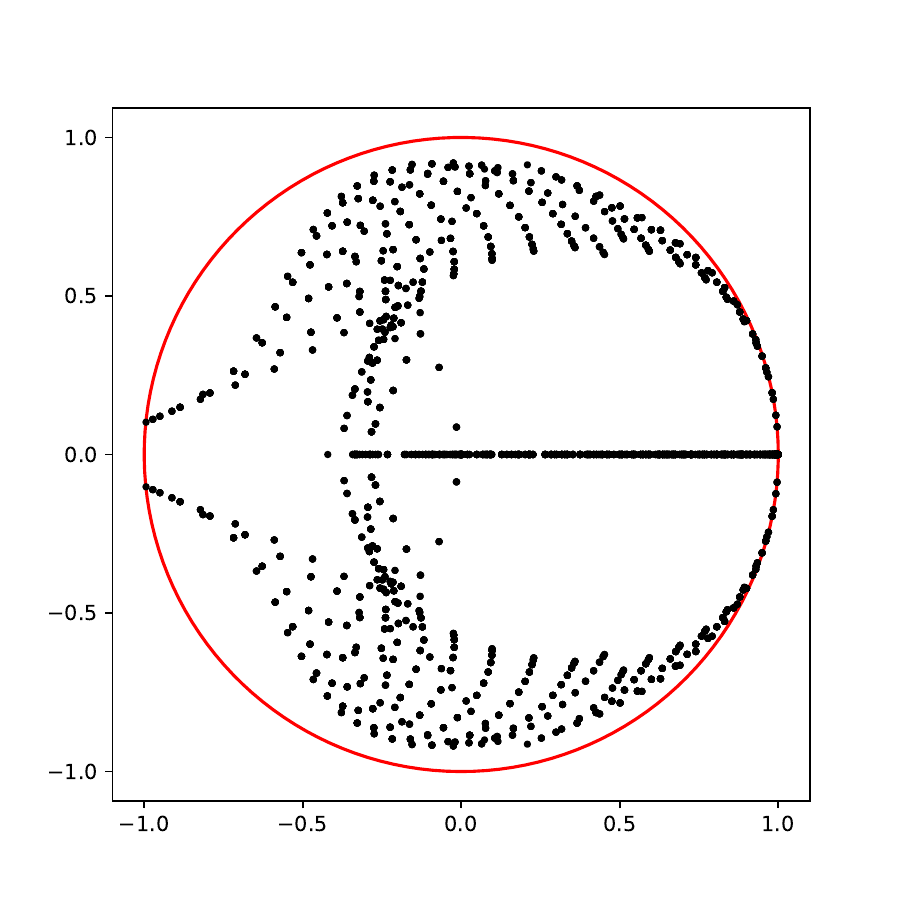}

\caption{\label{fig2}
Eigenvalues of the matrix $B$ for the AF method using exact evolution, EG2 and $\mathrm{EG}^{\text{quad}}$ with CFL $=0.279$. A $20\times 20$ grid with periodic boundary conditions was used.}
\end{figure}

In the Appendix, we present a further simplification of the AF methods with $\widehat{\mathrm{EG2}}_{1.0}$ and $\widehat{\mathrm{EG2}}_{1.0,0.2}$  operators using numerical quadrature for the integration along the base of the bicharacteristic cones.  Our numerical experiments confirm that these AF methods are stable up to a CFL number equal to 0.4.

\subsection{Investigation of accuracy} \label{sec:accuracy}

In this section, we present the accuracy results for the different new evolution operators and compare them with those obtained using the exact evolution operator. To this end, we consider Examples~\ref{ex:1} and~\ref{ex:2}. 

Tables~\ref{Tab:convergenceProblem1_eg2eg11_t01} and \ref{Tab:convergenceProblem1_eg2eg11_t1} show the results for Example~\ref{ex:1}. The computations are performed on grids with $64\times 64$, $128 \times 128$, and $256 \times 256$ grid cells. We measure the error of $p$ at $t = 0.1$ and $t = 1$ for the AF method using exact evolution, $\mathrm{EG}^{\text{quad}}$, EG2$_{0.7}$, and EG2$_{0.8,0.2}$. For each method, the time steps are chosen close to their respective stability limit, i.e., CFL$=0.5, 0.276, 0.418$, and $0.439$ were used for the AF method with the exact and approximate evolution operators   $\mathrm{EG}^{\text{quad}}$, EG2$_{0.7}$, and EG2$_{0.8,0.2}$, respectively.

All methods exhibit third-order accuracy. As expected, the method with the exact evolution yields the smallest errors for all test cases and time levels. All four methods yield errors of the same order of magnitude. Among the methods employing approximate evolution operators, EG2$_{0.8,0.2}$ produces the smallest errors, although the results are very similar to those obtained with $\mathrm{EG}^{\text{quad}}$ and EG2$_{0.7}$. In contrast, the errors obtained with EG2 are systematically larger than those of the other approximate methods.



Similar results are obtained for Example~\ref{ex:2}. We now measure the errors in $u$ and $v$ at times $t = 0.1$ and $t = 1$. The results show that the methods using approximate evolution operators yield nearly identical errors. Even though this is not directly apparent from the numerical values, it was shown in \cite[Table 5]{article:CHL2024} that the convergence order of the AF method with exact evolution is not third order at all times.

\ignore{  In the Tables \ref{Tab:convergenceProblem1_t01_eg2delta} - \ref{Tab:convergenceProblem2_t1_eg2delta} we investigate for the same test cases the performance of some EG2$_\delta$ with $\delta = 0.7, 0.8, 0.9,1$.  For all calculations, CFL = 0.4189 was used.  Again, we observe third-order accuracy. As expected, we observe that the method using the smallest $\delta$ produces the lowest error, and that the larger $\delta$ is, the larger the error. However, the difference is minimal. If it is necessary to keep the calculation costs as low as possible, then EG2$_1$ is the better choice, as the outer circle used to approximate the point value appears once again in the evolution formula.}

\begin{table}[!ht]
\centering
\caption{Errors measured in the $L_1$-norm and EOC for Example~\ref{ex:1} using exact evolution, $\mathrm{EG}^{\text{quad}}$, EG2$_{0.7}$, and EG2$_{0.8,0.2}$ at $t = 0.1$.}
\label{Tab:convergenceProblem1_eg2eg11_t01}
\vspace*{0.15cm}
\sisetup{scientific-notation = true,round-mode = places}
\renewcommand{\arraystretch}{0.9}
\setlength{\tabcolsep}{3pt}
\resizebox{\linewidth}{!}{$
\begin{tabular}{c*{4}{S[round-precision=2, table-format=1.2e-2]}*{4}{S[round-precision=3, table-format=1.4]}}
\toprule
Res. & \multicolumn{4}{c}{Error in $p$} & \multicolumn{4}{c}{EOC} \\
\cmidrule(lr){2-5} \cmidrule(lr){6-9}
& {exact} & {$\mathrm{EG}^{\text{quad}}$} & {EG2$_{0.7}$} & {EG2$_{0.8,0.2}$}
& {exact} & {$\mathrm{EG}^{\text{quad}}$} & {EG2$_{0.7}$} & {EG2$_{0.8,0.2}$} \\
\midrule
$64\times64$  & \num{1.6042545843454000e-05} & \num{2.6170254313972369e-05}
	& \num{2.1364407395546843e-05} & \num{2.4309962543038923e-05}
	& {---} & {---} & {---} & {---} \\
$128\times128$ & \num{2.0386219828327593e-06} & \num{3.3640431718672253e-06}
	& \num{3.1870953920483475e-06} & \num{3.0155002681889683e-06}
	& 2.9762 & 2.9597 & 2.7449 & 3.0111 \\
$256\times256$ & \num{2.5426897907752349e-07} & \num{4.2589073555490836e-07}
	& \num{3.9718607411956407e-07} & \num{3.7413924157942913e-07}
	& 3.0032 & 2.9816 & 3.0044 & 3.0108 \\
\bottomrule
\end{tabular}
$}
\end{table}

\begin{table}[!ht]
\centering
\caption{Errors measured in the $L_1$-norm and EOC for Example~\ref{ex:1} using exact evolution, EG2, EG2$_{0.7}$, and EG2$_{0.8,0.2}$ at $t = 1$.}
\label{Tab:convergenceProblem1_eg2eg11_t1}
\vspace*{0.15cm}
\sisetup{scientific-notation = true,round-mode = places}
\renewcommand{\arraystretch}{0.9}
\setlength{\tabcolsep}{3pt}
\resizebox{\linewidth}{!}{$
\begin{tabular}{c*{4}{S[round-precision=2, table-format=1.2e-2]}*{4}{S[round-precision=3, table-format=1.4]}}
\toprule
Res. & \multicolumn{4}{c}{Error in $p$} & \multicolumn{4}{c}{EOC} \\
\cmidrule(lr){2-5} \cmidrule(lr){6-9}
     & {exact} & {$\mathrm{EG}^{\text{quad}}$} & {EG2$_{0.7}$} & {EG2$_{0.8,0.2}$}
	 & {exact} & {$\mathrm{EG}^{\text{quad}}$} & {EG2$_{0.7}$} & {EG2$_{0.8,0.2}$} \\
\midrule
$64\times64$  & \num{1.9950350026772478e-04} & \num{2.9903225074379819e-04}
    & \num{3.0539450329744686e-04} & \num{2.9520821618622261e-04}
	& {---} & {---} & {---} & {---} \\
$128\times128$ & \num{2.5059362405697071e-05} & \num{3.7551963945446894e-05}
    & \num{3.8032190070219544e-05} & \num{3.7233539642388949e-05}
	& 2.9930 & 2.9933 & 2.9900 & 2.9871 \\
$256\times256$ & \num{3.1362045778548605e-06} & \num{4.7018103934318804e-06}
	& \num{4.7675390864590091e-06} & \num{4.6379581765285962e-06}
	& 2.9983 & 2.9976 & 2.9968 & 3.0050 \\
\bottomrule
\end{tabular}
$}
\end{table}

\begin{table}[!ht]
\centering
\caption{Errors measured in the $L_1$-norm and EOC for Example~\ref{ex:2} using exact evolution, EG2, EG2$_{0.7}$, and EG2$_{0.8,0.2}$  at $t = 0.1$.}
\label{Tab:convergenceProblem2_eg2eg11_t01}
\vspace*{0.15cm}
\sisetup{scientific-notation = true,round-mode = places}
\renewcommand{\arraystretch}{0.9}
\setlength{\tabcolsep}{3pt}
\resizebox{\linewidth}{!}{$
\begin{tabular}{c*{4}{S[round-precision=2, table-format=1.2e-2]}*{4}{S[round-precision=3, table-format=1.4]}}
\toprule
Res. & \multicolumn{4}{c}{Error in $u,v$} & \multicolumn{4}{c}{EOC} \\
\cmidrule(lr){2-5} \cmidrule(lr){6-9}
	 & {exact} & {$\mathrm{EG}^{\text{quad}}$} & {EG2$_{0.7}$} & {EG2$_{0.8,0.2}$}
	 & {exact} & {$\mathrm{EG}^{\text{quad}}$} & {EG2$_{0.7}$} & {EG2$_{0.8,0.2}$} \\
\midrule
$64\times64$   & \num{1.2652061957520213e-05} & \num{1.6456108465489550e-05}
	 & \num{1.9185690310434239e-05} & \num{1.9425967000877102e-05}
	 & {---} & {---} & {---} & {---} \\
$128\times128$  & \num{1.6021279860692110e-06} & \num{2.0688490140222487e-06}
	 & \num{2.3825592164870894e-06} & \num{2.4011283886559404e-06}
	 & 2.9813 & 2.9917 & 3.0094 & 3.0162 \\
$256\times256$  & \num{1.9969677502133022e-07} & \num{2.5935186911268085e-07}
	 & \num{2.9667984599859425e-07} & \num{2.9868476500225845e-07}
	 & 3.0041 & 2.9958 & 3.0055 & 3.0070 \\
\bottomrule
\end{tabular}
$}
\end{table}

\begin{table}[!ht]
\centering
\caption{Errors measured in the $L_1$-norm and EOC for Example~\ref{ex:2} using exact evolution, EG2, EG2$_{0.7}$, and EG2$_{0.8,0.2}$ at $t = 1$.}
\label{Tab:convergenceProblem2_eg2eg11_t1}
\vspace*{0.15cm}
\sisetup{scientific-notation = true,round-mode = places}
\renewcommand{\arraystretch}{0.9}
\setlength{\tabcolsep}{3pt}
\resizebox{\linewidth}{!}{$
\begin{tabular}{c*{4}{S[round-precision=2, table-format=1.2e-2]}*{4}{S[round-precision=3, table-format=1.4]}}
\toprule
Res. & \multicolumn{4}{c}{Error in $u,v$} & \multicolumn{4}{c}{EOC} \\
\cmidrule(lr){2-5} \cmidrule(lr){6-9}
	 & {exact} & {$\mathrm{EG}^{\text{quad}}$} & {EG2$_{0.7}$} & {EG2$_{0.8,0.2}$}
	 & {exact} & {$\mathrm{EG}^{\text{quad}}$} & {EG2$_{0.7}$} & {EG2$_{0.8,0.2}$} \\
\midrule
$64\times64$   & \num{1.5719665415403451e-04} & \num{2.3656062089539930e-04}
	 & \num{2.4068885409344508e-04} & \num{2.3262039501690976e-04}
	 & {---} & {---} & {---} & {---} \\
$128\times128$  & \num{1.9697610150159046e-05} & \num{2.9574579872342152e-05}
	 & \num{2.9902635357792019e-05} & \num{2.9260079918927454e-05}
	 & 2.9965 & 2.9998 & 3.0088 & 2.9910 \\
$256\times256$  & \num{2.4636731237458242e-06} & \num{3.6969280856873179e-06}
	 & \num{3.7457213637370746e-06} & \num{3.6434105203308627e-06}
	 & 2.9991 & 3.0000 & 2.9970 & 3.0056 \\
\bottomrule
\end{tabular}
$}
\end{table}

\noindent

Convergence studies for the corresponding AF methods that use numerical integration can be found in Appendix \ref{app:accuracy}.
\subsection{Approximation of the Stationary Vortex} \label{sec:vortex}
In this section, we compare the performance of the AF method that uses exact evolution, $\mathrm{EG}^{\text{quad}}$, EG2$_{0.7}$ and  EG2$_{0.8,0.2}$  for the stationary vortex described in Example \ref{ex:3}. Barsukow et al. \cite{article:BHKR2019} showed that the third-order accurate Cartesian grid AF method with exact evolution operator is stationary preserving. In \cite{article:CHL2024} Chudzik et al. showed that the method using EG2 is not stationary preserving.  As expected, none of the AF methods that use the new evolution operator preserves the stationary vortex. The numerical results are shown in Figure \ref{fig:EG2_EG11_Vortex}. The solutions were calculated on grids with $64 \times 64$ and $128 \times 128$ grid cells. Apart from the AF method, which uses exact evolution, all methods produce similar solutions. Furthermore, the solutions are similar to those obtained using EG2 as the evolution operator, as shown in Figure 8 of \cite{article:CHL2024}. We refer the reader to  Appendix \ref{app:vortex} for results obtained by simplified AF methods using numerical quadratures for evolution operators.

\ignore{
\begin{figure}[!htb]
  \centerline{
  \includegraphics[width=0.32\textwidth]{./figs/Vortex/2-64.eps}\hfil  
  \includegraphics[width=0.32\textwidth]{./figs/Vortex/2b-64.eps}\hfil
  \includegraphics[width=0.32\textwidth]{./figs/Vortex/2c-64.eps}}
\caption{  
Approximation of the stationary vortex using a grid with $64\times 64$
cells at time $t=10$.  Active Flux method using EG2 (left), EG2b
(middle) and EG2c (right). } 
\end{figure}
}
\begin{figure}[htb]
\centerline{\includegraphics[width=0.25\textwidth]{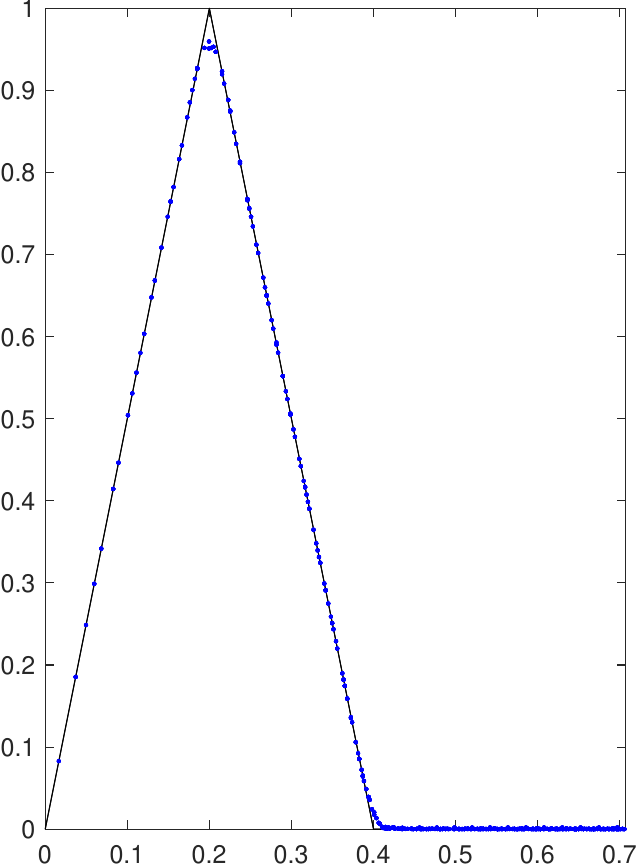}
            \includegraphics[width=0.25\textwidth]{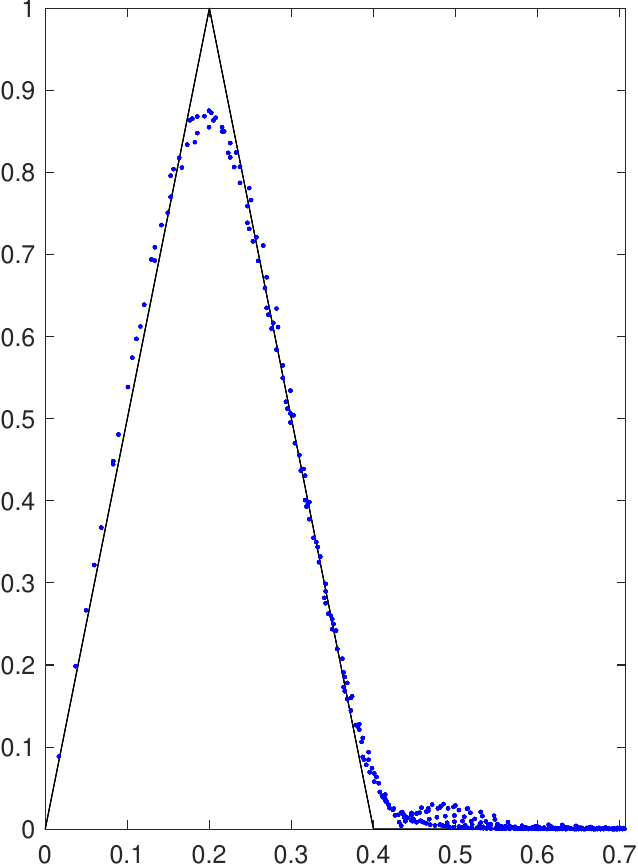}\hfil  
            \includegraphics[width=0.25\textwidth]{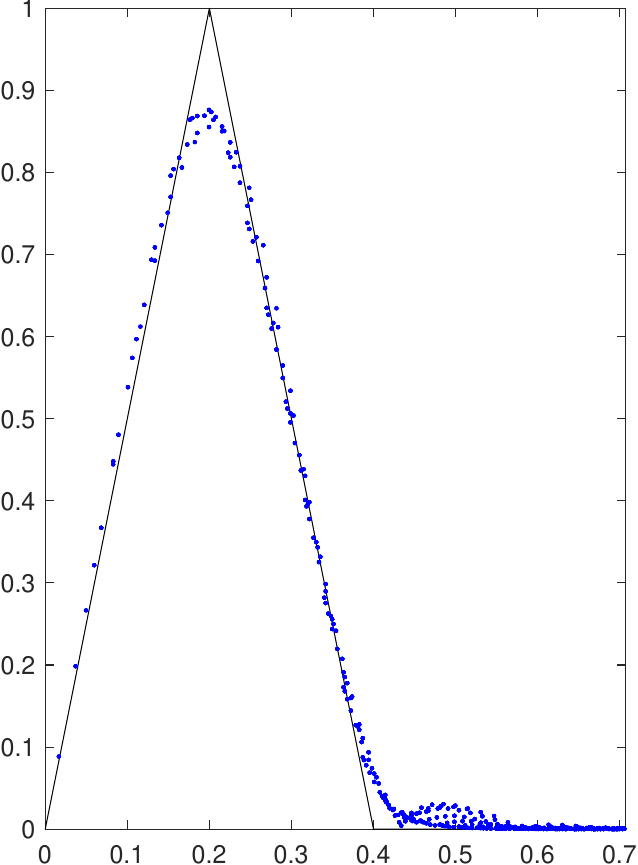}\hfil 
            \includegraphics[width=0.25\textwidth]{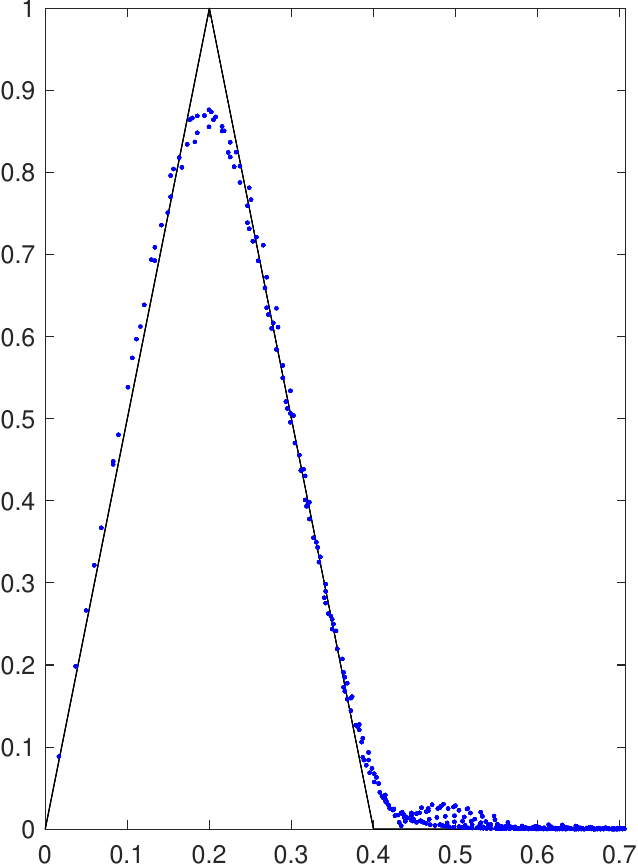}}
\centerline{\includegraphics[width=0.25\textwidth]{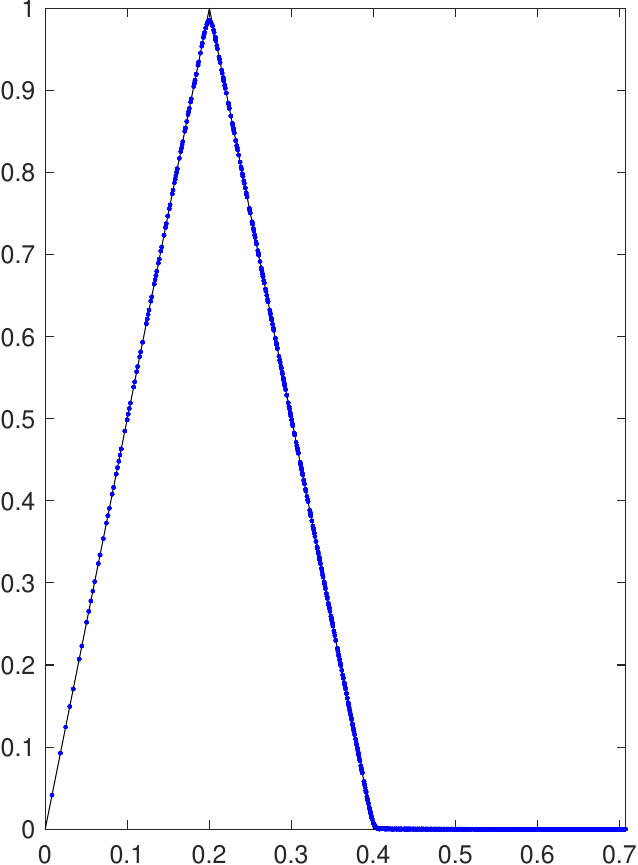}
	        \includegraphics[width=0.25\textwidth]{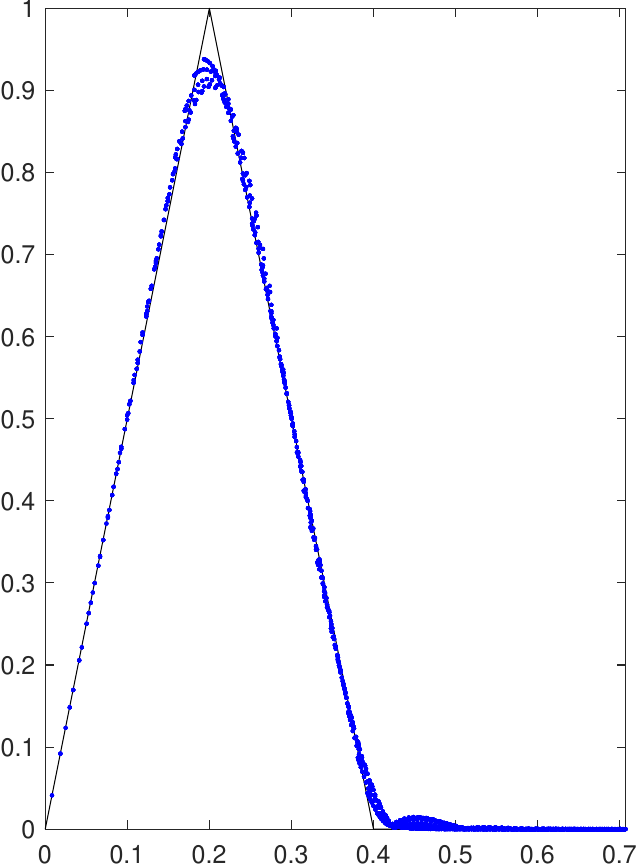}\hfil  
	        \includegraphics[width=0.25\textwidth]{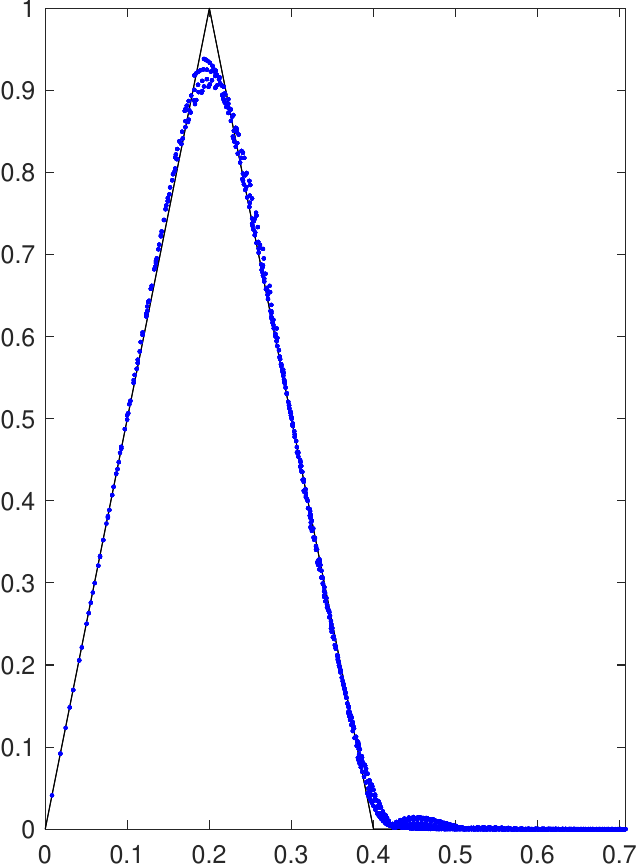}\hfil 
	        \includegraphics[width=0.25\textwidth]{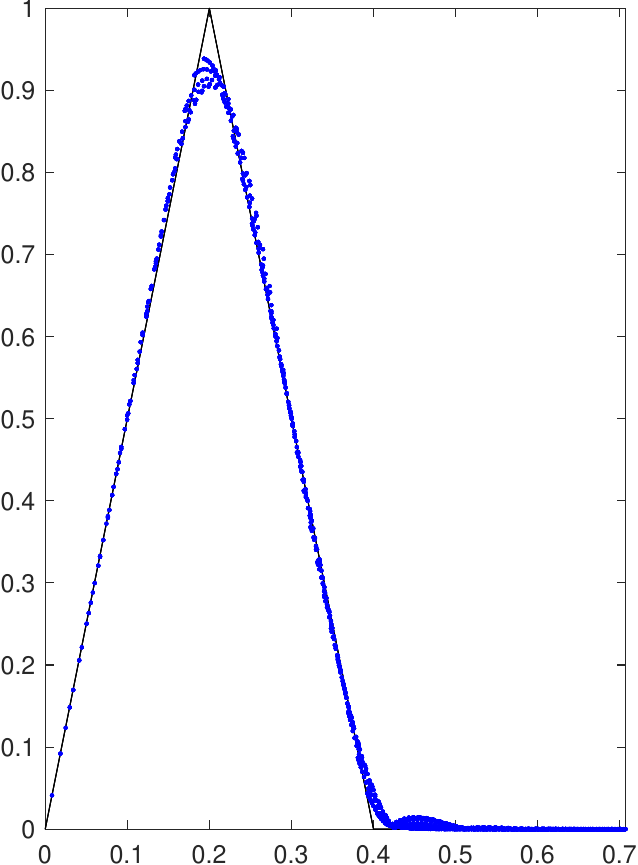}}
\caption{ \label{fig:EG2_EG11_Vortex}Approximation of the stationary vortex using a grid with $64\times64$ (top) and $128\times128$ (bottom) cells at $t=100$ with an AF method using exact evolution(left), $\mathrm{EG}^{\text{quad}}$ (center left), EG2$_{0.7}$ (center right) and EG2$_{0.8,0.2}$ (right).} 
\end{figure}

\subsection{Approximation of discontinuous solutions} \label{sec:disc}
We now investigate the performance of the AF methods with $\mathrm{EG}^{\text{quad}}$, EG2$_{0.7}$ and EG2$_{0.8,0.2}$ evolution operators for the approximation of a discontinuous solution as obtained in Example \ref{ex:4}. The results will be compared with those obtained by the AF method with the  exact evolution operator. All methods lead to accurate approximations even on coarse grids as shown in Figure \ref{fig:EG2cDiscontinuous}. 
\begin{figure}[!htbp]
\centerline{\includegraphics[width=0.3\textwidth]{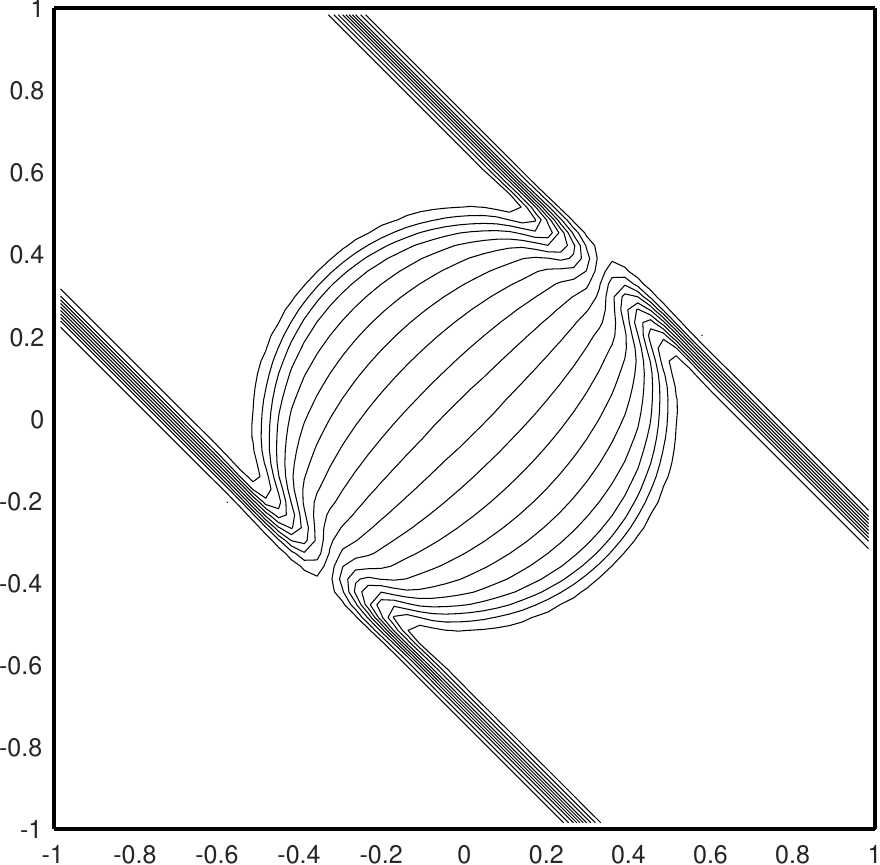}
		    \includegraphics[width=0.3\textwidth]{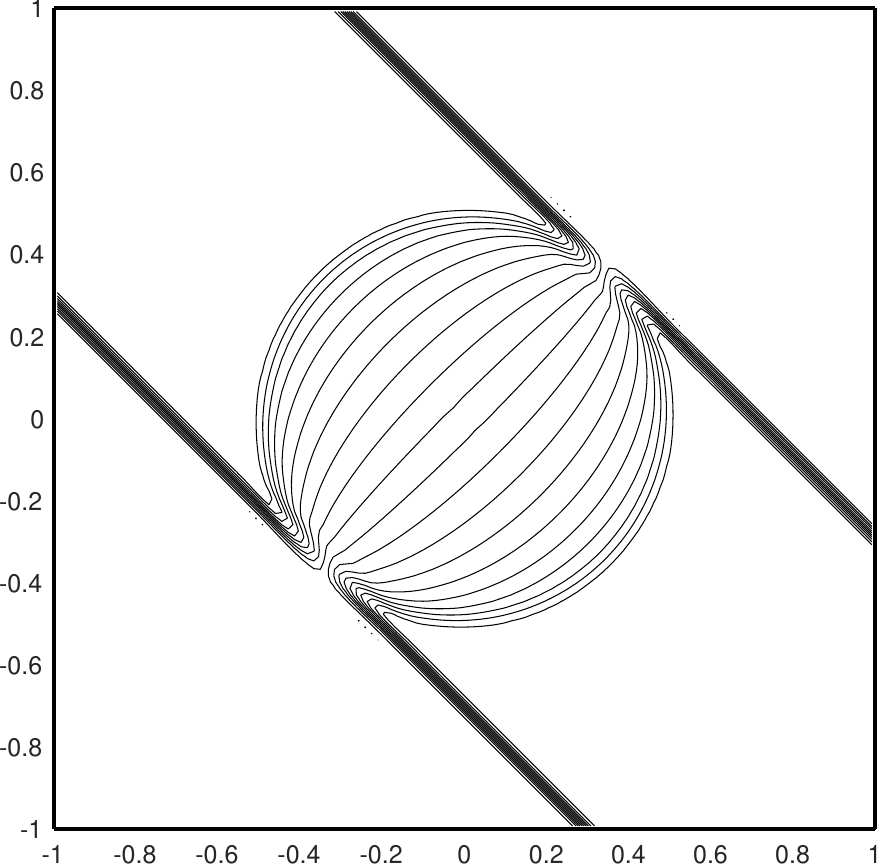}  
		    \includegraphics[width=0.3\textwidth]{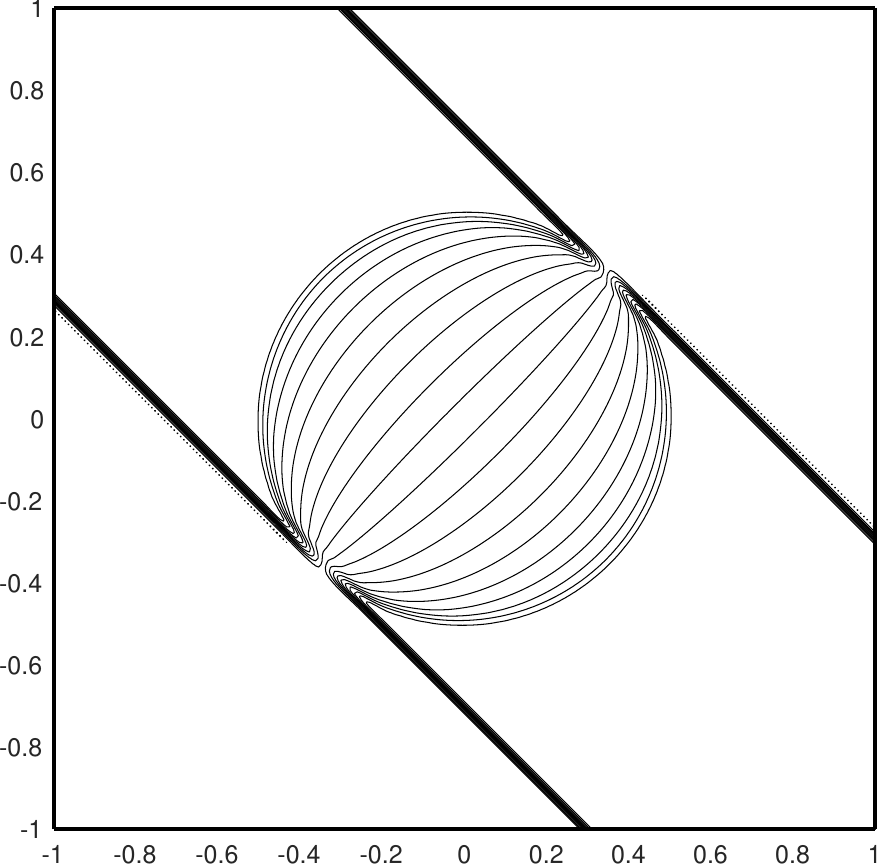}}
\centerline{\includegraphics[width=0.3\textwidth]{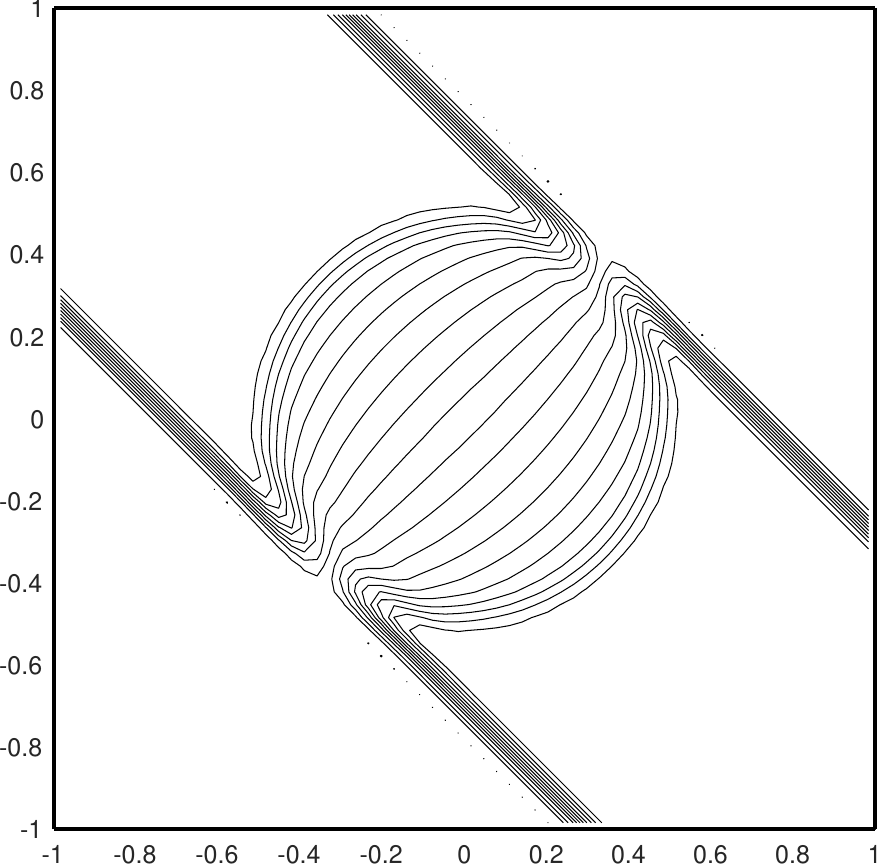} 
            \includegraphics[width=0.3\textwidth]{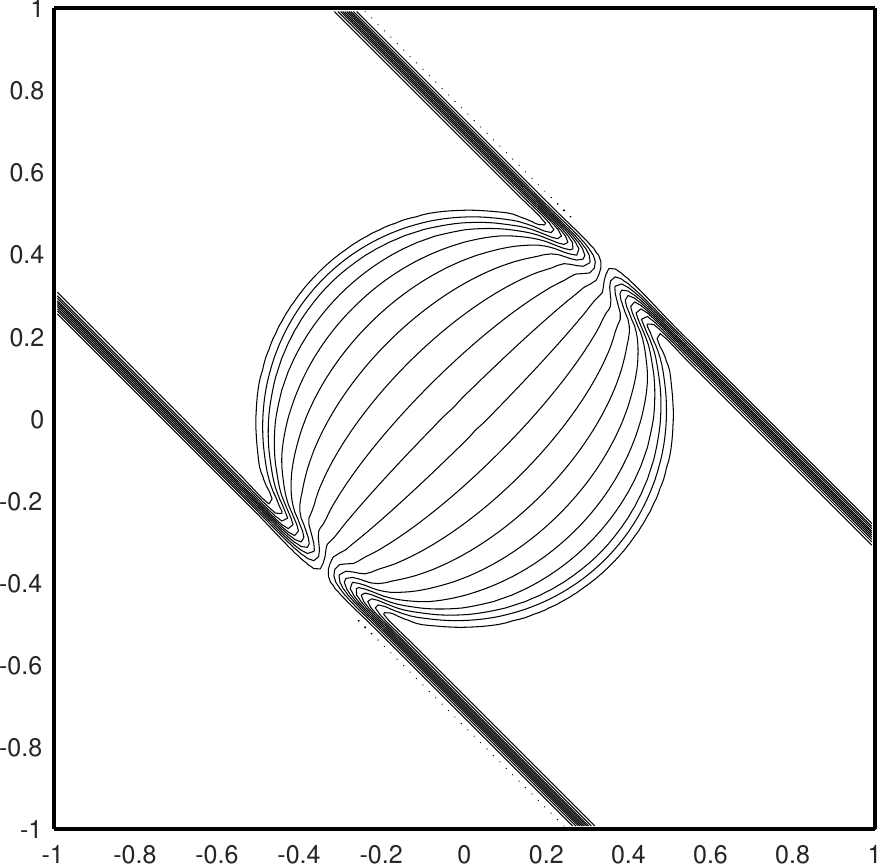} 
            \includegraphics[width=0.3\textwidth]{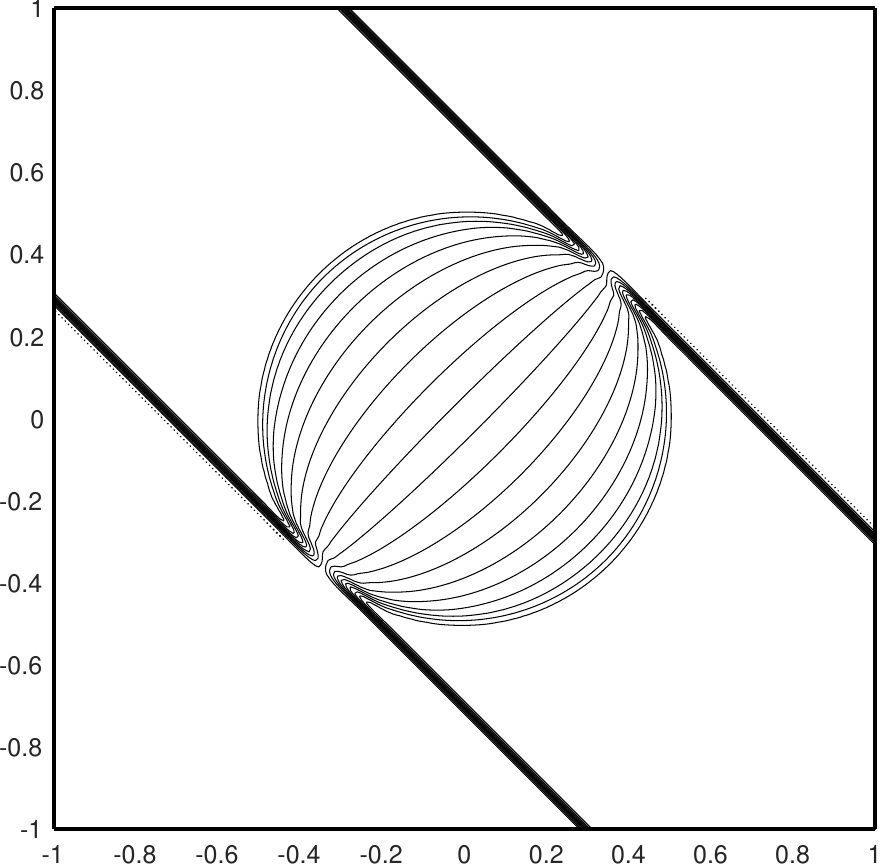}}
\centerline{\includegraphics[width=0.3\textwidth]{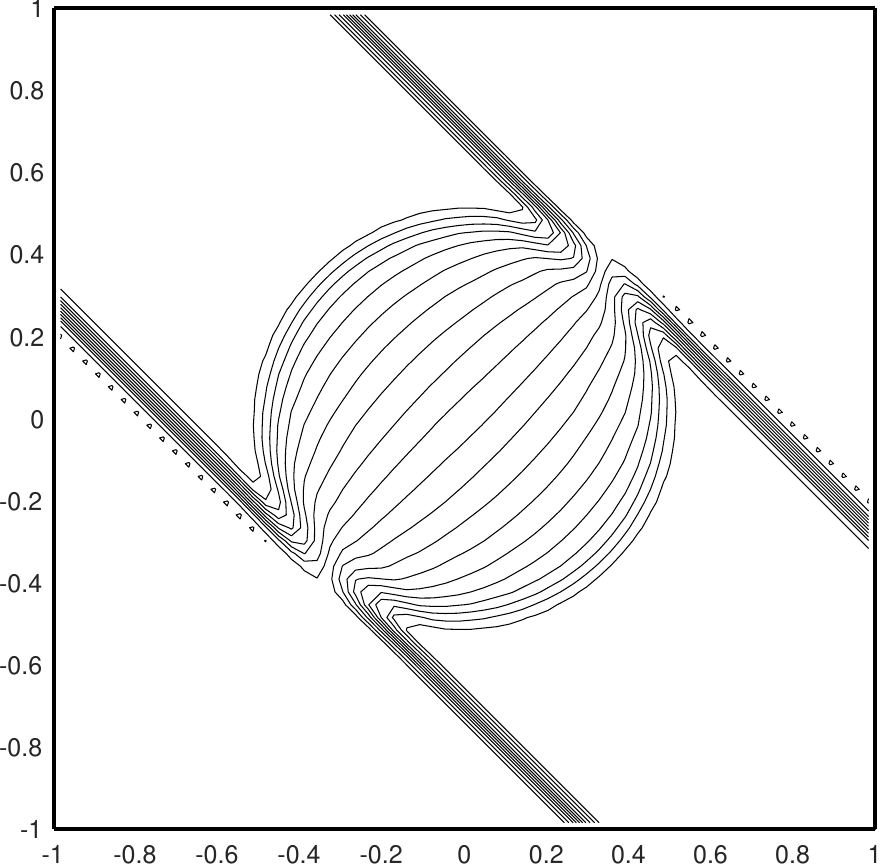}
            \includegraphics[width=0.3\textwidth]{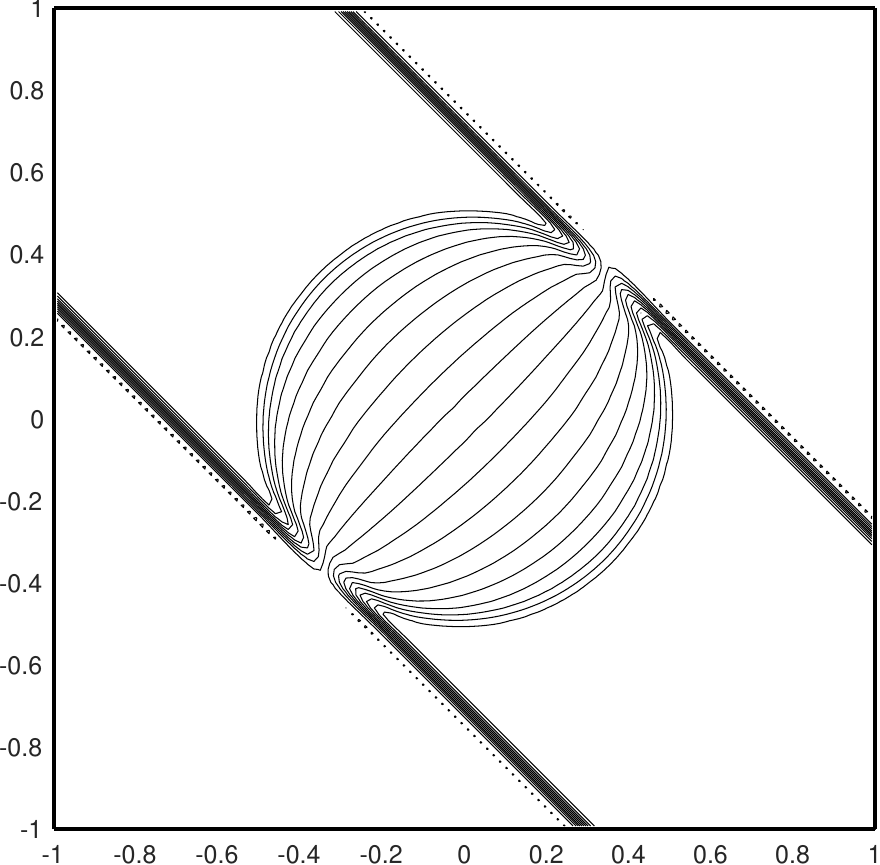} 
            \includegraphics[width=0.3\textwidth]{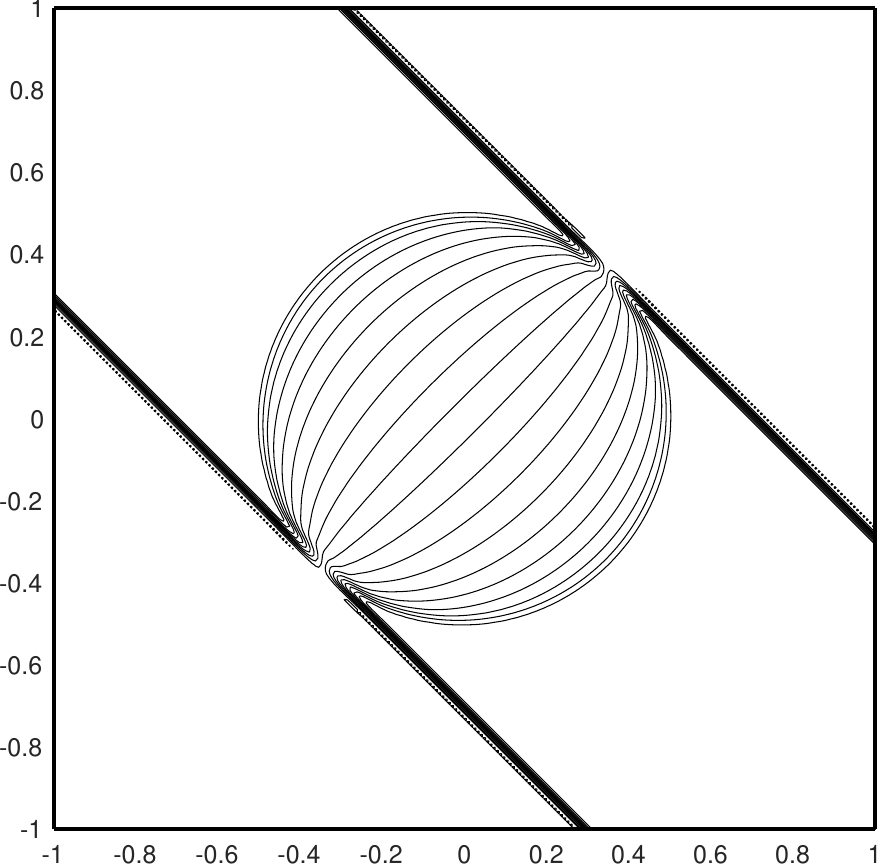}}
\centerline{\includegraphics[width=0.3\textwidth]{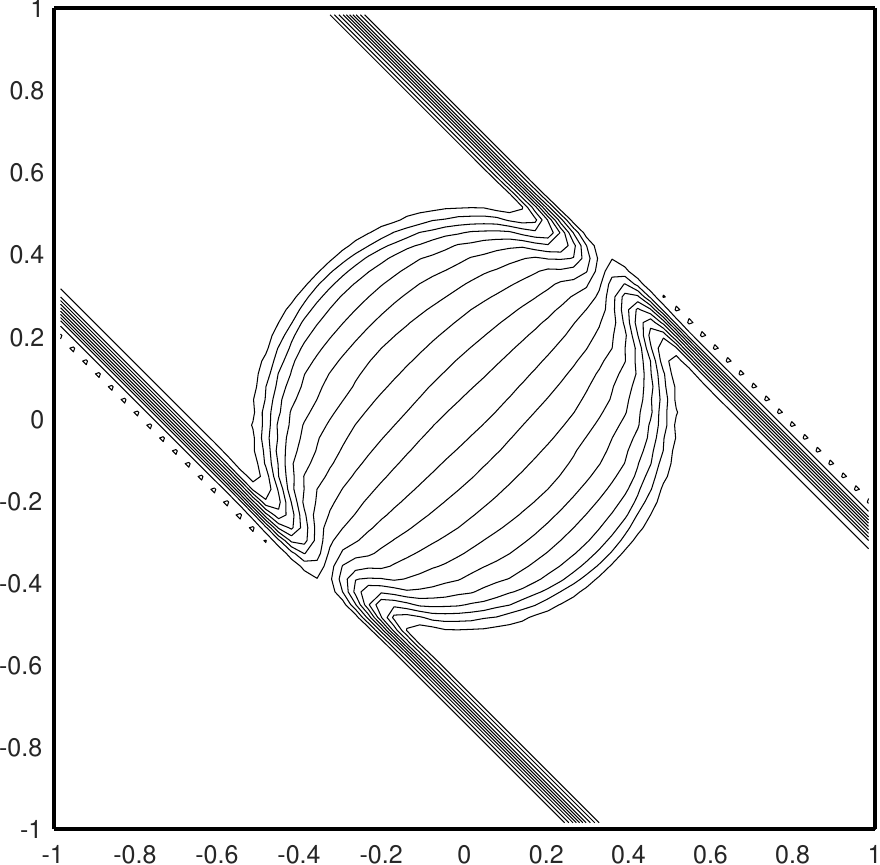} 
            \includegraphics[width=0.3\textwidth]{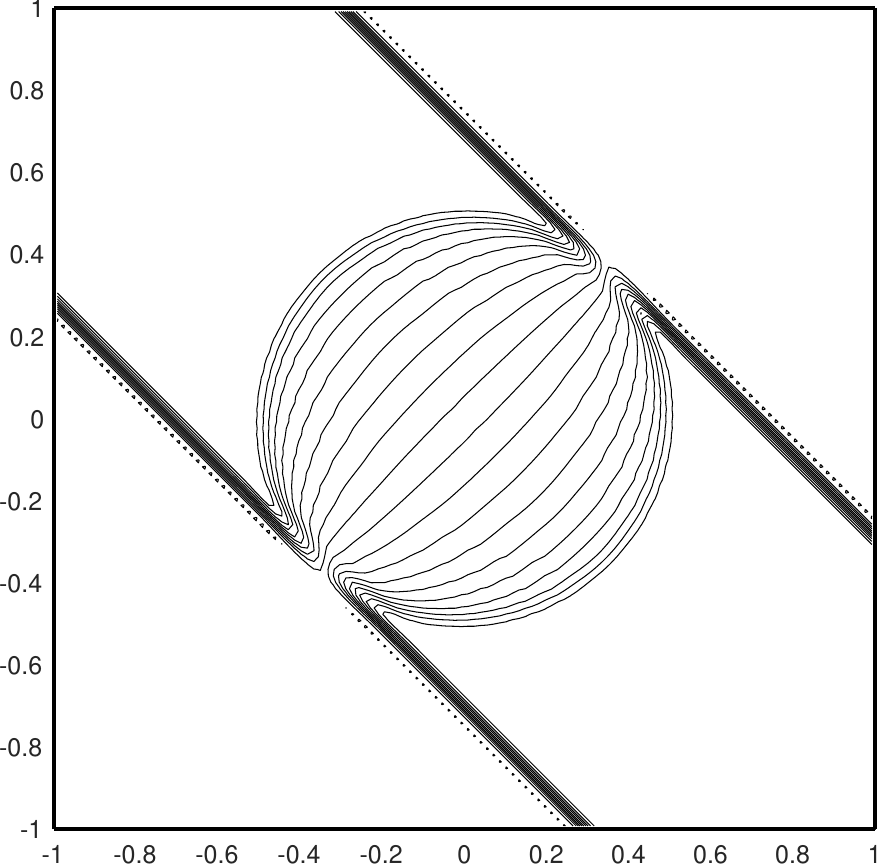}
            \includegraphics[width=0.3\textwidth]{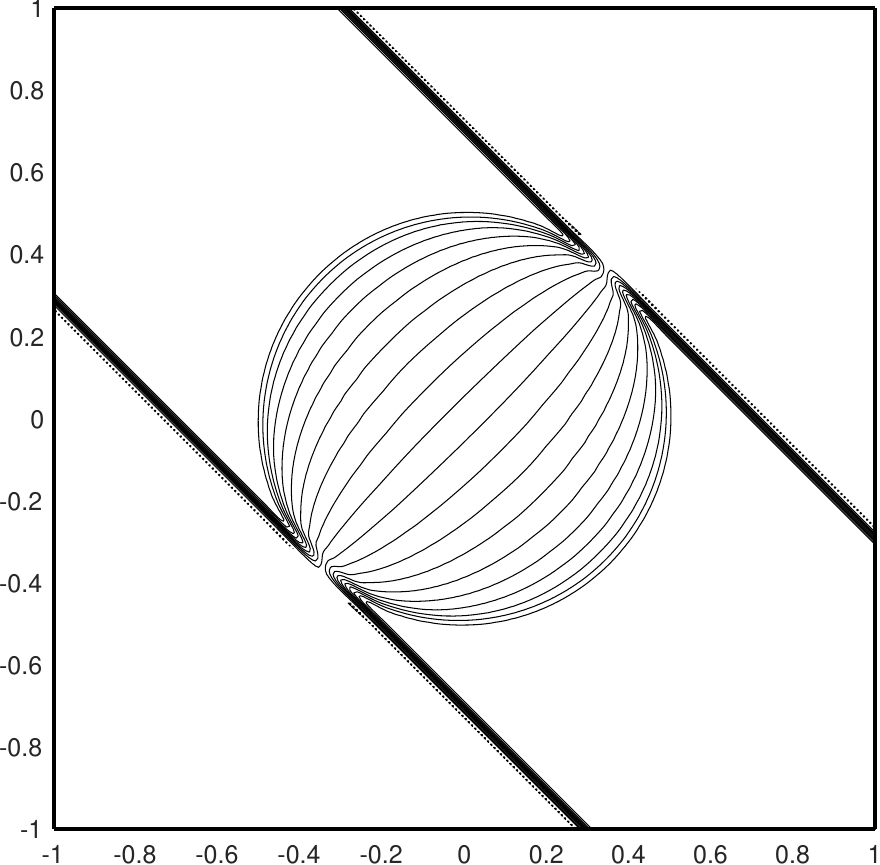}}
\caption{Approximation of discontinuous solution at $t=0.5$ on grids with $64\times 64$ (left), $128\times128$ (middle) and $256\times 256$ (right) cells using the AF method with exact evolution (first row), EG2 (second row), EG2$_{0.7}$ (third row) and $\mathrm{EG}^{\text{quad}}$ (fourth row). All methods use time steps near their stability limit.\label{fig:EG2cDiscontinuous} } 
\end{figure}
Appendix~\ref{app:disc} presents the corresponding results for simplified AF methods using numerical quadratures for the integration along the base of the bicharacteristic cone.

\section{Central Weighted Essentially Non-Oscillatory Reconstruction}\label{sec:CWENO}
In the preceding sections, the stability requirements of the AF method were primarily addressed through the design of new evolution operators. To further enhance stability, we now introduce a novel AFCW method that combines the CWENO reconstruction with the approximate EG operators proposed in the previous sections.

The fundamental advantage of any AF method is its compact, high-order representation of the solution, whose Degrees of Freedom (DoFs) consist of both conservative cell averages and non-conservative point values on edges. By employing the CWENO reconstruction, we effectively increase the DoFs and provide the necessary flexibility to incorporate upwinding information via the evolution operators. Note that we still use the point values at edges as independently evolved variables.  The resulting combination leverages the compact, high-order character of AF and the robust, non-oscillatory limiting of CWENO, thereby delivering the enhanced stability required for hyperbolic systems.
		
To achieve high-order spatial accuracy within each computational cell, we adopt a third-order CWENO reconstruction procedure; see, e.g., \cite{LevyPuppoRusso1999,LevyPuppoRusso2000}. Unlike standard CWENO applications that supply only interface point values for use in numerical fluxes, here we reconstruct, in each mesh cell, a piecewise quadratic polynomial that will be used in the approximate evolution operators. Then in \eqref{recon-formula}, we set $m=5$ with coefficients $C_k$ and basis functions $N_k$ listed in Table \ref{recon}.
\begin{table}[h]
\centering
\caption{Coefficients and basis functions for CWENO reconstruction}
\medskip 
\begin{tabular}{c c c  } 
\hline 
$ k$ & $C_k$ &$N_k$ \\
\hline
0 & $Q_{i,j}$ & $1$\\
1 & $2\omega_0\lb c_1-S_a\rb +\displaystyle\sum_{m=1}^{4}\omega_m\,a_m$ & $\dfrac{\xi}{2}$\\
2 & $2\omega_0\lb c_2-S_b\rb +\displaystyle\sum_{m=1}^{4}\omega_m\,b_m$ & $\dfrac{\eta}{2}$\\
3 & $2\omega_0\,c_3$ & $\dfrac{\xi^2}{4}-\dfrac{1}{12}$\\[2.ex] 
4 & $2\omega_0\,c_4$ & $\dfrac{\eta^2}{4}-\dfrac{1}{12}$\\[2.ex] 
5 & $2\omega_0\,c_5$ & $\dfrac{\xi\,\eta}{4}$\\[1.ex] 
\hline
\end{tabular}
\label{recon}
\end{table}
\ignore{
Assume that the computational domain is partitioned into uniform rectangular cells
\begin{equation*}
\Omega_{i,j}:=[x_{i-\frac{1}{2}},x_{i+ \frac{1}{2}}]\times[y_{j-\frac{1}{2}}, y_{j+ \frac{1}{2}}]
\end{equation*}
centered at $\lb x_i,y_j\rb =\lb \lb x_{i-\frac12}+x_{i+\frac12}\rb /2,\ \lb y_{j-\frac12}+y_{j+\frac12}\rb /2\rb $, with $x_{i+\frac12}-x_{i-\frac12}\equiv\Delta x$ and $y_{j+\frac12}-y_{j-\frac12}\equiv\Delta y$ for all $i,j$. Denote by $Q_{i,j}\lb t\rb $ the cell average of $Q\lb \cdot,\cdot,t\rb $ over $\Omega_{i,j}$
\begin{equation*}
Q_{i,j}\lb t\rb \;\approx\;\frac{1}{\Delta x\,\Delta y}\int_{\Omega_{i,j}} Q\lb x,y,t\rb \,\mathrm{d}x\,\mathrm{d}y,
\end{equation*}
and suppose that all $Q_{i,j}\lb t\rb $ are available at a given time level $t=t^n\ge0$. On each cell we introduce reference coordinates $\lb \xi,\eta\rb \in[-1,1]^2$ via the affine mapping
$$
\xi = \frac{2\lb x - x_i\rb }{\Delta x}, \quad \eta = \frac{2\lb y - y_j\rb }{\Delta y}, \quad \lb \xi, \eta\rb  \in [-1, 1]^2.
$$
We reconstruct a local quadratic polynomial on $\Omega_{i,j}$ at time $t^n$ in a modal basis,
\begin{equation*}
q_{i,j}\lb \xi,\eta\rb =\sum\limits^{6}_{k=1}C_kN_k\lb \xi,\eta\rb 
\end{equation*}
with coefficients $c_k$ and basis functions $N_k$ listed in Table \ref{recon}. The quadratic modes are chosen to have zero cell average so that $C_1=U_{i,j}$.
}
The quadratic modes are chosen to have zero cell average so that $C_0=Q_{i,j}$.
The auxiliary central coefficients $c_\ell$ are computed from neighbouring cell averages
\begin{equation*}
\begin{aligned}
&c_0=Q_{i,j},\quad c_1=\frac{1}{2}\lb Q_{i+1,j}-Q_{i-1,j}\rb ,\quad c_2=\frac{1}{2}\lb Q_{i,j+1}-Q_{i,j-1}\rb ,\\[2pt]
&c_3=\frac{1}{2}\lb Q_{i+1,j}+Q_{i-1,j}\rb -Q_{i,j},\quad c_4=\frac{1}{2}\lb Q_{i,j+1}+Q_{i,j-1}\rb -Q_{i,j},\\[2pt]
&c_5=\frac{1}{4}\lb Q_{i+1,j+1}-Q_{i+1,j-1}-Q_{i-1,j+1}+Q_{i-1,j-1}\rb .
\end{aligned}
\end{equation*}
The parameters $a_m$ and $b_m$ ($m=1,\ldots,4$)  denote the one-sided first differences associated with four directional substencils (east–north, west–north, west–south, east–south):
\begin{equation*}
\begin{aligned}
&a_1=Q_{i+1,j}-Q_{i,j},  &&a_2=Q_{i,j}-Q_{i-1,j},  &&a_3=Q_{i,j}-Q_{i-1,j},  &&a_4=Q_{i+1,j}-Q_{i,j},\\
&b_1=Q_{i,j+1}-Q_{i,j},  && b_2=Q_{i,j+1}-Q_{i,j}, &&b_3=Q_{i,j}-Q_{i,j-1},  &&b_4=Q_{i,j}-Q_{i,j-1}.\\
\end{aligned}
\end{equation*}
We also define the (scaled) averages
\begin{equation*}
S_a=\frac{1}{8}\sum_{m=1}^{4} a_m,\quad
S_b=\frac{1}{8}\sum_{m=1}^{4} b_m.
\end{equation*}
To achieve third-order accuracy in smooth regions and non-oscillatory behaviour near discontinuities, the nonlinear CWENO weights are
\begin{equation*}
\tilde{\omega}_m=\frac{\gamma_m}{\lb \varepsilon+\beta_m\rb ^r},\qquad \omega_m=\frac{\tilde{\omega}_m}{\sum_{s=0}^{4}\tilde{\omega}_s},\qquad m=0,1,2,3,4,
\end{equation*}
with linear weights $\gamma_0=\tfrac{1}{2}$ and $\gamma_1=\gamma_2=\gamma_3=\gamma_4=\tfrac{1}{8}$. Here $\varepsilon>0$ is a small parameter and $r\ge1$ is usually chosen even; in all numerical examples, we take $\varepsilon=10^{-12}$ and $r=2$. By construction, $\omega_m\ge0$ and $\sum_{m=0}^{4}\omega_m=1$.

Finally, the smoothness indicators are
\begin{equation*}
\beta_0=4\lb c_1-S_a\rb ^2+4\lb c_2-S_b\rb ^2+\frac{5}{3}\lb c_3^2+c_4^2\rb +\frac{4}{3}c_5^2,
\end{equation*}
and
\begin{equation*}
\beta_m=a_m^2+b_m^2,\qquad m=1,2,3,4.
\end{equation*}
With these definitions, the polynomial $q_{i,j}$ provides a third-order accurate, non-oscillatory CWENO reconstruction on each cell.
		
\subsection{Investigation of accuracy}
In this section, we present the accuracy study for the AFCW methods  obtained using the third-order CWENO reconstruction in combination with the EG2, $\mathrm{EG}^{\text{quad}}$, and EG2$_{0.8,0.2}$ operators. We consider Examples \ref{ex:1} and \ref{ex:2}, compute the numerical results on a sequence of grids with $64 \times 64$, $128 \times 128$, and $256 \times 256$ cells at times $t=0.1$ and $t=1$, and report the obtained results in Tables \ref{TabCWENOex1t01}–\ref{TabCWENOex2t1}. As one can see, all AFCW methods achieve third-order accuracy and are stable for $\mathrm{CFL}=0.5$. In addition, one can observe that the method employing the $\mathrm{EG}^{\text{quad}}$ operator provides higher accuracy compared to the methods using the EG2 or EG2$_{0.8,0.2}$ operators. In comparison with the errors of the AF methods using the original reconstruction (see Tables \ref{Tab:convergenceProblem1_eg2eg11_t01}–\ref{Tab:convergenceProblem2_eg2eg11_t1}), the errors obtained in the present case are noticeably larger
for CFL=0.5. Note, however, that for CFL=0.7 the AFCW method with EG$^{\rm{quad}}$ operator yields  comparable errors with those reported in Tables~\ref{Tab:convergenceProblem1_eg2eg11_t01}–\ref{Tab:convergenceProblem2_eg2eg11_t1}; see Tables \ref{TabCWENOex1t01_CFL07}–\ref{TabCWENOex2t1_CFL07}. In this case, approximate evolution operators integrate over the base circle of the bicharacteristic cone, taking into account all six neighbouring cells around the midpoint of an edge.

\begin{table}[!ht]
\centering
\caption{Errors measured in the $L_1$-norm and EOC for Example~\ref{ex:1} using AFCW method with EG2, $\mathrm{EG}^{\text{quad}}$ and $\mathrm{EG2}_{0.8,0.2}$ with $\mathrm{CFL} = 0.5$ at $t = 0.1$.}
\label{TabCWENOex1t01}
\vspace*{0.15cm}
\sisetup{scientific-notation = true,round-mode = places}
\renewcommand{\arraystretch}{0.9}
\setlength{\tabcolsep}{4pt}
\begin{tabular}{c*{3}{S[round-precision=2, table-format=1.2e-2]}*{3}{S[round-precision=3, table-format=1.4]}}
\toprule
Res. & \multicolumn{3}{c}{Error in $p$} & \multicolumn{3}{c}{EOC} \\
\cmidrule(lr){2-4} \cmidrule(lr){5-7}
	 & {$\mathrm{EG2}$} & {$\mathrm{EG}^{\text{quad}}$}& {$\mathrm{EG2}_{0.8,0.2}$} & {$\mathrm{EG2}$} & {$\mathrm{EG}^{\text{quad}}$}&$\mathrm{EG2}_{0.8,0.2}$ \\
\midrule
$64\times64$    & \num{2.10e-04} & \num{1.21e-04} & \num{2.04e-04} & {---}  & {---} & {---}\\
$128\times128$  & \num{2.52e-05} & \num{1.42e-05} & \num{2.46e-05} & 3.0584 & 3.0890& 3.0515\\
$256\times256$  & \num{3.03e-06} & \num{1.67e-06} & \num{3.02e-06} & 3.0521 & 3.0835& 3.0245 \\
\bottomrule
\end{tabular}
\end{table}
		
\begin{table}[!ht]
\centering
\caption{Errors measured in the $L_1$-norm and EOC for Example~\ref{ex:1} using AFCW method with EG2, $\mathrm{EG}^{\text{quad}}$ and $\mathrm{EG2}_{0.8,0.2}$ with $\mathrm{CFL} = 0.5$ at $t = 1$.}
\vspace*{0.15cm}
\sisetup{scientific-notation = true,round-mode = places}
\renewcommand{\arraystretch}{0.9}
\setlength{\tabcolsep}{4pt}
\begin{tabular}{c*{3}{S[round-precision=2, table-format=1.2e-2]}*{3}{S[round-precision=3, table-format=1.4]}}
\toprule
Res. & \multicolumn{3}{c}{Error in $p$} & \multicolumn{3}{c}{EOC} \\
\cmidrule(lr){2-4} \cmidrule(lr){5-7}
     & {$\mathrm{EG2}$} & {$\mathrm{EG}^{\text{quad}}$}& {$\mathrm{EG2}_{0.8,0.2}$} & {$\mathrm{EG2}$} & {$\mathrm{EG}^{\text{quad}}$} & {$\mathrm{EG2}_{0.8,0.2}$} \\
\midrule
$64\times64$    & \num{2.2999950248132738e-03} & \num{1.2671479390351835e-03} & \num{2.4027728007210000e-03} & {---} & {---} & {---} \\
$128\times128$  & \num{2.9078094668496765e-04} & \num{1.5914103398114896e-04 } & \num{2.9431065168686234e-04} &2.9836 & 2.9932 & 3.0293  \\
$256\times256$  & \num{3.6430188984395761e-05} & \num{1.9896900068457520e-05}  & \num{3.6808457753564990e-05}& 2.9967 & 2.9997 & 2.9992 \\
\bottomrule
\end{tabular}
\end{table}
		
\begin{table}[!ht]
\centering
\caption{Errors measured in the $L_1$-norm and EOC for Example~\ref{ex:2} using AFCW method with EG2, $\mathrm{EG}^{\text{quad}}$ and $\mathrm{EG2}_{0.8,0.2}$ with $\mathrm{CFL} = 0.5$ at $t = 0.1$.}
\vspace*{0.15cm}
\sisetup{scientific-notation = true,round-mode = places}
\renewcommand{\arraystretch}{0.9}
\setlength{\tabcolsep}{4pt}
\begin{tabular}{c*{3}{S[round-precision=2, table-format=1.2e-2]}*{3}{S[round-precision=3, table-format=1.4]}}
\toprule
Res. & \multicolumn{3}{c}{Error in $u, v$} & \multicolumn{3}{c}{EOC} \\
\cmidrule(lr){2-4} \cmidrule(lr){5-7}
	 & {$\mathrm{EG2}$} & {$\mathrm{EG}^{\text{quad}}$} & {$\mathrm{EG2}_{0.8,0.2}$}& {$\mathrm{EG2}$} & {$\mathrm{EG}^{\text{quad}}$}& {$\mathrm{EG2}_{0.8,0.2}$} \\
\midrule
$64\times64$   & \num{1.65e-04} & \num{9.50e-05} & \num{1.61e-04}& {---} & {---}  & {---} \\
$128\times128$ & \num{1.98e-05} & \num{1.11e-05} & \num{1.93e-05}& 3.0629& 3.0928 & {3.0558}  \\
$256\times256$ & \num{2.38e-06} & \num{1.31e-06} & \num{2.38e-06}& 3.0532 & 3.0846 & {3.0257}  \\
\bottomrule
\end{tabular}
\end{table}
		
\begin{table}[!ht]
\centering
\caption{Errors measured in the $L_1$-norm and EOC for Example~\ref{ex:2} using AFCW method with EG2, $\mathrm{EG}^{\text{quad}}$ and $\mathrm{EG2}_{0.8,0.2}$ with $\mathrm{CFL} = 0.5$ at $t = 1$.}
\label{TabCWENOex2t1}
\vspace*{0.15cm}
\sisetup{scientific-notation = true,round-mode = places}
\renewcommand{\arraystretch}{0.9}
\setlength{\tabcolsep}{4pt}
\begin{tabular}{c*{3}{S[round-precision=2, table-format=1.2e-2]}*{3}{S[round-precision=3, table-format=1.4]}}
\toprule
Res. & \multicolumn{3}{c}{Error in $u, v$} & \multicolumn{3}{c}{EOC} \\
\cmidrule(lr){2-4} \cmidrule(lr){5-7}
	 & {$\mathrm{EG2}$} & {$\mathrm{EG}^{\text{quad}}$}& {$\mathrm{EG2}_{0.8,0.2}$} & {$\mathrm{EG2}$} & {$\mathrm{EG}^{\text{quad}}$}& {$\mathrm{EG2}_{0.8,0.2}$} \\
\midrule
$64\times64$    & \num{1.81e-03} & \num{9.99e-04} & \num{1.89e-03} & {---} & {---} & {---} \\
$128\times128$  & \num{2.29e-04} & \num{1.25e-04} & \num{2.31e-04} & 2.9878& 2.9975& 3.0335 \\
$256\times256$  & \num{2.86e-05} & \num{1.56e-05} & \num{2.89e-05} & 2.9978& 3.0008& 3.0003 \\
\bottomrule
\end{tabular}
\end{table}

\begin{table}[!ht]
\centering
\caption{Errors measured in the $L_1$-norm and EOC for Example~\ref{ex:1} using AFCW method with EG2, $\mathrm{EG}^{\text{quad}}$ and $\mathrm{EG2}_{0.8,0.2}$ with $\mathrm{CFL} = 0.7$ at $t = 0.1$.}
\label{TabCWENOex1t01_CFL07}
\vspace*{0.15cm}
\sisetup{scientific-notation = true,round-mode = places}
\renewcommand{\arraystretch}{0.9}
\setlength{\tabcolsep}{4pt}
\begin{tabular}{c*{3}{S[round-precision=2, table-format=1.2e-2]}*{3}{S[round-precision=3, table-format=1.4]}}
\toprule
Res. & \multicolumn{3}{c}{Error in $p$} & \multicolumn{3}{c}{EOC} \\
\cmidrule(lr){2-4} \cmidrule(lr){5-7}
	 & {$\mathrm{EG2}$} & {$\mathrm{EG}^{\text{quad}}$}& {$\mathrm{EG2}_{0.8,0.2}$} & {$\mathrm{EG2}$} & {$\mathrm{EG}^{\text{quad}}$}&$\mathrm{EG2}_{0.8,0.2}$ \\
\midrule
$64\times64$    & \num{1.6529842877988467e-04} & \num{ 6.3614765155576598e-05} & \num{1.3799492577844946e-04} & {---}  & {---} & {---}\\
$128\times128$  & \num{1.8098757857884977e-05} & \num{5.3714482357923278e-06} &\num{1.6062041393765128e-05} & 3.1911 & 3.5660& 3.1029\\
$256\times256$  & \num{2.0669817936028865e-06} & \num{ 4.7933093650915119e-07} & \num{1.9716934397094635e-06}&  3.1303 & 3.4862& 3.0261 \\
\bottomrule
\end{tabular}
\end{table}
		
\begin{table}[!ht]
\centering
\caption{Errors measured in the $L_1$-norm and EOC for Example~\ref{ex:1} using AFCW method with EG2, $\mathrm{EG}^{\text{quad}}$ and $\mathrm{EG2}_{0.8,0.2}$ with $\mathrm{CFL} = 0.7$ at $t = 1$.}
\vspace*{0.15cm}
\sisetup{scientific-notation = true,round-mode = places}
\renewcommand{\arraystretch}{0.9}
\setlength{\tabcolsep}{4pt}
\begin{tabular}{c*{3}{S[round-precision=2, table-format=1.2e-2]}*{3}{S[round-precision=3, table-format=1.4]}}
\toprule
Res. & \multicolumn{3}{c}{Error in $p$} & \multicolumn{3}{c}{EOC} \\
\cmidrule(lr){2-4} \cmidrule(lr){5-7}
     & {$\mathrm{EG2}$} & {$\mathrm{EG}^{\text{quad}}$}& {$\mathrm{EG2}_{0.8,0.2}$} & {$\mathrm{EG2}$} & {$\mathrm{EG}^{\text{quad}}$} & {$\mathrm{EG2}_{0.8,0.2}$} \\
\midrule
$64\times64$   & \num{1.4932108929948821e-03} & \num{3.0618672398544349e-04} &\num{1.6428883404075796e-03} & {---} & {---} & {---} \\
$128\times128$  & \num{1.8935967776209959e-04} & \num{3.6806198228908385e-05 } & \num{1.9777004045576916e-04} & 2.9792 & 3.0564 & 3.0543 \\
$256\times256$  & \num{2.3666825362751506e-05} & \num{4.4474429936816163e-06}  & \num{2.4246322239157922e-05}& 3.0002 & 3.0489& 3.0280 \\
\bottomrule
\end{tabular}
\end{table}
		
\begin{table}[!ht]
\centering
\caption{Errors measured in the $L_1$-norm and EOC for Example~\ref{ex:2} using AFCW method with EG2, $\mathrm{EG}^{\text{quad}}$ and $\mathrm{EG2}_{0.8,0.2}$ with $\mathrm{CFL} = 0.7$ at $t = 0.1$.}
\vspace*{0.15cm}
\sisetup{scientific-notation = true,round-mode = places}
\renewcommand{\arraystretch}{0.9}
\setlength{\tabcolsep}{4pt}
\begin{tabular}{c*{3}{S[round-precision=2, table-format=1.2e-2]}*{3}{S[round-precision=3, table-format=1.4]}}
\toprule
Res. & \multicolumn{3}{c}{Error in $u, v$} & \multicolumn{3}{c}{EOC} \\
\cmidrule(lr){2-4} \cmidrule(lr){5-7}
	 & {$\mathrm{EG2}$} & {$\mathrm{EG}^{\text{quad}}$} & {$\mathrm{EG2}_{0.8,0.2}$}& {$\mathrm{EG2}$} & {$\mathrm{EG}^{\text{quad}}$}& {$\mathrm{EG2}_{0.8,0.2}$} \\
\midrule
$64\times64$   & \num{1.3033536226893282e-04} & \num{5.0088080753390898e-05} & \num{1.0885022858007386e-04}& {---} & {---}  & {---} \\
$128\times128$ & \num{1.4229102803590538e-05} & \num{4.2215230890196998e-06} & \num{1.2628816121587375e-05} & 3.1953& 3.5686 & 3.1076 \\
$256\times256$ & \num{1.6238280739671434e-06} & \num{3.7653856401327411e-07} & \num{1.5489895927865034e-06} & 3.1314& 3.4869 & 3.0273 \\
\bottomrule
\end{tabular}
\end{table}
		
\begin{table}[!ht]
\centering
\caption{Errors measured in the $L_1$-norm and EOC for Example~\ref{ex:2} using AFCW method with EG2, $\mathrm{EG}^{\text{quad}}$ and $\mathrm{EG2}_{0.8,0.2}$ with $\mathrm{CFL} = 0.7$ at $t = 1$.}
\label{TabCWENOex2t1_CFL07}
\vspace*{0.15cm}
\sisetup{scientific-notation = true,round-mode = places}
\renewcommand{\arraystretch}{0.9}
\setlength{\tabcolsep}{4pt}
\begin{tabular}{c*{3}{S[round-precision=2, table-format=1.2e-2]}*{3}{S[round-precision=3, table-format=1.4]}}
\toprule
Res. & \multicolumn{3}{c}{Error in $u, v$} & \multicolumn{3}{c}{EOC} \\
\cmidrule(lr){2-4} \cmidrule(lr){5-7}
	 & {$\mathrm{EG2}$} & {$\mathrm{EG}^{\text{quad}}$}& {$\mathrm{EG2}_{0.8,0.2}$} & {$\mathrm{EG2}$} & {$\mathrm{EG}^{\text{quad}}$}& {$\mathrm{EG2}_{0.8,0.2}$} \\
\midrule
$64\times64$     & \num{ 1.1772379604963702e-03} & \num{2.4152168513751114e-04} & 1.2951318304316234e-03  & {---} & {---} & {---} \\
$128\times128$  & \num{1.4886691841044361e-04} & \num{2.8937870705542948e-05} & 1.5547672447887987e-04  &2.9833&  3.0611& 3.0583 \\
$256\times256$  & \num{1.8592381854432380e-05} & \num{ 3.4941301275595900e-06 } &1.9047595098989102e-05 &  3.0012& 3.0500&3.0290  \\
\bottomrule
\end{tabular}
\end{table}
		
\subsection{Approximation of the Stationary Vortex}   
In this section, we investigate the numerical behaviour of the AFCW methods for the stationary vortex described in Example~\ref{ex:3}. The corresponding results for EG2, $\mathrm{EG}^{\text{quad}}$ and EG2$_{0.8,0.2}$ on $64 \times 64$ and $128 \times 128$ grids are presented in Figure \ref{fig:WENO-EG2_EG11_Vortex}. On relatively coarse grids, the reconstructed solution exhibits oscillations and a noticeable deviation from the stationary state. As the grid resolution increases, these oscillations gradually diminish, and the overall solution becomes more consistent with the expected stationary pattern. This trend can be attributed to the wider spatial stencil employed in the CWENO reconstruction, which reduces the performance on coarse meshes. Nevertheless,  EG2, $\mathrm{EG}^{\text{quad}}$ and EG2$_{0.8,0.2}$ operators remain stable even at the maximum admissible $\mathrm{CFL}$ number. Moreover, for the same grid resolution, the $\mathrm{EG}^{\text{quad}}$ operator tends to preserve the stationary state more effectively than the EG2 operator and EG2$_{0.8,0.2}$.
\begin{figure}[htbp]
\centerline{\includegraphics[width=0.3\textwidth]{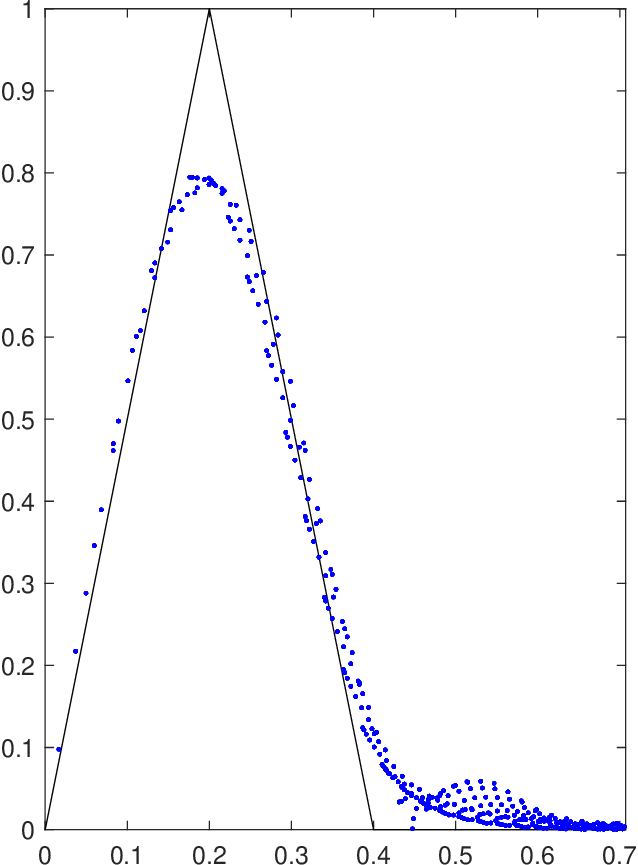}
			\includegraphics[width=0.3\textwidth]{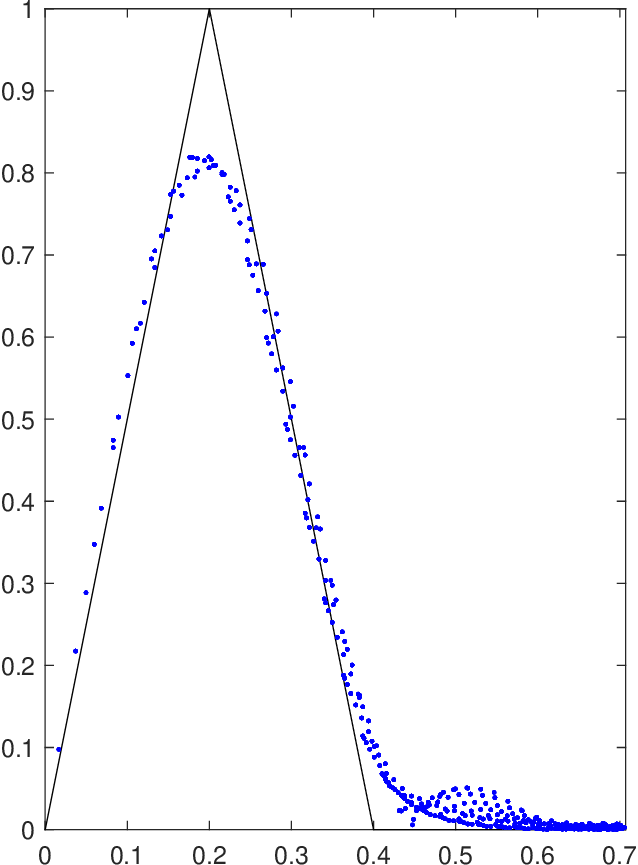}
			\includegraphics[width=0.3\textwidth]{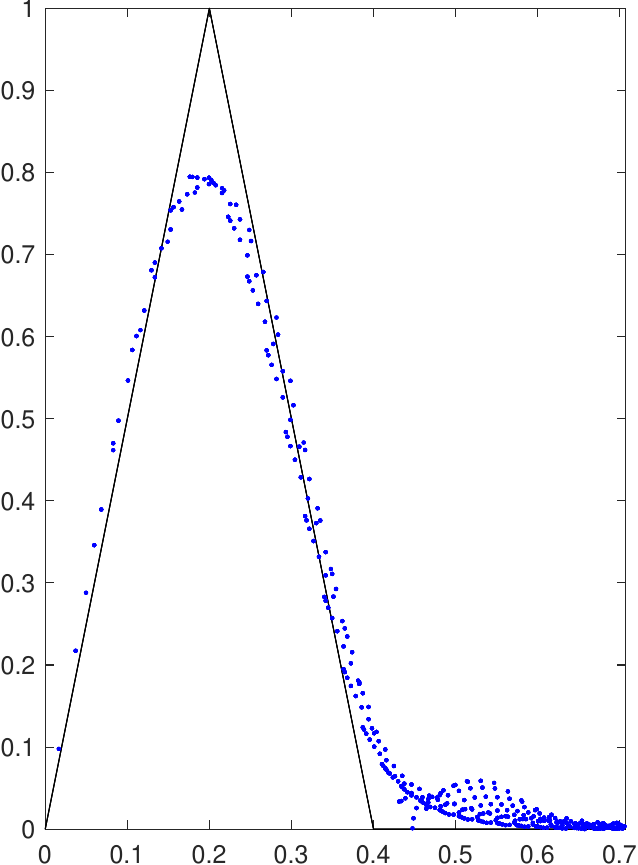}}
\centerline{\includegraphics[width=0.3\textwidth]{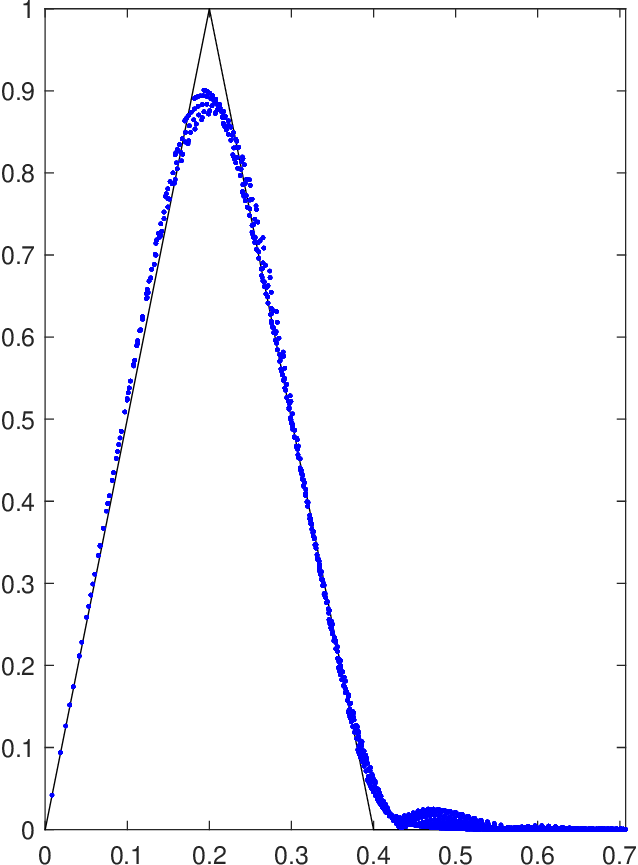}
			\includegraphics[width=0.3\textwidth]{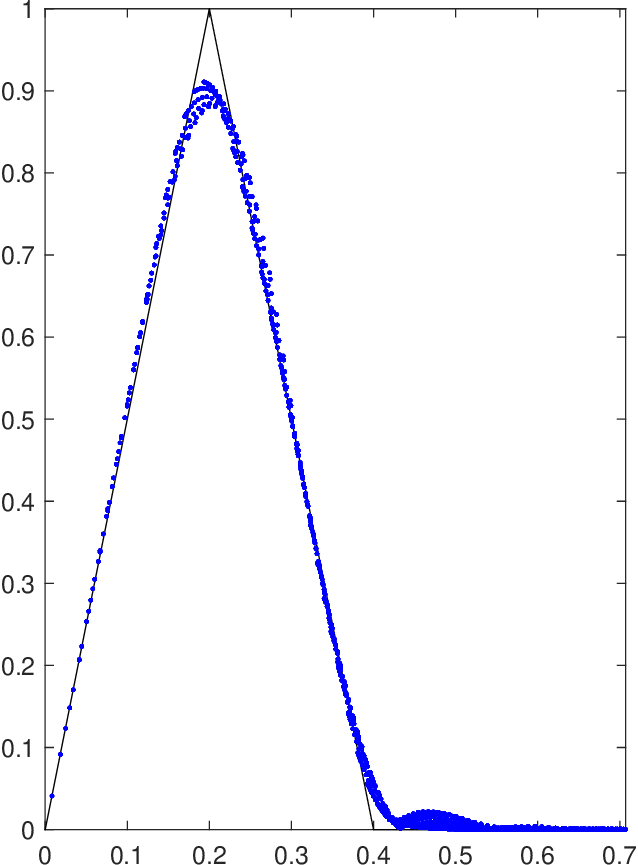}
			\includegraphics[width=0.3\textwidth]{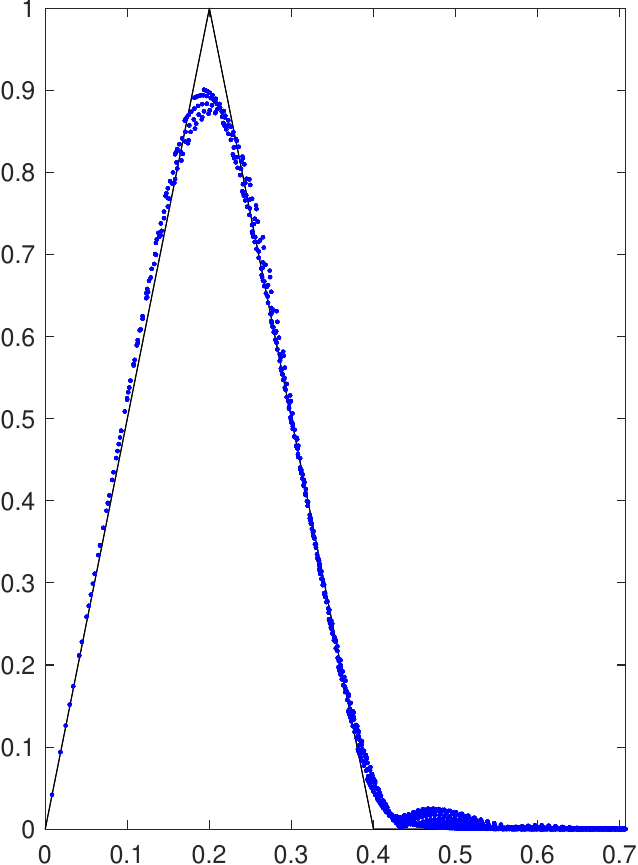}}
\caption{ \label{fig:WENO-EG2_EG11_Vortex}Approximation of the stationary vortex using a grid with $64 \times 64$ (top) and $128 \times 128$ (bottom) cells at $t=100$ with AFCW method using EG2 (left), $\mathrm{EG}^{\text{quad}}$ (middle) and EG2$_{0.8,0.2}$ (right) operators, respectively.}
\end{figure}

\subsection{Approximation of discontinuous solution}
We now present the performance of the AFCW methods with EG2, $\mathrm{EG}^{\text{quad}}$ and EG2$_{0.8,0.2}$ evolution operators for the discontinuous problem, as described in Example~\ref{ex:4}. As illustrated in Figure~\ref{fig:wenoDiscontinuous}, all schemes are capable of capturing the discontinuities sharply and without spurious oscillations. 
The results for  $\mathrm{EG}^{\text{quad}}$ (second row) and  EG$_{0.8,0.2}$ operator (third row) demonstrate performance highly comparable to the AFCW method with EG2 operator, maintaining sharpness at the discontinuities similar to the first row. The benefit of increasing the grid resolution from $64\times64$ (left column) to $256 \times 256$ (right column) is clearly demonstrated, yielding a corresponding increase in the fidelity and resolution of the computed solution features. In comparison with the AF methods presented in Figure~\ref{fig:EG2cDiscontinuous}, the AFCW methods  are slightly smoother near the discontinuities. 
This suggests that the AFCW methods may possess slightly higher numerical dissipation while still preserving the overall sharpness characteristic.

\begin{figure}[htb]
\centerline{\includegraphics[width=0.3\textwidth]{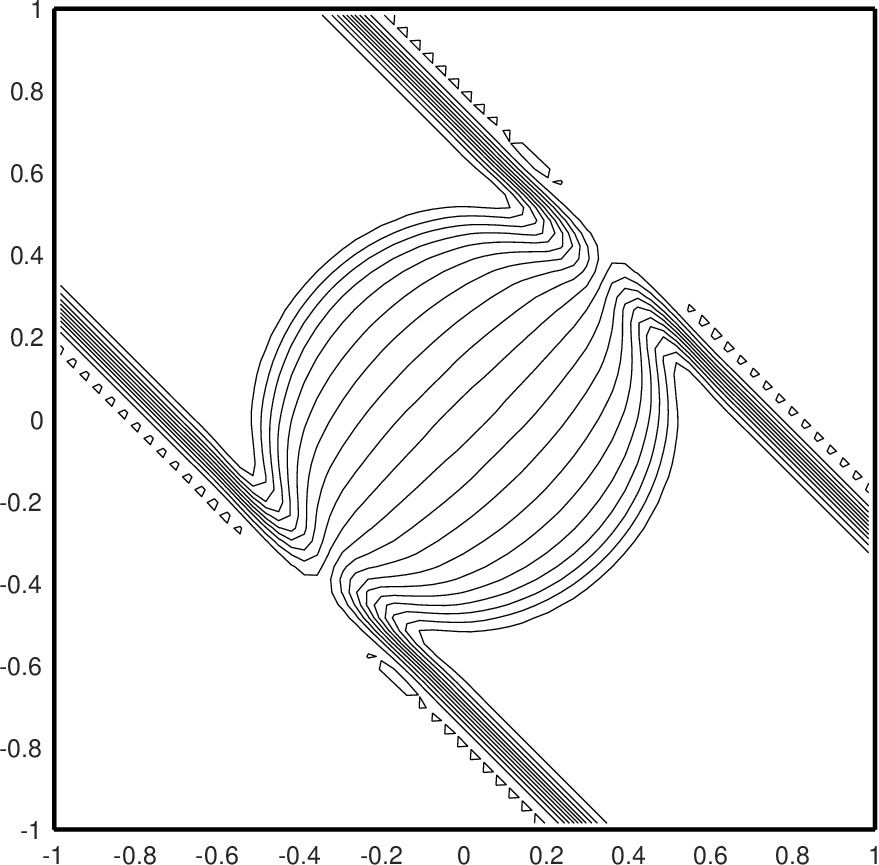} 
			\includegraphics[width=0.3\textwidth]{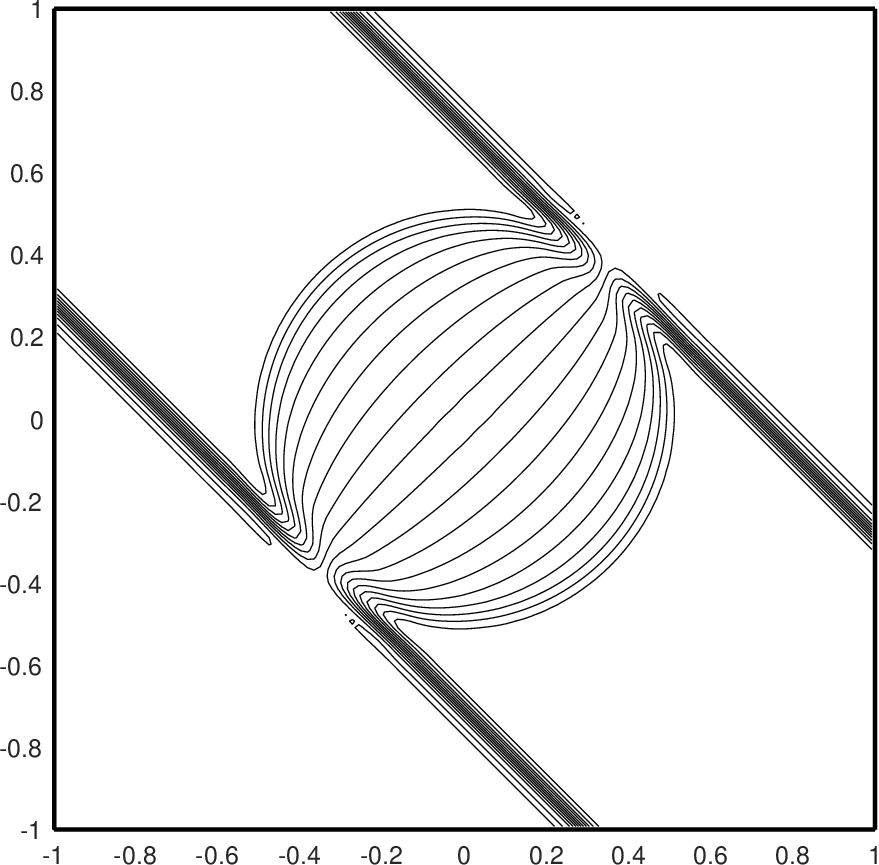}
			\includegraphics[width=0.3\textwidth]{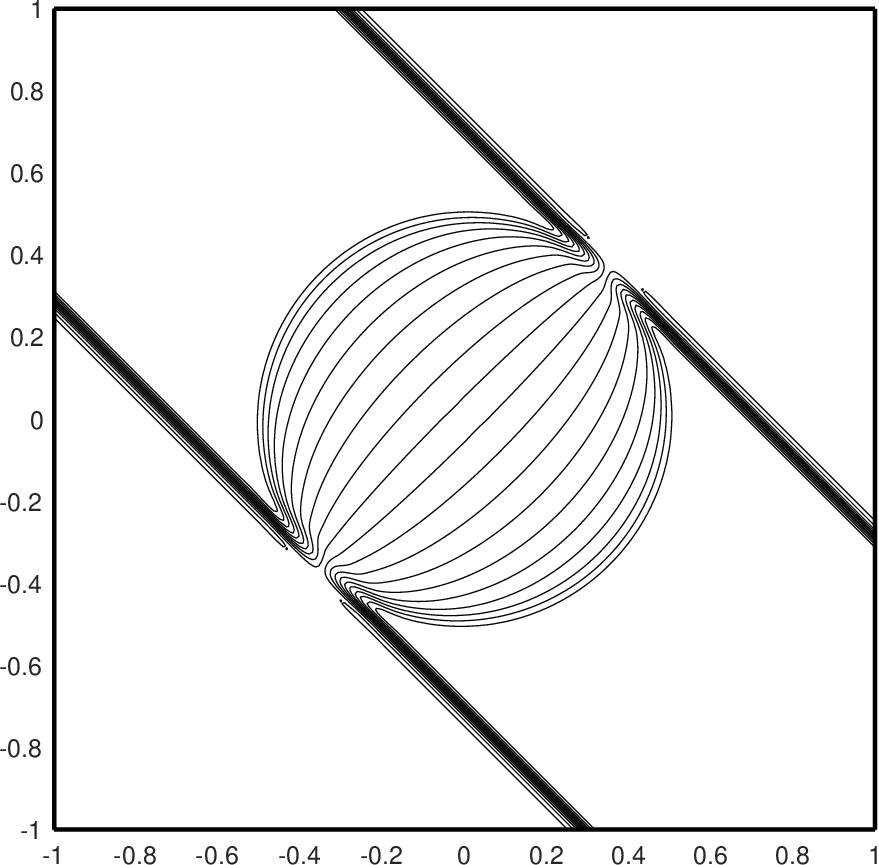}}
\centerline{\includegraphics[width=0.3\textwidth]{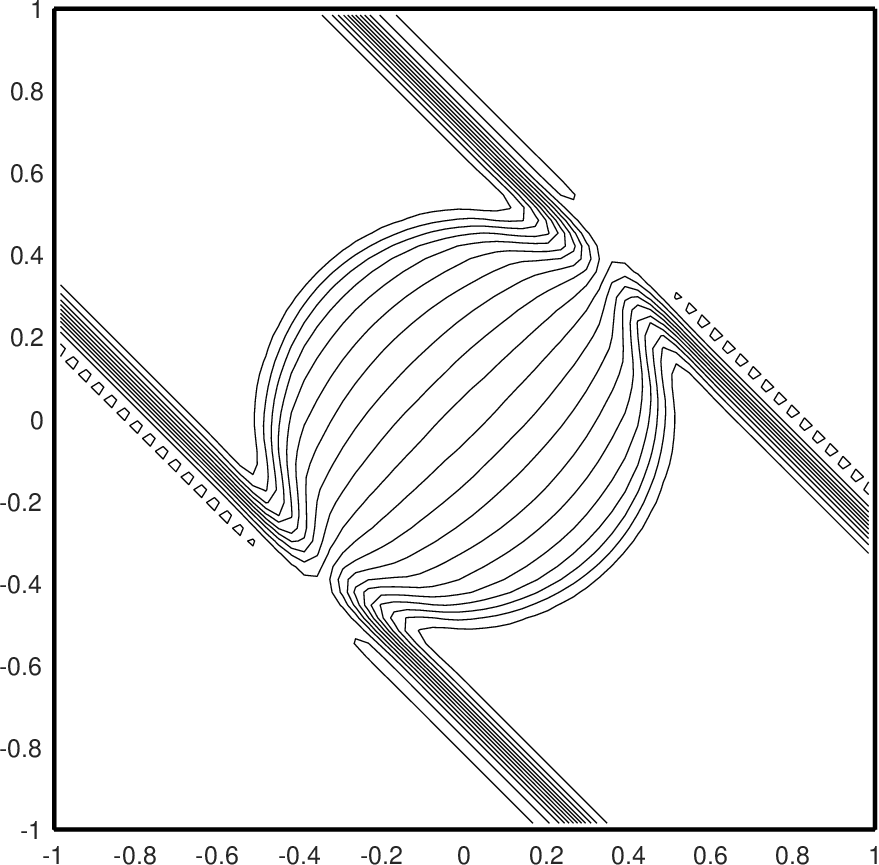}
			\includegraphics[width=0.3\textwidth]{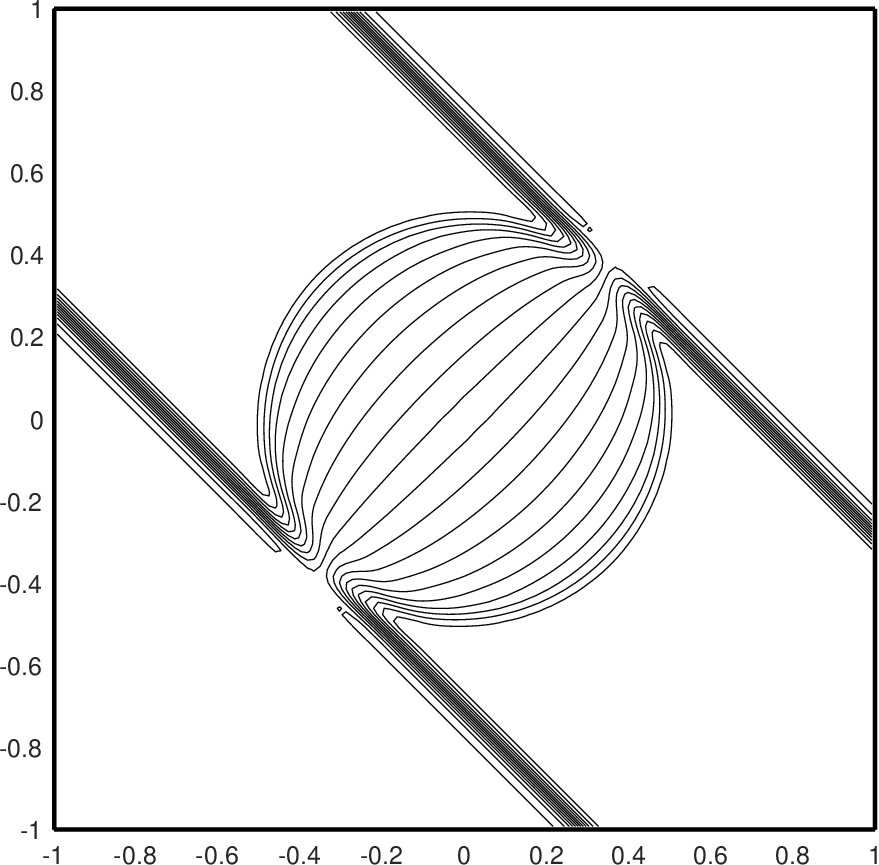} 
			\includegraphics[width=0.3\textwidth]{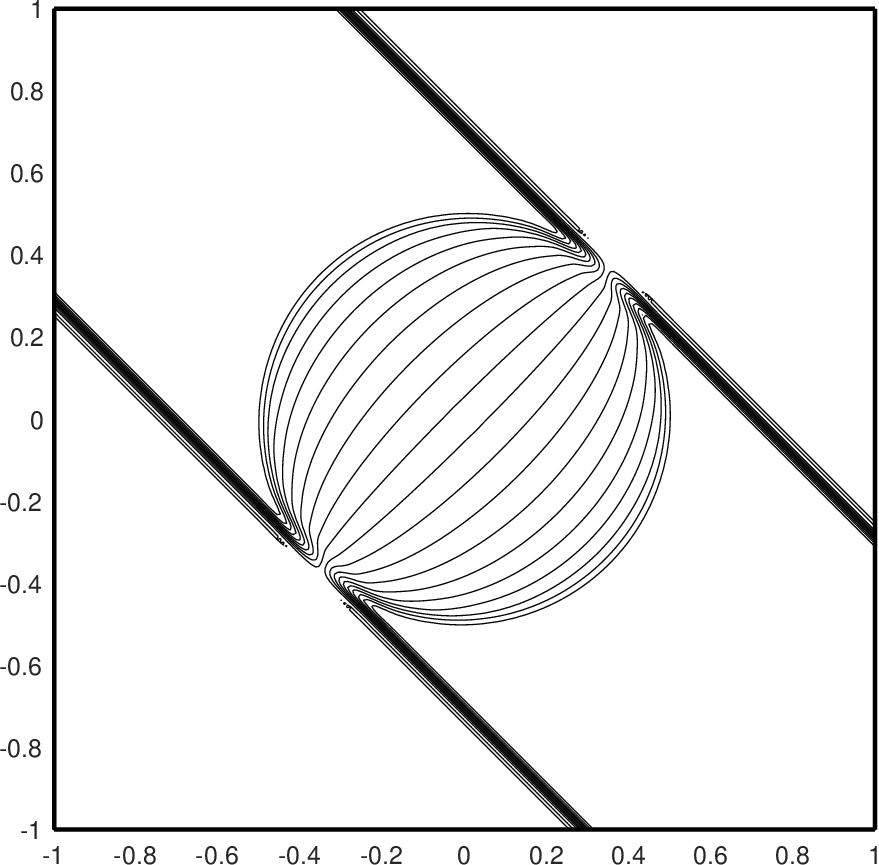}}
\centerline{\includegraphics[width=0.3\textwidth]{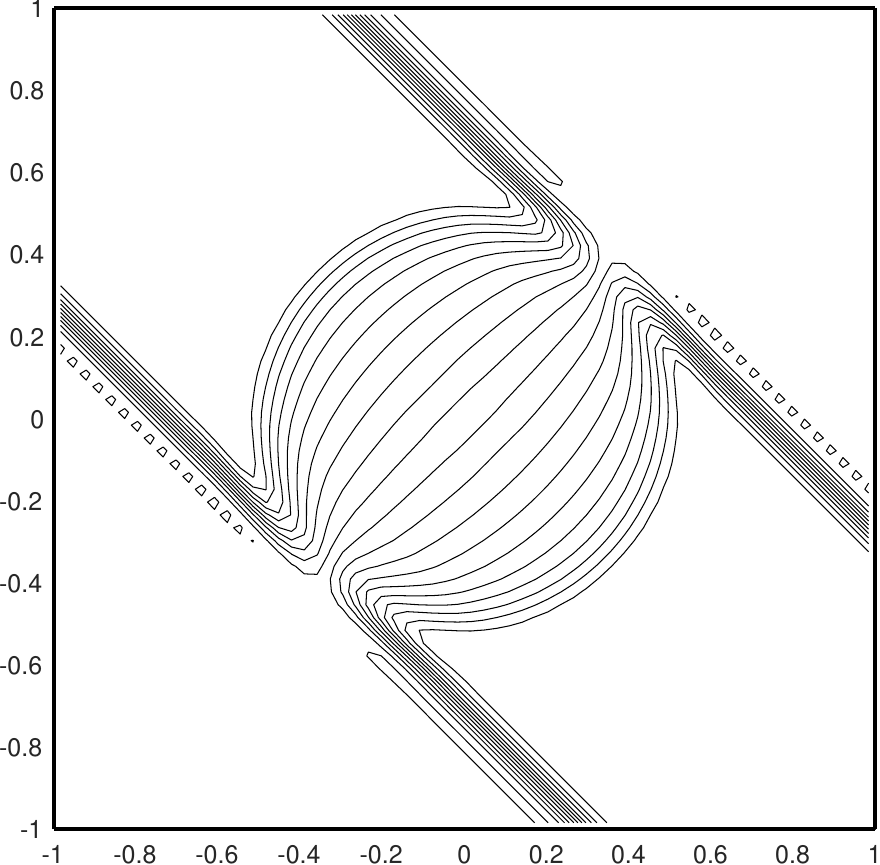}
			\includegraphics[width=0.3\textwidth]{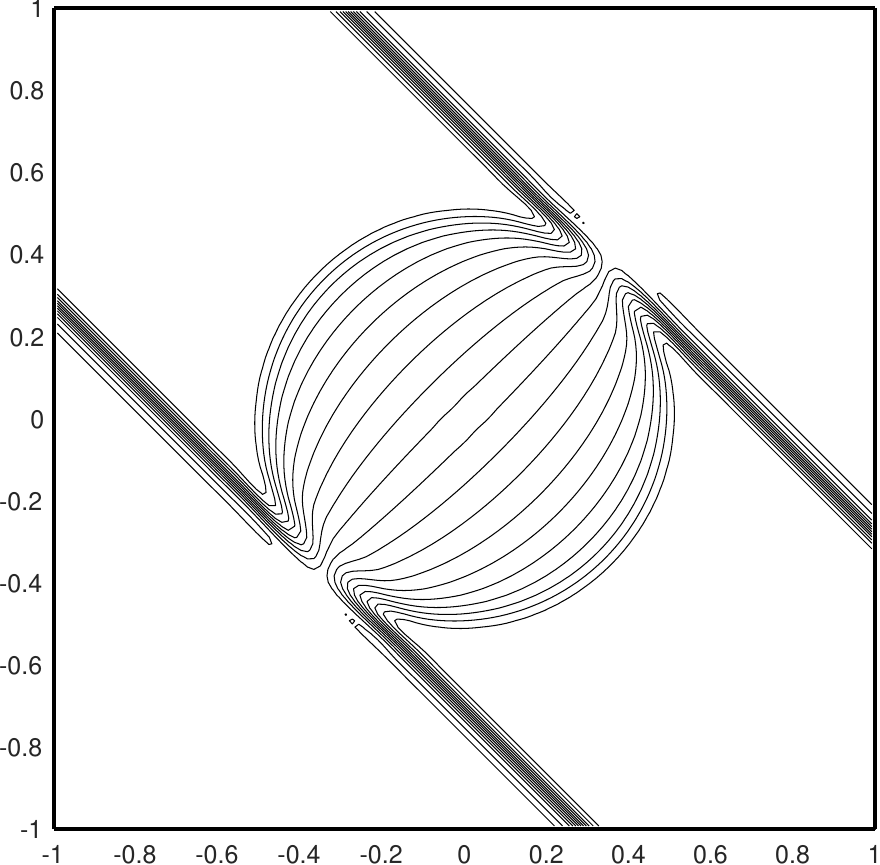} 
			\includegraphics[width=0.3\textwidth]{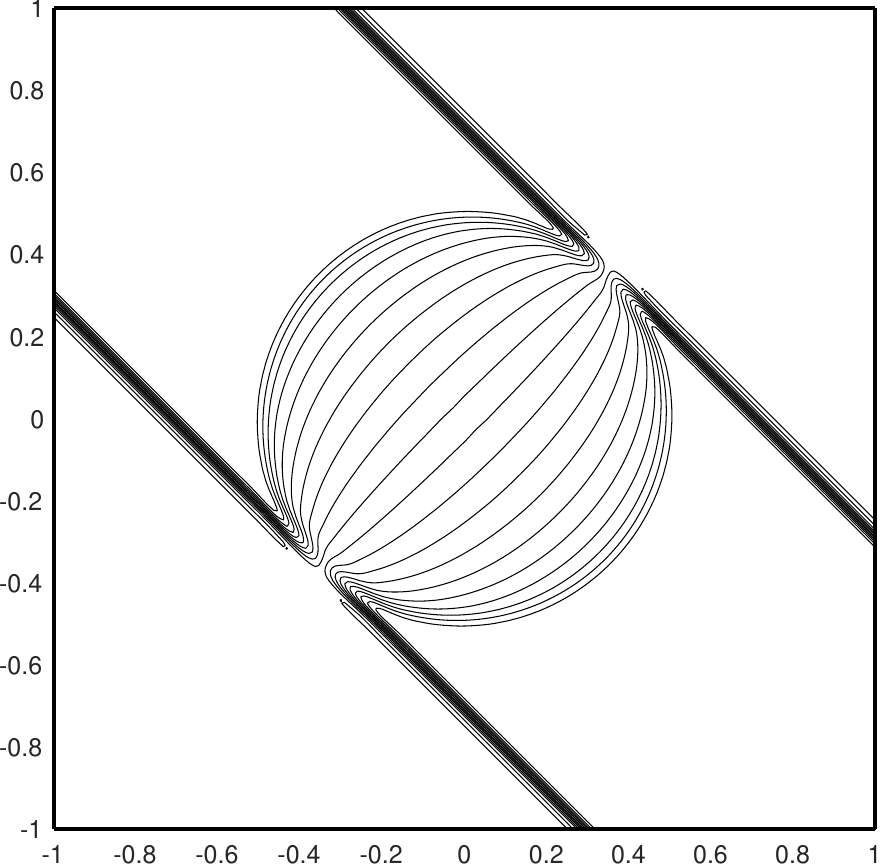}}
\caption{Approximation of discontinuous solution at $t=0.5$ on grids with $64\times 64$ (left), $128\times128$ (middle) and $256\times 256$ (right) cells using AFCW method with EG2 (first row), $\mathrm{EG}^{\text{quad}}$ (second row) and EG2$_{0.8,0.2}$ (third row) operators, respectively. All methods use $\mathrm{CFL} =0.7$.}\label{fig:wenoDiscontinuous}
\end{figure}

\section{Conclusion}

We have developed and analysed several new evolution operators for the fully discrete two-dimensional Active Flux method. The proposed operators are designed to retain the compactness and high accuracy of the AF method while significantly improving stability. We first constructed the  EG$^{\rm{quad}}$  operator that is exact for quadratic plane waves, leading to a third-order accurate method. 
This new formulation replaces the mantle and base integrals in the exact evolution with a third-order approximated evolution operator with a reduced computational complexity without sacrificing accuracy. We then introduced two families of modified operators, EG2$_\delta$ and EG2$_{\delta,\nu}$, obtained by systematically modifying the EG2 formulation through two-circle angular averaging and 
the Taylor-based approximations of the point values. A linear stability analysis has demonstrated a substantial enlargement of the admissible time step: the maximal stable CFL number increases from $0.279$ for EG2 to approximately $0.419$ for EG2$_\delta$ and up to $0.440$ for properly chosen EG2$_{\delta,\nu}$. Accuracy studies for smooth periodic test problems have confirmed that all new variants preserve third-order accuracy. For steady vortex configurations, we have observed that only the scheme based on the exact evolution operator maintains stationary states precisely. In contrast, the approximate operators exhibit minor but non-negligible deviations. We have also investigated 
the influence of different reconstructions and 
compared the  AF schemes using the compact AF reconstruction with the third-order CWENO reconstruction. The latter yields the AFCW methods that have a more compact stencil than standard semi-discrete WENO methods and achieve third-order accuracy with a CFL number up to 0.7.

Future work will focus on extending the present analysis to nonlinear systems, such as the Euler equations of gas dynamics. The AFCW methods offer an efficient way for limiting  and positivity preservation.

\section*{Acknowledgements}
This work was funded by DFG Projects 525800857 and 525853336 funded within Focused Programme SPP 2410 ``Hyperbolic Balance Laws: Complexity, Scales and Randomness". The work of S. Chu was supported in part by the DFG through HE5386/19-3, 27-1. M.L.-M. gratefully acknowledges the support of the Mainz Institute of Multiscale Modeling and Gutenberg Research College. 

\section*{Data availability statement}
The data that support the findings of this study and codes developed by the authors and used to obtain all of the presented numerical results are available from the corresponding author upon reasonable request.

\bibliography{references}
\bibliographystyle{plain}

\appendix
\section{Useful lemma}
\begin{lemma}\label{lem}
Let $f\in C^3(\Omega), r>0$, denote
$$
T_n = \frac{2\pi}{n}\sum\limits^{n-1}_{k=0}f\lb r,\frac{2\pi k}{n}\rb,\quad n\in \mathbb{N}.
$$
Then it holds that 
\begin{equation*}
\int^{2\pi}_0f\lb r,\theta\rb \dthe=T_n+\mathcal{O}\lb r^3\rb,\quad {\rm for} \,\, n=4,8.
\end{equation*}
\end{lemma}
\begin{proof}
This can be done by Taylor's expansion. Let us expand $f\lb x,y\rb , \lb x,y\rb =\lb r\cos\theta,r\sin\theta\rb $ at $\lb 0,0\rb $
\begin{align*}
f(x,y) =&f(0,0)+x\partial_xf(0,0)+y\partial_{y}f(0,0) \\
	   &+\frac12\left(x^2\partial_{xx}f(0,0)+2xy\partial_{xy}f(0,0)+y^2\partial_{yy}f(0,0)\right)+\mathcal{O}(r^3),
\end{align*}
which yields
\begin{align}\label{A1}
\int^{2\pi}_0f\lb x,y\rb \dthe&=2\pi f\lb 0,0\rb +\frac{r^2}{2}\partial_{xx}f\lb 0,0\rb \int^{2\pi}_0\cos^2\theta\dthe+\frac{r^2}{2}\partial_{yy}f\lb 0,0\rb \int^{2\pi}_0\sin^2\theta\dthe\\\nonumber
				        &=2\pi f\lb 0,0\rb +\frac{\pi r^2}{2}\lb \partial_{xx}+\partial_{yy}\rb f\lb 0,0\rb +\mathcal{O}\lb r^3\rb .
\end{align}
Let n=4, then $f\lb r,\frac{2\pi k}{n}\rb , k = 0,\dots,3$ at $\lb 0,0\rb $
\begin{align*}
f\lb r,0\rb              &=f\lb 0,0\rb +r\partial_{x}f\lb 0,0\rb +\frac12r^2\partial_{xx}f\lb 0,0\rb +\mathcal{O}\lb r^3\rb ,\\
f\lb r,\pi\rb            &=f\lb 0,0\rb -r\partial_{x}f\lb 0,0\rb +\frac12r^2\partial_{xx}f\lb 0,0\rb +\mathcal{O}\lb r^3\rb ,\\
f\lb r,\frac{\pi}{2}\rb  &=f\lb 0,0\rb +r\partial_{y}f\lb 0,0\rb +\frac12r^2\partial_{yy}f\lb 0,0\rb +\mathcal{O}\lb r^3\rb ,\\
f\lb r,\frac{3\pi}{2}\rb &=f\lb 0,0\rb -r\partial_{y}f\lb 0,0\rb +\frac12r^2\partial_{yy}f\lb 0,0\rb +\mathcal{O}\lb r^3\rb .
\end{align*}
Consequently, we have
\begin{align}\label{A2}
T_4&=\frac{\pi}{2}\sum\limits^{3}_{k=0}f\lb \frac{2\pi k}{4}\rb =\frac{\pi}{2}\lb f\lb r,0\rb +f\lb r,\pi\rb +f\lb r,\frac{\pi}{2}\rb +f\lb r,\frac{3\pi}{2}\rb \rb \\\nonumber
&=2\pi f\lb 0,0\rb +\frac{\pi}{2}r^2\lb \partial_{xx}+\partial_{yy}\rb f\lb 0,0\rb +\mathcal{O}\lb r^3\rb .
\end{align}
Comparing \eqref{A1} and \eqref{A2}, yields the desired results. Analogous computations can be applied for $n=8$.
\end{proof}

\section{Numerical Integration}
In this Section, we compare the performance of AF methods that employ approximate evolution formulas with exact integration to those that use numerical integration.  The main goal of using numerical integration is to reduce costs. Since multiple point evaluations can be used for $\delta = 1$, we use different deltas here than in Section \ref{sec:accuracy}. For this purpose, we consider the operators EG2$_{1.0}$, EG2$_{1.0,0.2}$, $\widehat{\mathrm{EG2}}_{1.0}$, and $\widehat{\mathrm{EG2}}_{1.0,0.2}$. To this end, we first perform convergence studies using Examples \ref{ex:1} and \ref{ex:2} (see Section \ref{app:accuracy}), and then compare the results for Example \ref{ex:3} (see Section \ref{app:vortex}) and Example \ref{ex:4} (see Section \ref{app:disc}). We use CFL = 0.39 for all calculations.

\subsection{Accuracy Results}\label{app:accuracy}
Tables \ref{Tab:convergenceProblem1_eghat_t01}–\ref{Tab:convergenceProblem1_eghat_t1} show the results for Examples \ref{ex:1} and \ref{ex:2} at $t = 0.1$ and $t = 1$. The methods that compute the evolution formulas using numerical integration are approximately as accurate as those employing exact integration at this point.
\begin{table}[!ht]
\centering
\caption{Errors measured in the $L_1$-norm and EOC for Example~\ref{ex:1} using EG2$_{1.0}$, $\widehat{\mathrm{EG2}}_{1.0}$, EG2$_{1.0,0.2}$ and $\widehat{\mathrm{EG2}}_{1.0,0.2}$ at $t = 0.1$.}
\label{Tab:convergenceProblem1_eghat_t01}
\vspace*{0.15cm}
\sisetup{scientific-notation = true,round-mode = places}
\renewcommand{\arraystretch}{0.9}
\setlength{\tabcolsep}{3pt}
\begin{tabular}{c*{4}{S[round-precision=2, table-format=1.2e-2]}*{4}{S[round-precision=4, table-format=1.4]}}
\toprule
Res. & \multicolumn{4}{c}{Error in $p$} & \multicolumn{4}{c}{EOC} \\
\cmidrule(lr){2-5} \cmidrule(lr){6-9}
	 & {EG2$_{1.0}$} & {$\widehat{\mathrm{EG2}}_{1.0}$} & {EG2$_{1.0,0.2}$} & {$\widehat{\mathrm{EG2}}_{1.0,0.2}$ }
	 & {EG2$_{1.0}$} & {$\widehat{\mathrm{EG2}}_{1.0}$} & {EG2$_{1.0,0.2}$} & {$\widehat{\mathrm{EG2}}_{1.0,0.2}$ } \\
\midrule
64   & \num{2.6595646687295479e-05} & \num{2.5338176315082262e-05} & \num{2.5553697028442009e-05} & \num{2.4434477732955670e-05}
	 & {---} & {---} & {---} & {---} \\
128  & \num{3.3171151146272079e-06} & \num{3.1613047791377671e-06}  & \num{3.1906165449364757e-06} & \num{3.0543530050883854e-06}
	 & 3.0032 & 3.0027  &  3.0016 & 3.0000 \\
256  & \num{4.1319794738418438e-07} & \num{3.9373830001912823e-07}  & \num{3.9762901382736824e-07} & \num{3.8057873413829300e-07}
	 & 3.0050 & 3.0052 &  3.0043 &  3.0046  \\
\bottomrule
\end{tabular}
\end{table}

\begin{table}[!ht]
\centering
\caption{Errors measured in the $L_1$-norm and EOC for Example~\ref{ex:1} using EG2$_{1.0}$, $\widehat{\mathrm{EG2}}_{1.0}$, EG2$_{1.0,0.2}$ and $\widehat{\mathrm{EG2}}_{1.0,0.2}$ at $t = 1$.}
\label{Tab:convergenceProblem1_eghat_t1}
\vspace*{0.15cm}
\sisetup{scientific-notation = true,round-mode = places}
\renewcommand{\arraystretch}{0.9}
\setlength{\tabcolsep}{3pt}
\begin{tabular}{c*{4}{S[round-precision=2, table-format=1.2e-2]}*{4}{S[round-precision=4, table-format=1.4]}}
\toprule
Res. & \multicolumn{4}{c}{Error in $p$} & \multicolumn{4}{c}{EOC} \\
\cmidrule(lr){2-5} \cmidrule(lr){6-9}
	 & {EG2$_{1.0}$} & {$\widehat{\mathrm{EG2}}_{1.0}$} & {EG2$_{1.0,0.2}$} & {$\widehat{\mathrm{EG2}}_{1.0,0.2}$ }
	 & {EG2$_{1.0}$} & {$\widehat{\mathrm{EG2}}_{1.0}$} & {EG2$_{1.0,0.2}$} & {$\widehat{\mathrm{EG2}}_{1.0,0.2}$ } \\
\midrule
64   & \num{3.2187267146899527e-04} & \num{3.2256840176350096e-04} & \num{3.1644285862701595e-04} & \num{3.1803944365606934e-04}
	 & {---} & {---} & {---} & {---} \\
128  & \num{4.0540010709008755e-05} & \num{4.0650678392636756e-05} & \num{3.9835835877694806e-05} & \num{4.0055852010725435e-05}
	 & 2.9891 & 2.9883  & 2.9898  & 2.9891 \\
256  & \num{5.0698713477737136e-06} & \num{5.0844860373556694e-06} & \num{4.9821559898756295e-06} & \num{5.0101706139213363e-06}
	 & 2.9993 & 2.9991 & 2.9992 & 2.9991    \\
\bottomrule
\end{tabular}
\end{table}

\begin{table}[!ht]
\centering
\caption{Errors measured in the $L_1$-norm and EOC for Example \ref{ex:2} using EG2$_{1.0}$, $\widehat{\mathrm{EG2}}_{1.0}$, EG2$_{1.0,0.2}$ and $\widehat{\mathrm{EG2}}_{1.0,0.2}$ at $t = 0.1$.}
\label{Tab:convergenceProblem2_eghat_t01}
\vspace*{0.15cm}
\sisetup{scientific-notation = true,round-mode = places}
\renewcommand{\arraystretch}{0.9}
\setlength{\tabcolsep}{3pt}
\begin{tabular}{c*{4}{S[round-precision=2, table-format=1.2e-2]}*{4}{S[round-precision=4, table-format=1.4]}}
\toprule
Res. & \multicolumn{4}{c}{Error in $u,v$} & \multicolumn{4}{c}{EOC} \\
\cmidrule(lr){2-5} \cmidrule(lr){6-9}
	 & {EG2$_{1.0}$} & {$\widehat{\mathrm{EG2}}_{1.0}$} & {EG2$_{1.0,0.2}$} & {$\widehat{\mathrm{EG2}}_{1.0,0.2}$ }
	 & {EG2$_{1.0}$} & {$\widehat{\mathrm{EG2}}_{1.0}$} & {EG2$_{1.0,0.2}$} & {$\widehat{\mathrm{EG2}}_{1.0,0.2}$ } \\
\midrule
64   & \num{2.0878787233222443e-05} & \num{2.1947776280825189e-05} & \num{2.0862726511601006e-05} & \num{2.1931796316003220e-05}
	 & {---} & {---} & {---} & {---} \\
128  & \num{2.5954606789910305e-06} & \num{2.7318019500385547e-06 } & \num{2.5971047900655746e-06} & \num{2.7328907645635087e-06}
	 &3.0080  & 3.0061  &  3.0060  & 3.0045 \\
256  & \num{3.2301144135668633e-07} & \num{3.4017867863900450e-07} & \num{3.2352632177264750e-07} & \num{3.4059323905861161e-07}
	 & 3.0063 & 3.0055 &  3.0049  &  3.0043 \\
\bottomrule
\end{tabular}
\end{table}

\begin{table}[!ht]
\centering
\caption{Errors measured in the $L_1$-norm and EOC for Example~\ref{ex:2} using EG2$_{1.0}$, $\widehat{\mathrm{EG2}}_{1.0}$, EG2$_{1.0,0.2}$ and $\widehat{\mathrm{EG2}}_{1.0,0.2}$ at $t = 1$.}
\label{Tab:convergenceProblem2_eghat_t1}
\vspace*{0.15cm}
\sisetup{scientific-notation = true,round-mode = places}
\renewcommand{\arraystretch}{0.9}
\setlength{\tabcolsep}{3pt}
\begin{tabular}{c*{4}{S[round-precision=2, table-format=1.2e-2]}*{4}{S[round-precision=4, table-format=1.4]}}
\toprule
Res. & \multicolumn{4}{c}{Error in $u,v$} & \multicolumn{4}{c}{EOC} \\
\cmidrule(lr){2-5} \cmidrule(lr){6-9}
	 & {EG2$_{1.0}$} & {$\widehat{\mathrm{EG2}}_{1.0}$} & {EG2$_{1.0,0.2}$} & {$\widehat{\mathrm{EG2}}_{1.0,0.2}$ }
	 & {EG2$_{1.0}$} & {$\widehat{\mathrm{EG2}}_{1.0}$} & {EG2$_{1.0,0.2}$} & {$\widehat{\mathrm{EG2}}_{1.0,0.2}$ } \\
\midrule
64   & \num{2.5367748805894721e-04} & \num{2.5409223749359318e-04} & \num{2.4932560517675490e-04} & \num{2.5046527476316942e-04}
	 & {---} & {---} & {---} & {---} \\
128  & \num{3.1872517874105865e-05 } & \num{3.1949024659726255e-05} & \num{3.1313024720884184e-05} & \num{3.1475675994757708e-05}
	 & 2.9926 & 2.9915  & 2.9932  & 2.9923 \\
256  & \num{3.9830684227024555e-06 } & \num{3.9940750604201853e-06} & \num{3.9138665899952770e-06} & \num{3.9353721351101768e-06}
	 &3.0004  & 2.9998 & 3.0001 &  2.9997  \\
\bottomrule
\end{tabular}
\end{table}

\subsection{Approximation of the Stationary Vortex}\label{app:vortex}
Figure \ref{fig:appendix_Vortex} shows the results for Example \ref{ex:3} at $t = 100$. The performance of the methods that compute the evolution formulas using numerical integration is approximately the same as for methods employing exact integration.
\begin{figure}[htb]
\centerline{\includegraphics[width=0.25\textwidth]{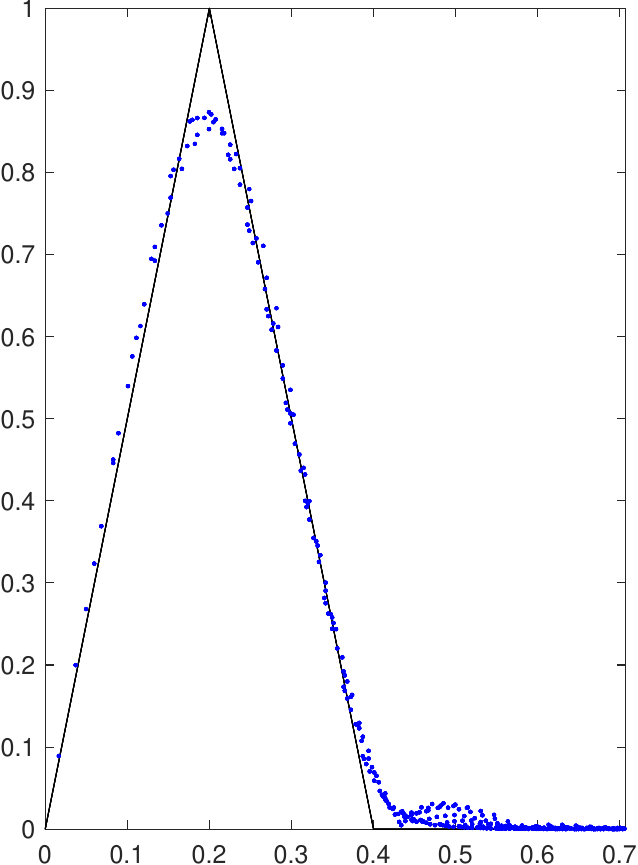}
            \includegraphics[width=0.25\textwidth]{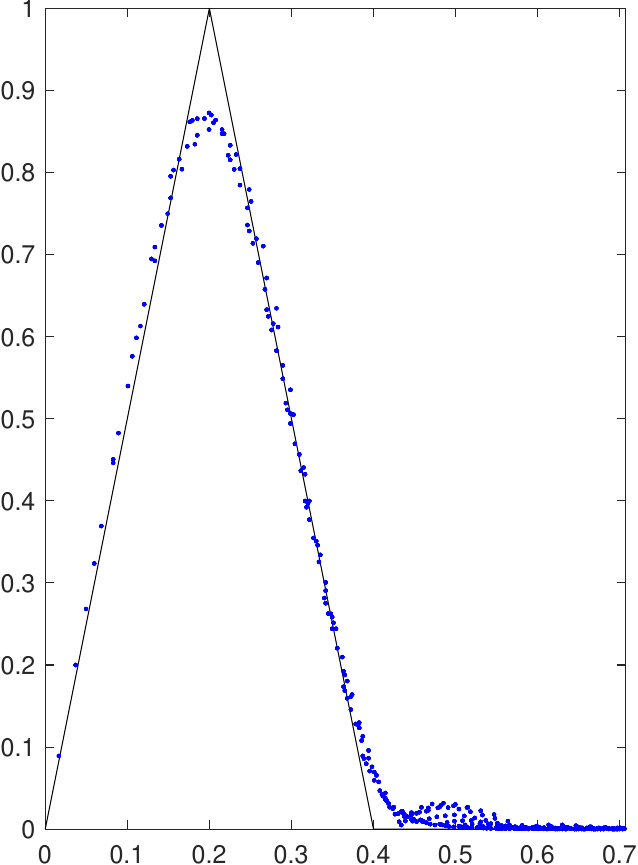}\hfil  
            \includegraphics[width=0.25\textwidth]{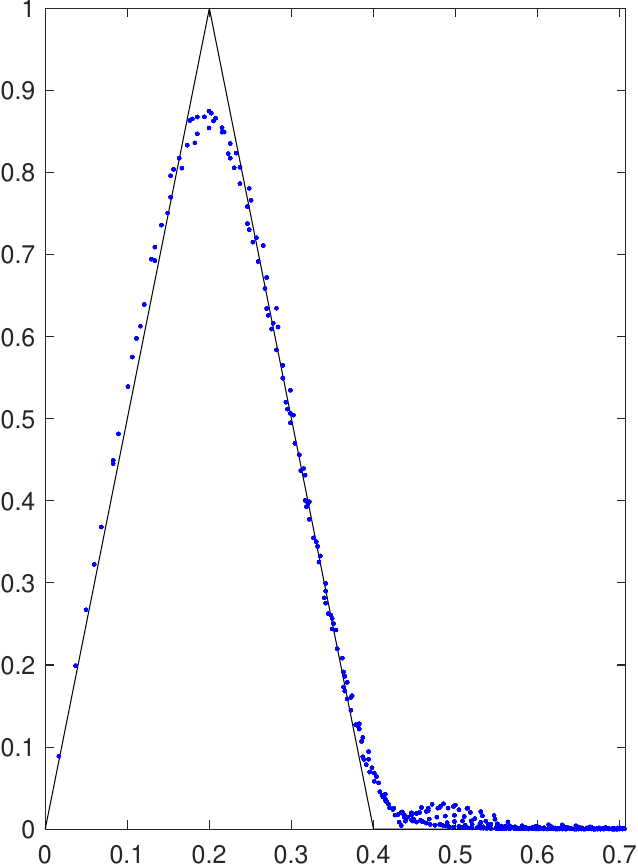}\hfil 
            \includegraphics[width=0.25\textwidth]{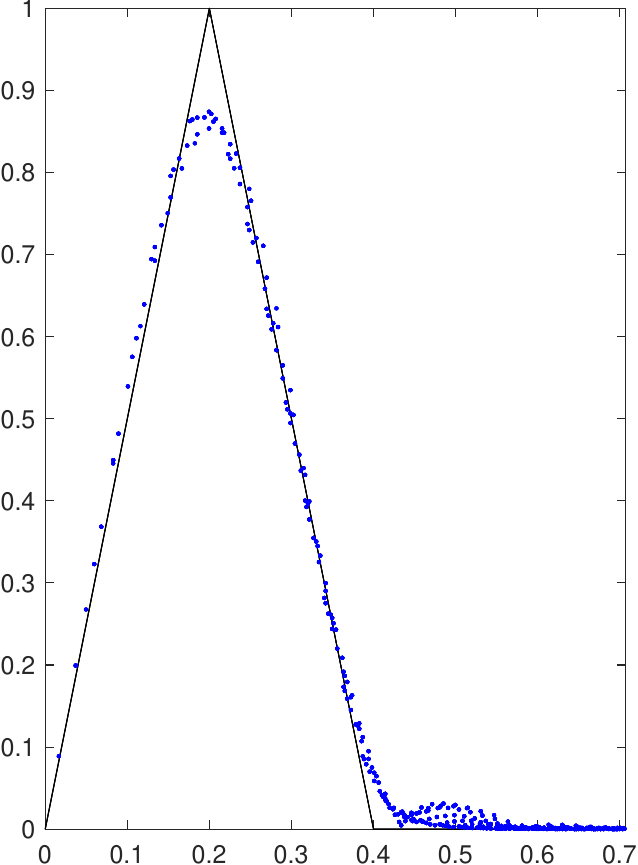}}

\centerline{\includegraphics[width=0.25\textwidth]   {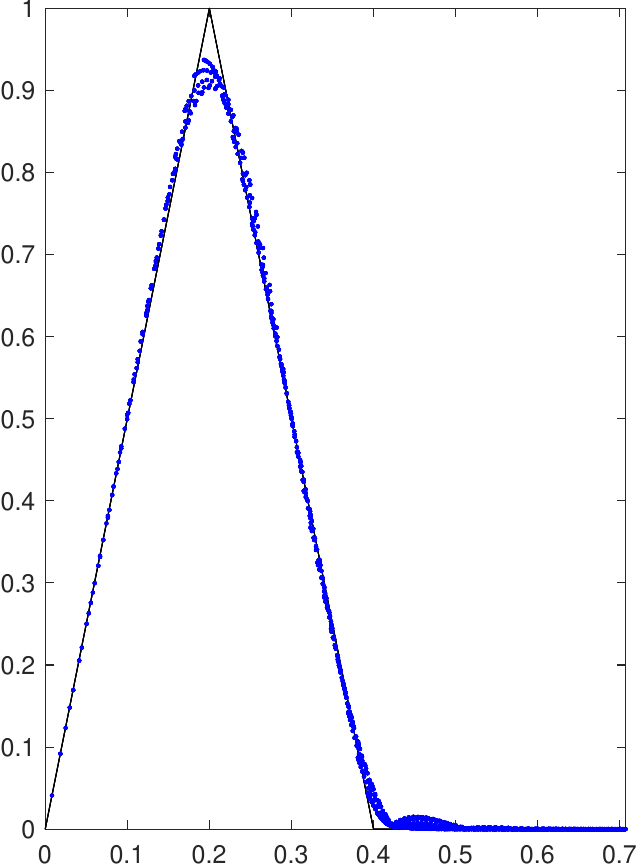}
            \includegraphics[width=0.25\textwidth]{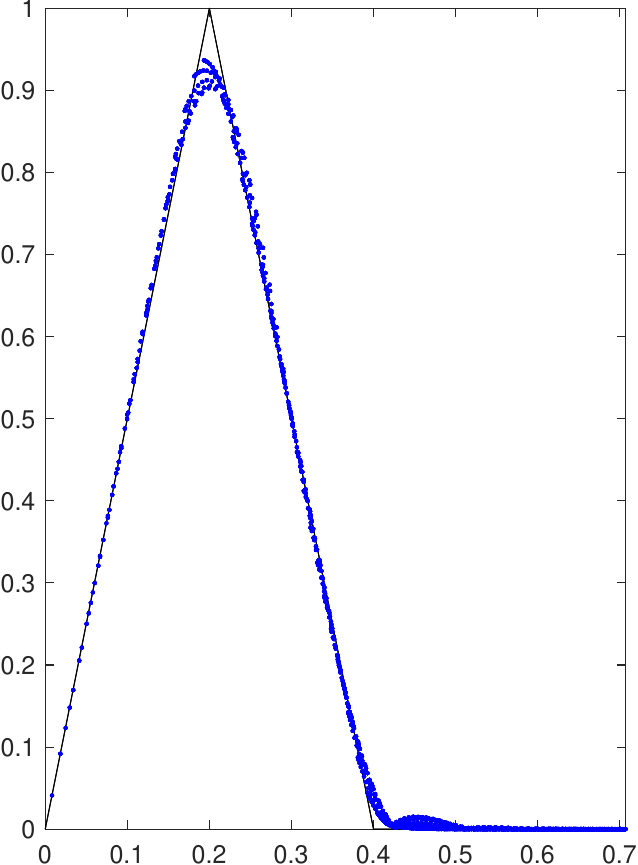}\hfil  
            \includegraphics[width=0.25\textwidth]{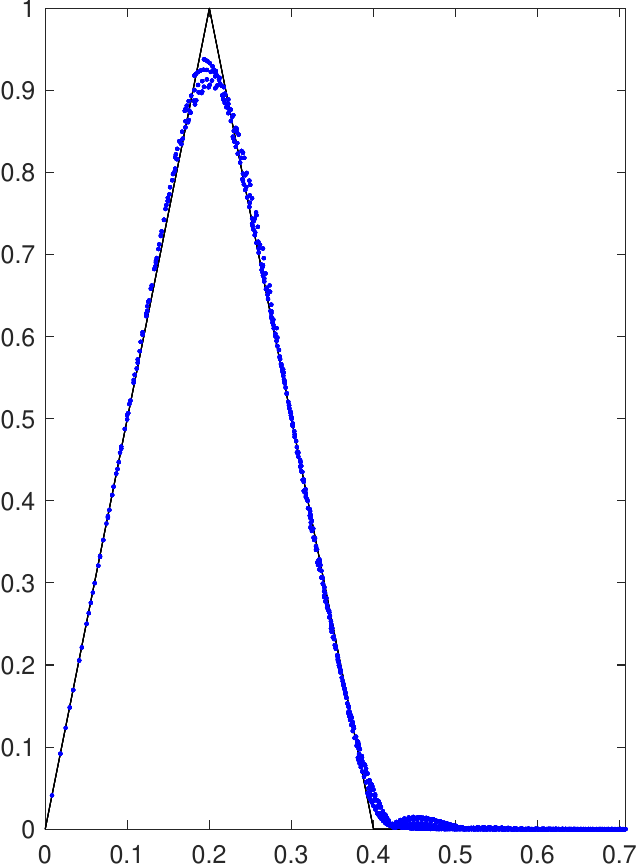}\hfil 
            \includegraphics[width=0.25\textwidth]{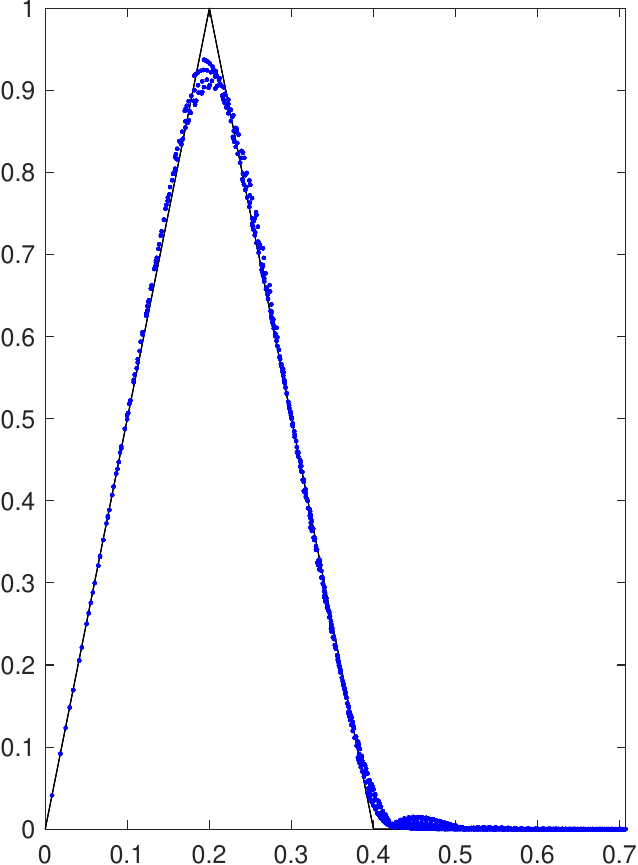}}
            
\caption{ \label{fig:appendix_Vortex}Approximation of the stationary vortex using a grid with $64\times64$ (top) and $128\times128$ (bottom) cells at $t=100$ with an AF method using EG2$_{1.0}$ (left), $\widehat{\mathrm{EG2}}_{1.0}$ (center left), EG2$_{1.0,0.2}$ (center right) and $\widehat{\mathrm{EG2}}_{1.0,0.2}$ (right).} 
\end{figure}

\subsection{Approximation of Discontinuous Solution Structure}\label{app:disc}
Figure \ref{fig:appendix_disc} shows the results for Example \ref{ex:4} at $t = 0.5$. The performance of the methods that compute the evolution formulas using numerical integration is approximately the same as that of those employing exact integration.
\begin{figure}[!htbp]
\centerline{\includegraphics[width=0.3\textwidth]{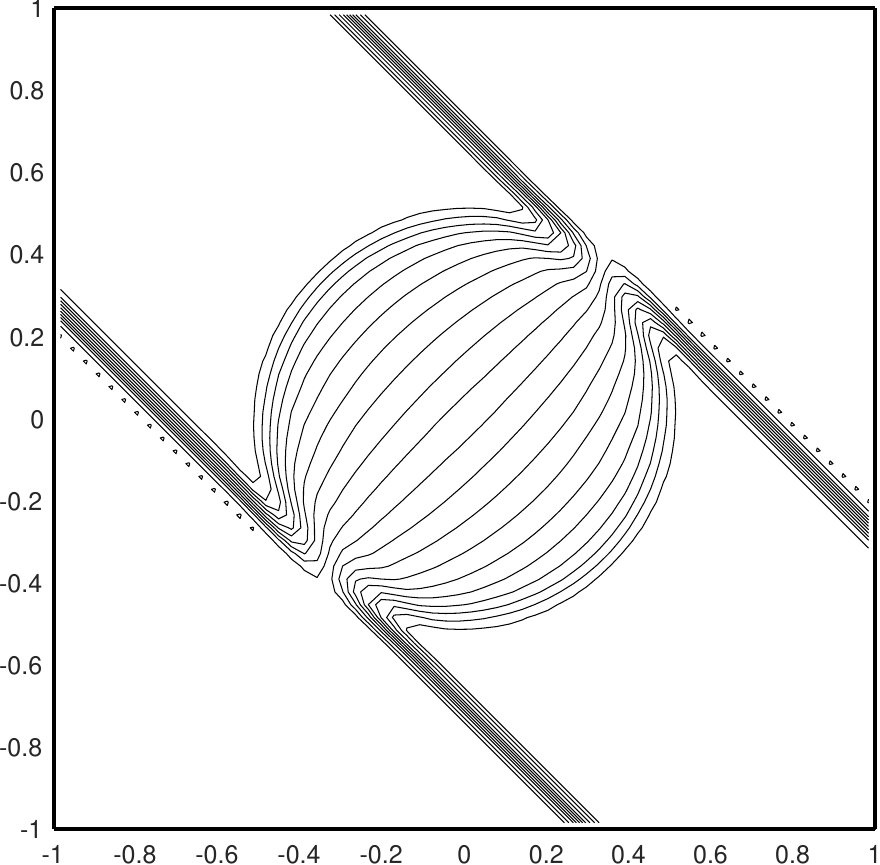}
		    \includegraphics[width=0.3\textwidth]{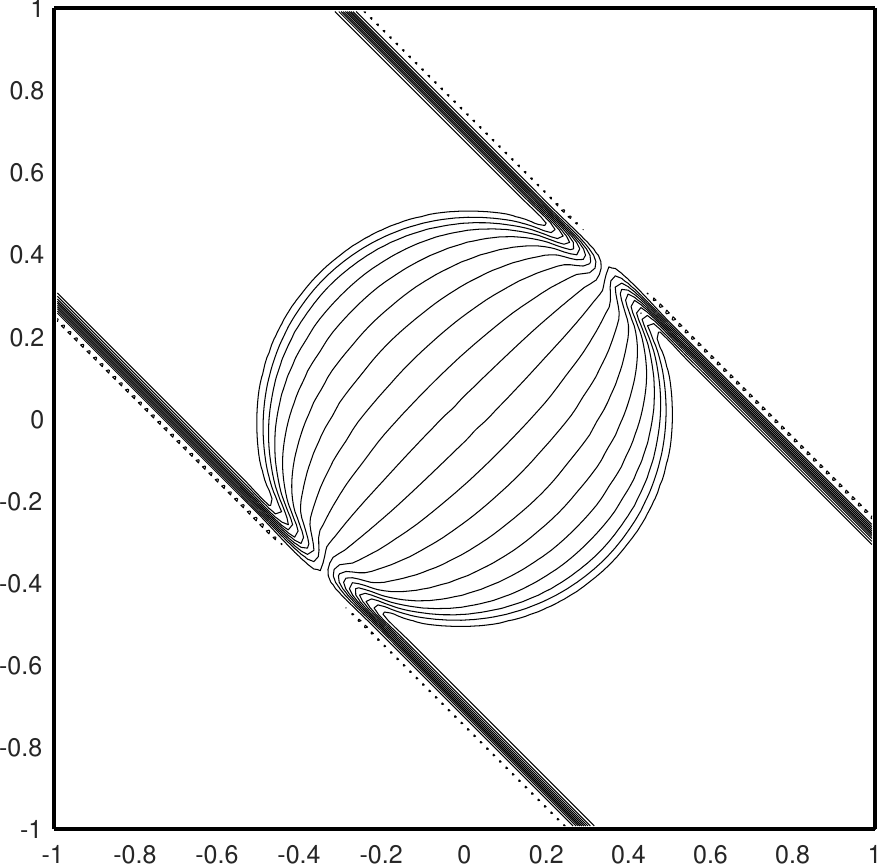}  
		    \includegraphics[width=0.3\textwidth]{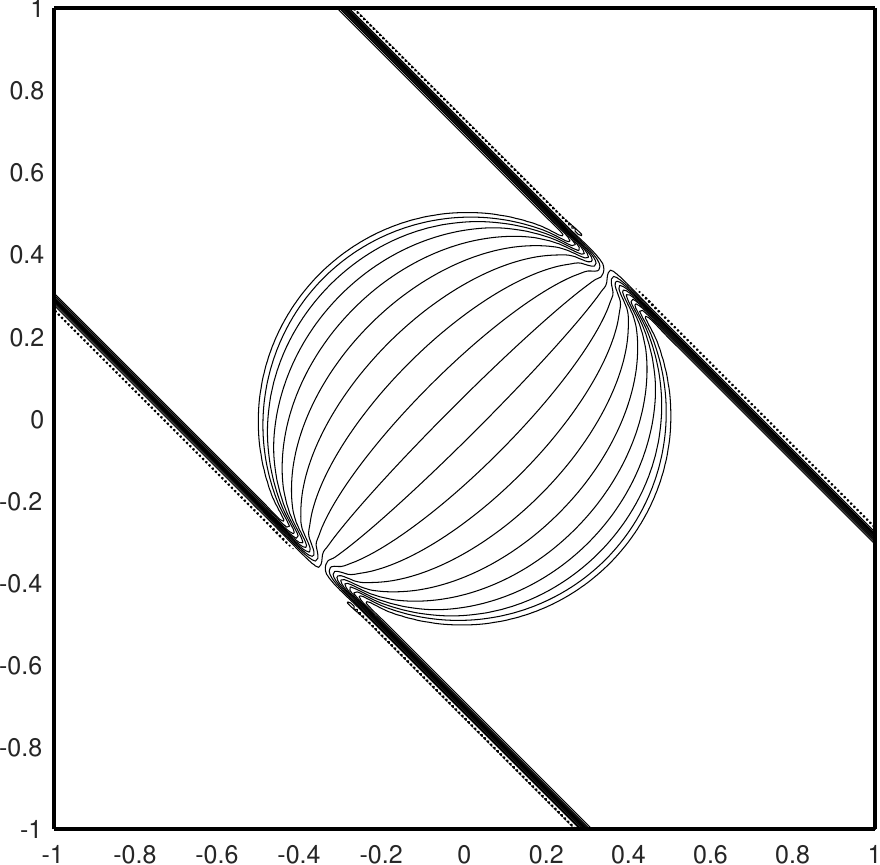}}
\centerline{\includegraphics[width=0.3\textwidth]{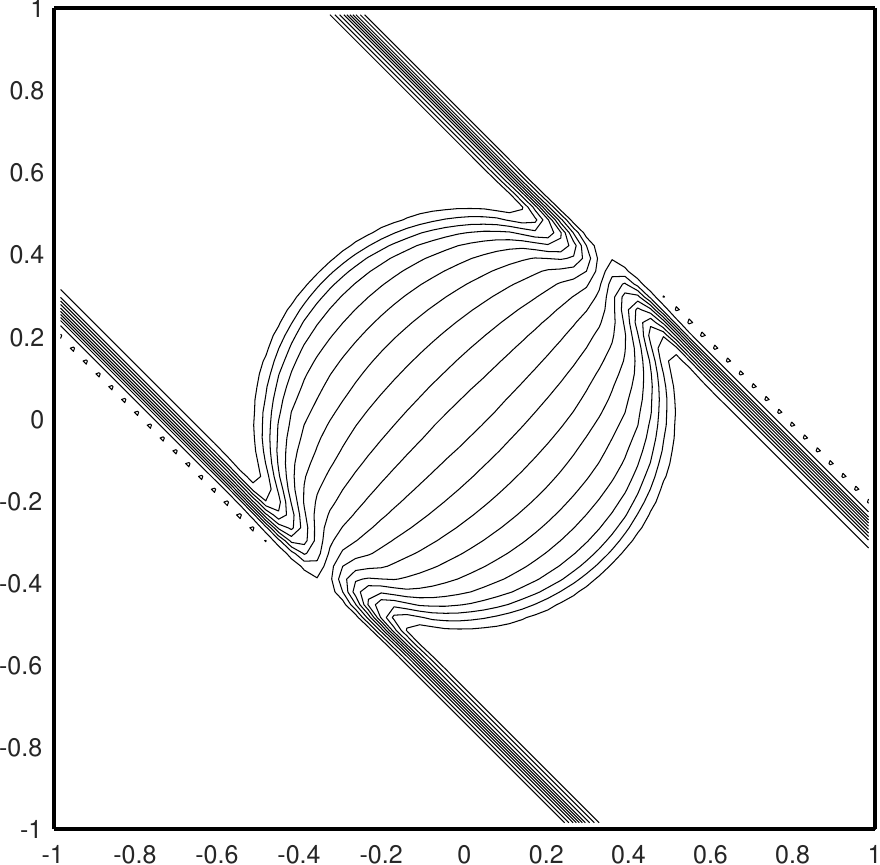}
		    \includegraphics[width=0.3\textwidth]{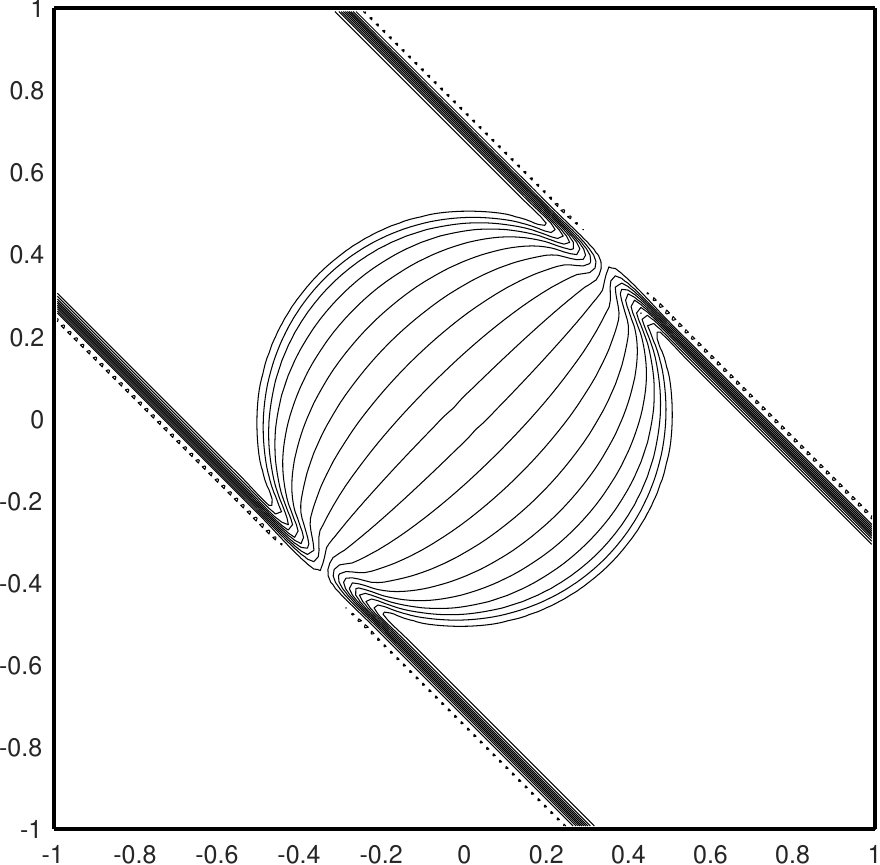}  
		    \includegraphics[width=0.3\textwidth]{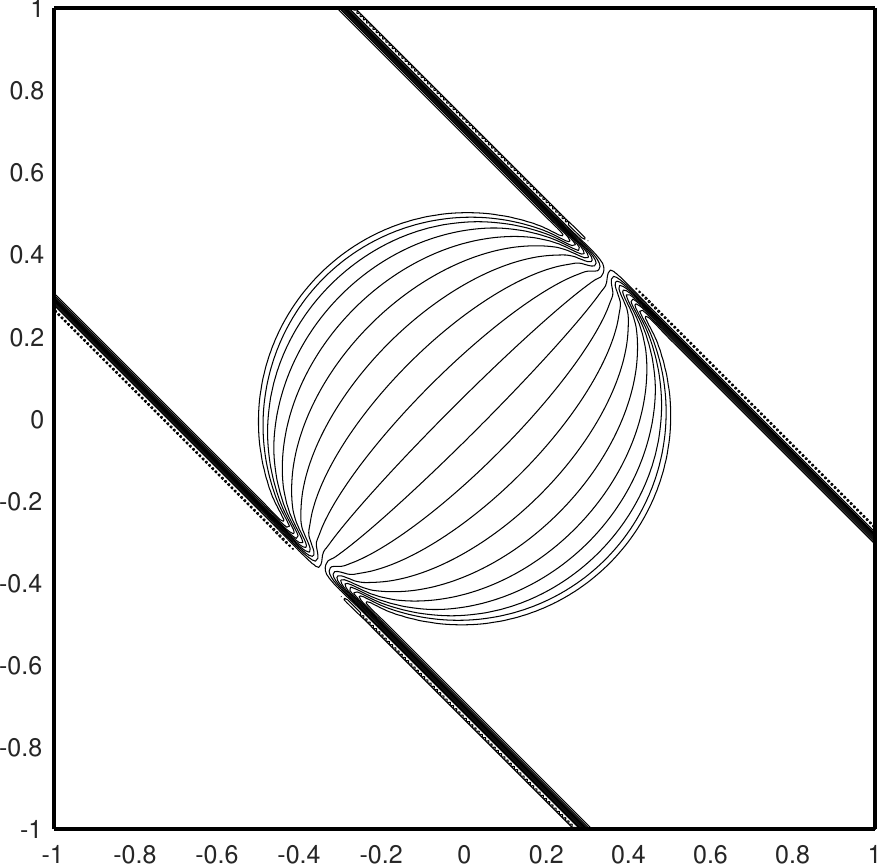}}
\centerline{\includegraphics[width=0.3\textwidth]{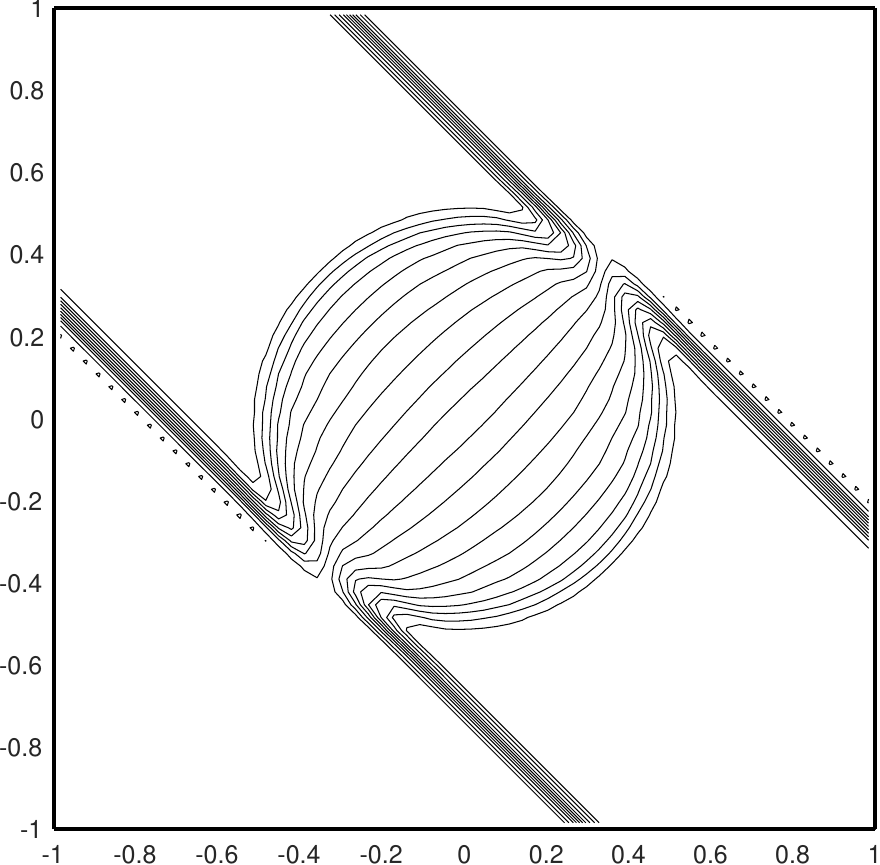}
		    \includegraphics[width=0.3\textwidth]{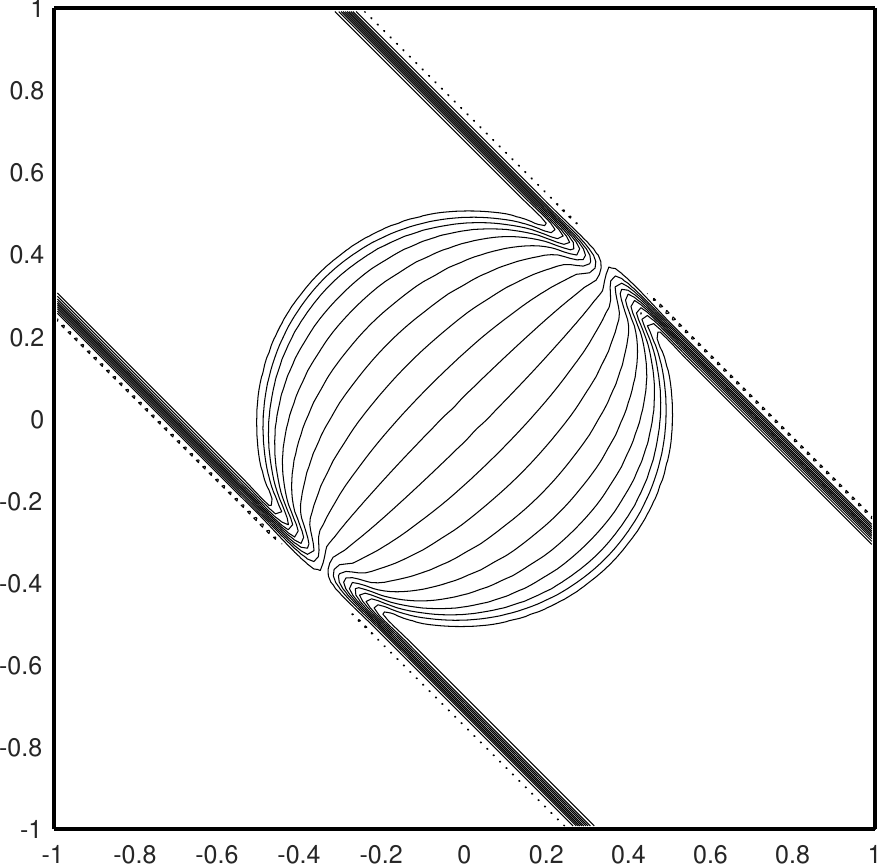}  
		    \includegraphics[width=0.3\textwidth]{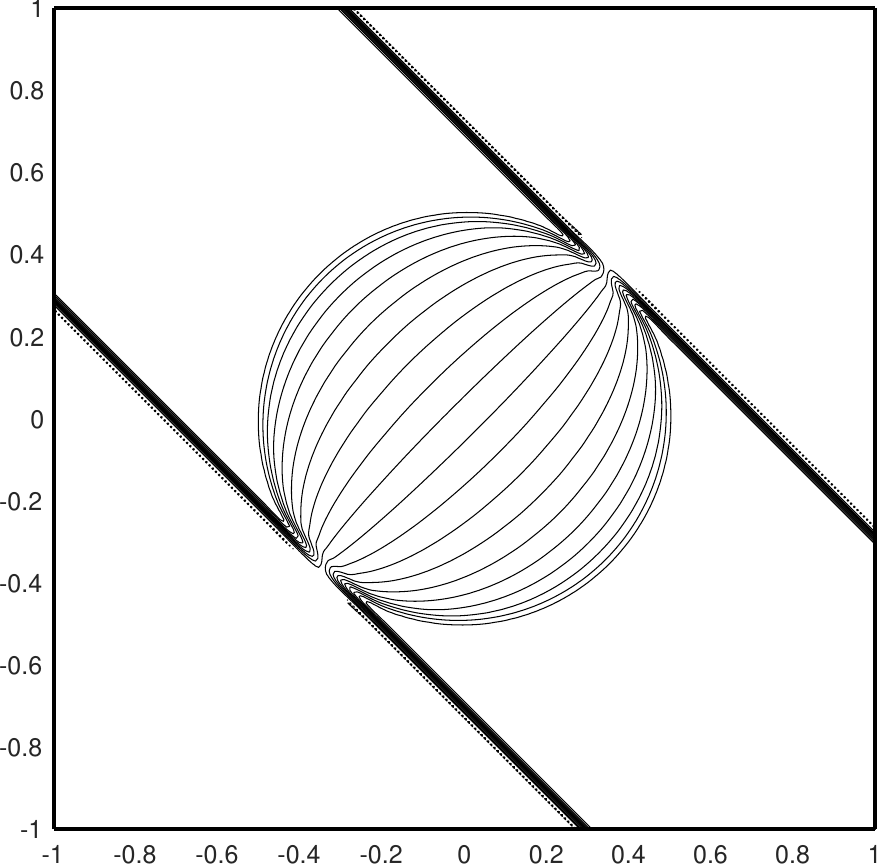}}
\centerline{\includegraphics[width=0.3\textwidth]{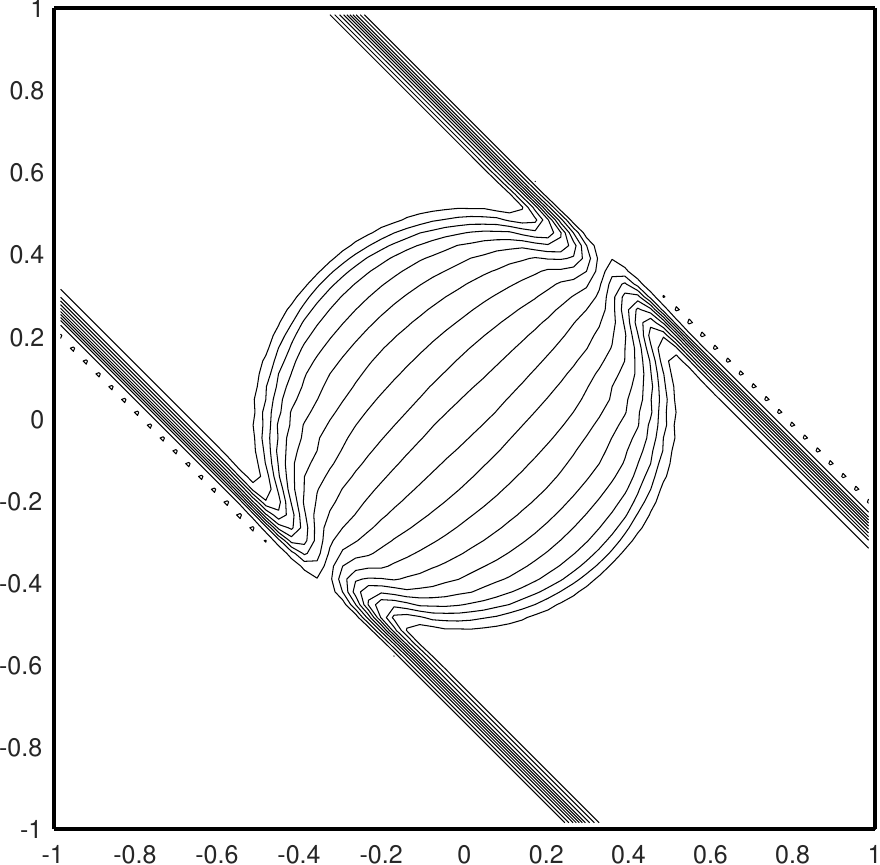}
		    \includegraphics[width=0.3\textwidth]{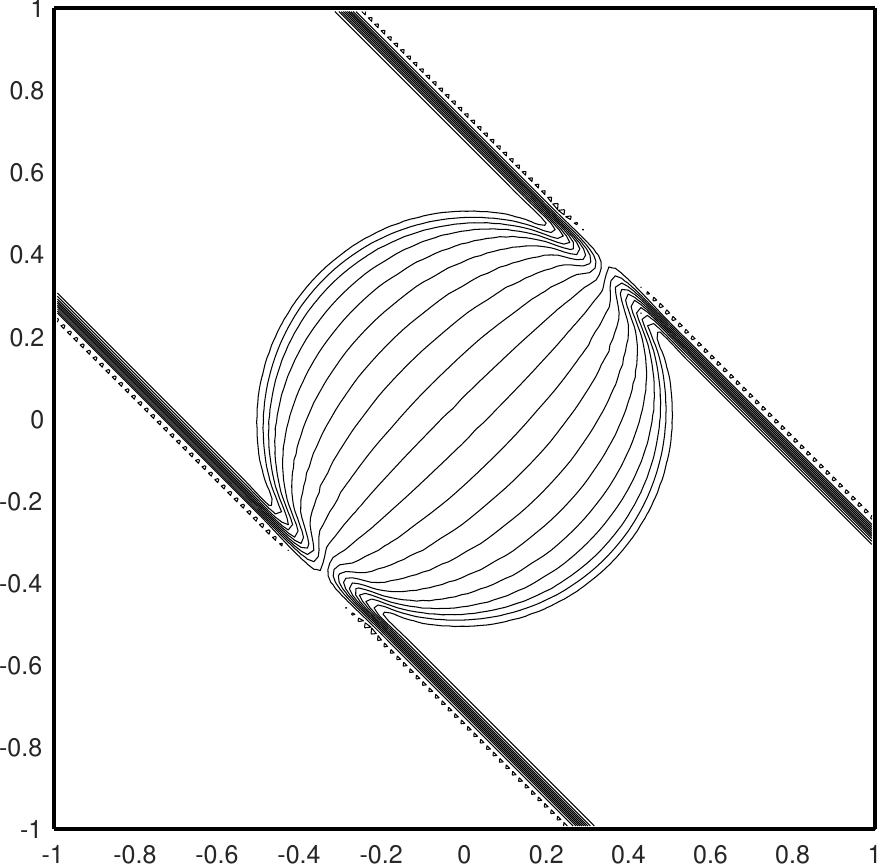}  
		    \includegraphics[width=0.3\textwidth]{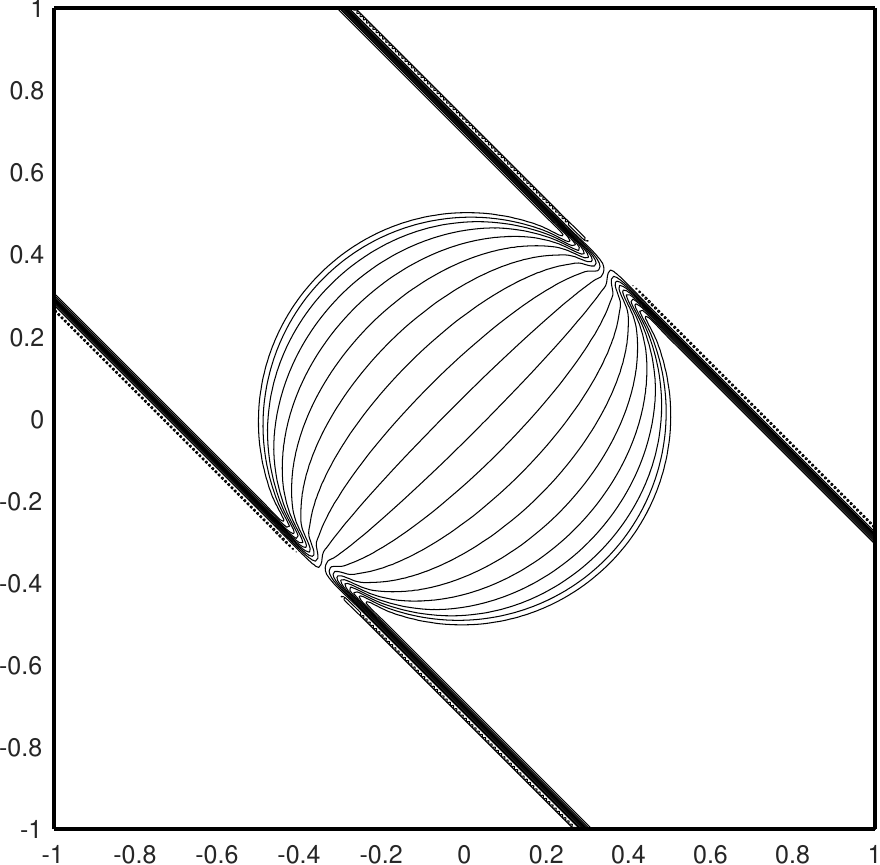}}

\caption{Approximation of discontinuous solution at $t=0.5$ on grids with $64\times 64$ (left), $128\times128$ (middle) and $256\times 256$ (right) cells using the AF method with EG2$_{1.0}$ (first row), $\widehat{\mathrm{EG2}}_{1.0}$ (second row), EG2$_{1.0,0.2}$ (third row) and $\widehat{\mathrm{EG2}}_{1.0,0.2}$  (fourth row).\label{fig:appendix_disc} } 
\end{figure}

\end{document}